\documentclass[12pt,a4paper,reqno]{amsart}

\usepackage{thmtools}
\usepackage{thm-restate}

\usepackage{hyperref}

\usepackage{cleveref}

\usepackage{calc,verbatim,enumitem,tikz,url,hyperref,mathrsfs,cite}

\usepackage{float}

\usepackage{pgfplots}

\pgfplotsset{
     standard/.style={
        axis x line=middle,
        axis y line=middle,
        every axis x label/.style={at={(current axis.right of origin)},anchor=north west},
        every axis y label/.style={at={(current axis.above origin)},anchor=north east}
    }
}

\pgfplotsset{width=13cm,compat=1.13}

\usepgfplotslibrary{fillbetween}

\usepackage{amsmath, amssymb, latexsym, amsthm, amsbsy}

\usepackage{graphics}

\usepackage{graphicx,relsize}

\usepackage[mathscr]{eucal}

\usepackage{bbm}

\def\NORMAL{\scalebox{0.8}{\textrm{NORMAL}}}
\def\STAR{\scalebox{0.8}{\textrm{STAR}}}
\def\HUB{\scalebox{0.8}{\textrm{HUB}}}
\def\CLIQUE{\scalebox{0.8}{\textrm{CLIQUE}}}
\def\DEV{\scalebox{0.9}{\textrm{DEV}}}

\usepackage[a4paper,top=3cm,bottom=3cm,inner=3cm,outer=3cm,includehead,includefoot]{geometry}

\newcommand{\ssq}{\subseteq} 
\newcommand{\subgp}[1]{\langle{{\hash}1}\rangle}

\def\F{\mathcal{F}}

\def\LL{\mathcal{L}}

\def\S{\mathcal{S}}
\def\T{\mathcal{T}}

\def\X{\mathbb{X}}

\def\ourmax{\mathrm{max}}

\newcommand{\Ex}[1]{\mathbb{E}\left[#1\right]}

\newcommand{\pr}[1]{\mathbb{P}\left(#1\right)}
\newcommand{\expb}[1]{\exp\left(#1\right)}

\newcommand{\fl}[1]{\ensuremath{\lfloor #1 \rfloor}}
\newcommand{\norm}[1]{\left\lVert#1\right\rVert}

\newcommand{\eq}[1]{\begin{equation}\label{eq:#1}}
\newcommand{\eqe}{\end{equation}}
\newcommand{\eqr}[1]{\eqref{eq:#1}}

\numberwithin{equation}{section}

\def\Var{\textup{Var}}

\def\RR{\mathbb{R}}
\def\ZZ{\mathbb{Z}}

\def\rdpn{r(\delta, p,n)}
\def\rdn{r(\delta_n,p,n)}

\def\owedge{\mathsmaller{\bigwedge}}

\newtheorem{theorem}{Theorem}[section]
\newtheorem{lem}[theorem]{Lemma}

\newtheorem{cor}[theorem]{Corollary}
\newtheorem{question}{Question}

\newtheorem{prop}[theorem]{Proposition}
\theoremstyle{definition}\newtheorem{definition}[theorem]{Definition}
\theoremstyle{definition}\newtheorem{remark}[theorem]{Remark}

\def\le{\leqslant}
\def\ge{\geqslant}

\renewcommand{\leq}{\leqslant}
\renewcommand{\geq}{\geqslant}

\usepackage{soul}

\definecolor{sgreen}{rgb}{0.3, 0.9, 0.3} 
\definecolor{dgreen}{rgb}{0.2, 0.8, 0.2} 
\definecolor{dred}{rgb}{1, 0.2, 0.2}

\definecolor{lblue}{rgb}{0.6, 0.6, 1} 
\definecolor{lgr}{rgb}{0.8, 0.8, 0.8} 
\definecolor{purp}{rgb}{0.9, 0, 0.9} 
\definecolor{dy}{rgb}{0.9, 0.8, 0.2}

\begin{document}

\title[Moderate deviations of triangle counts]{Moderate deviations of triangle counts in sparse Erd\H os-R\' enyi random graphs $G(n,m)$ and $G(n,p)$}

\author{Jos\'e D. Alvarado, Leonardo Gon\c{c}alves de Oliveira and Simon Griffiths}

\thanks{Jos\'e received postdoctoral grants from the Brazilian funding agencies CNPq (Proc. 153903/2018-0), FAPERJ (Proc. E-26/202.011/2020) and FAPESP (Proc. 2020/10796-0) under the supervision of Simon at PUC-Rio and Guilherme Mota at USP, Leonardo was supported by a PhD grant from the CAPES funding agency and Simon received research support from CNPq (Proc. 307521/2019-2) and FAPERJ (Proc. E-26/202.713/2018 and Proc. E-26/201.194/2022).}

\address{Faculty of Mathematics and Physics, University of Ljubljana, Jadranska ulica 19, 1000 Ljubljana, Slovenia}\email{jose.alvarado@fmf.uni-lj.si}

\address{Departamento de Matem\'atica, PUC-Rio, Rua Marqu\^{e}s de S\~{a}o Vicente 225, G\'avea, 22451-900 Rio de Janeiro, Brazil}
%\email{leonardogonoli@mat.puc-rio.br}

\address{Departamento de Matem\'atica, PUC-Rio, Rua Marqu\^{e}s de S\~{a}o Vicente 225, G\'avea, 22451-900 Rio de Janeiro, Brazil}\email{simon@puc-rio.br}

\begin{abstract}
We consider the question of determining the probability of triangle count deviations in the Erd\H os-R\' enyi random graphs $G(n,m)$ and $G(n,p)$ with densities larger than $n^{-1/2}(\log{n})^{1/2}$.  In particular, we determine the log probability $\log\pr{N_{\triangle}(G)\, >\, (1+\delta)p^3n^3}$ up to a constant factor across essentially the entire range of possible deviations, in both the $G(n,m)$ and $G(n,p)$ model.  For the $G(n,p)$ model we also prove a stronger result, up to a $(1+o(1))$ factor, in the non-localised regime.  We also obtain some results for the lower tail and for counts of cherries (paths of length $2$).
\end{abstract}

\maketitle

\section{Introduction}

The problem of triangle count deviations in the random graph $G(n,p)$ has been extensively studied in recent decades.  We give an overview of the literature in Section~\ref{sec:mainp}.  The majority of results focus on a particular range of deviations, such as large deviations (of the order of the mean), whereas we shall prove results which hold across essentially the whole range of possible deviations.

%most articles on this subject have restricted themselves either to large deviations (of the order of the mean) or relatively small deviations (of the order of the standard deviation), we consider the whole range of deviations in between.  Other articles, including one by the last author together with Goldschmidt and Scott~\cite{GGS}, also considered moderate deviations.  We extend on those results to cover all densities down to $n^{-1/2}\sqrt{\log{n}}$.

One may also consider these problems in the $G(n,m)$ random graph model, in which the random graph $G\sim G(n,m)$ is selected uniformly from all graphs with $n$ vertices and $m$ edges.  In conceptual terms, the advantage of working in the $G(n,m)$ model is that we have removed one potential cause of deviations -- the variation in the number of edges.  

There are also practical reasons to prefer the $G(n,m)$ model -- we may generate $G(n,m)$ by revealing one edge at a time, for a fixed number of steps $m$, and this is particularly well suited to martingale techniques.  Furthermore, the results proved in $G(n,m)$ may be easily transferred to results in $G(n,p)$.

Let us begin then with a discussion of triangle count deviations in $G(n,m)$.  After we have stated our results in this setting we return in Section~\ref{sec:mainp} to the $G(n,p)$ setting.

\subsection{Our results for triangle counts in $G(n,m)$}\label{sec:mainm}

We use the notation $N:=\binom{n}{2}$ and set $t:= m/N$.  Note that $t$ is the proportion of edges which are present in $G_m\sim G(n,m)$.  We use the falling factorial notation $(n)_k=n(n-1)\dots (n-k+1)$.  

We remark that our approach is similar to that Goldschmidt, Scott and the third author~\cite{GGS}.  However, we stress that we introduce new ideas and techniques which are needed to extend the results proved in~\cite{GGS} to a much wider range of densities.

We let $N_{\triangle}(G_m)$ denote the number of isomorphic copies of the triangle in $G_m$.  For example $N_{\triangle}(K_4)=24$.  We set
\[
L_{\triangle}(m)\, :=\, \Ex{N_{\triangle}(G_m)}\, =\, \frac{(m)_3(n)_3}{(N)_3}\, ,
\]
and so 
\[
D_{\triangle}(G_m) \, :=\, N_{\triangle}(G_m)\, -\, L_{\triangle}(m)
\]
is the deviation of this triangle count from its mean.

Our main result in the $G(n,m)$ model will be to determine the rate associated with deviations up to a constant multiplicative factor across essentially the whole range of deviations, for all densities $t\gg n^{-1/2}(\log{n})^{1/2}$.  That is, given a sequence $a=a_n$ with $t^{3/2}n^{3/2}\sqrt{\log{n}}\, \ll \,a_n\, \ll\, t^{3/2}n^3$, we find a sequence $b= b_n$ such that $\pr{D_{\triangle}(G_m)>a}=\exp(-\Theta(b))$.  

In fact, we find that it is more intuitive to state these results in terms of the deviation $a$ which has probability $\exp(-b)$.  It will be clear that this is equivalent.  Let $\DEV_{\triangle}(b,t)$ denote the triangle count deviation which has probability only $e^{-b}$.  That is, $\DEV_{\triangle}(b,t)$ is the minimal value of $a$ such that
\[
\pr{N_{\triangle}(G_m)\, >\, \Ex{N_{\triangle}(G_m)}\, +\, a}\, \le \, e^{-b}\, .
\]

It is said that a picture is worth a thousand words.  So let us present a figure first, and then state our result more formally.  In particular the figure will highlight the four different regimes which correspond to different potential ``causes'' of the triangle count deviation.

We may parameterise the deviation $a$ as $\delta t^3n^3$, which is practically $\delta$ times the mean.  The figure considers densities $t$ of the form $t=n^{\gamma}$, and $\delta$ of the form $\delta=n^{\theta}$.  For each $\gamma\in (-1/2,0)$ we obtain results for deviations between the order of magnitude of the standard deviation (order $t^{3/2}n^{3/2}$) and the order of magnitude of the largest possible deviation (order $t^{3/2}n^3$).  This explains the interval of values of $\theta \in (-3/2-3\gamma/2,-3\gamma/2)$.

\begin{figure}[H]
\centering

\begin{tikzpicture}
%\addplot[name path=gas,very thick] coordinates {(15,0) (30,30)};
%\addplot[name path=liq,very thick] coordinates {(0,30) (2,22) (30,30)};
%\addplot[name path=help1] coordinates {(30,30) (0,30) (0,0)};
%\addplot[name path=help2] coordinates {(15,0) (30,0) (30,30)};

 %y axis line style={draw opacity=0},    tick style={line width=6pt},

%xtick={0.1,5.6,11.2,16.8,22.4,28,33.6,39.2,44.8,50.4,56,61.6,67.2,72.8,78.4,84,89.6,95.2,100.8,106.4,112,117.6,123.2,126},
 %   xticklabels={$\hspace{1mm} -1.5$, , , , , $-1$, , , , , $-0.5$, , , , , $0$, , , , , $0.5$, , ,$0.75$ },

\begin{axis}[xmin=0,ymin=0,xmax=136,ymax=33,axis lines=middle, 
standard,  axis line style={->},
xlabel=$\theta$,
ylabel=$\gamma$,
                minor xtick={0.1,5.6,11.2,16.8,22.4,28,33.6,39.2,44.8,50.4,56,61.6,67.2,72.8,78.4,84,89.6,95.2,100.8,106.4,112,117.6,123.2,126},
                tick style={line width=1pt},
          xtick={0.1,28,56,84,112,126},
          xticklabels={$\hspace{1mm} -1.5$, $-1$, $-0.5$, $0$, $0.5$, $0.75$ },
          minor ytick={2.8,8.4,14,19.6,25.2},
     ytick={0.1,5.6,11.2,16.8,22.4,28},
    yticklabels={$\vspace{-2mm}-0.5$,$-0.4$,$-0.3$,$-0.2$,$-0.1$,$0$}]
    
    \addplot[name path=LHS,thick] coordinates {(0,28) (0,0) (42,0)};
    \addplot[name path=toleft,thick] coordinates {(0,28) (42,0)};

\addplot[name path=left,thick] coordinates {(0,28) (42,0) (63,0)};
\addplot[name path=N,thick] coordinates {(0,28) (28,28) (60,4) (63,0)};
\addplot[name path=NC,thick] coordinates {(60,4) (63,0)};
\addplot[name path=star,thick] coordinates {(28,28) (84,0) (60,4)};
\addplot[name path=C1,thick] coordinates {(60,4) (84,0) (63,0)};
\addplot[name path=H,thick] coordinates {(84,0) (28,28) (83.5,28)};
\addplot[name path=C2,thick] coordinates {(84.5,28) (126,0) (84.5,0)};
\addplot[name path=LDL] coordinates {(83.5,0) (83.5,28) (84.5,28) };
\addplot[name path=LDR] coordinates {(83.5,0) (84.5,0) (84.5,28)};

%\node[align=right] at (0,0) {{\tiny $(0,0)$}};

%\addplot[blue] fill between[of=gas and liq];

\addplot[dgreen] fill between[of=left and N];
\addplot[lblue] fill between[of= N and star];
\addplot[dred] fill between[of= NC and C1];
\addplot[blue] fill between[of= H and LDL];
\addplot[dred] fill between[of= LDR and C2];
\addplot[lgr] fill between[of= LHS and toleft];
\addplot[purp]  fill between[of= LDL and LDR];

\end{axis}

\end{tikzpicture}

\caption{In this figure, grey represents very small deviations, and the purple line ($\theta=0$) corresponds to the traditional large deviation results.  Each of the other colours represents a different ``most likely cause'' of the corresponding deviation.   In the green region: Good luck without a structural cause.  In the light blue region: A star.  In the dark blue region: A hub.  In the red regions: A clique.} \label{fig:ta}
\end{figure}
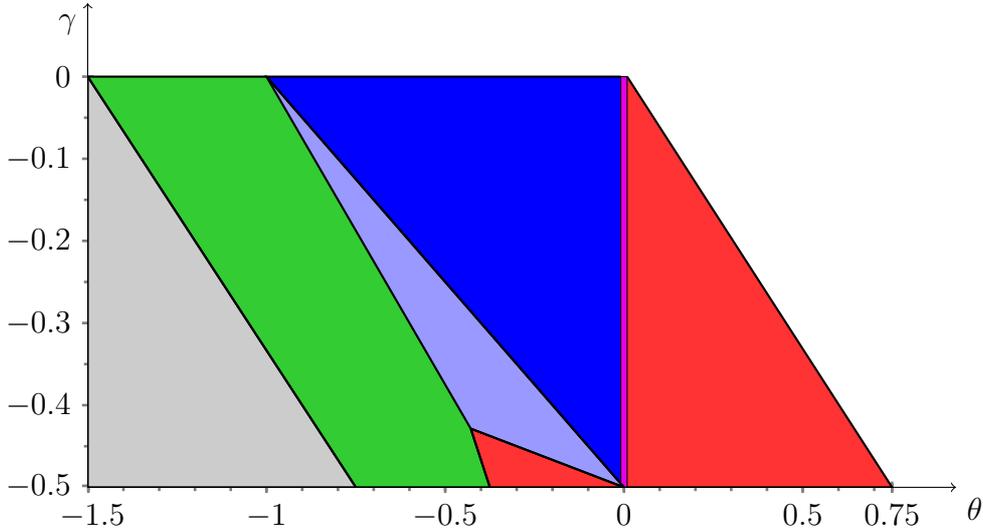

Let us now discuss each of the four possible causes of deviations.  Here and throughout let $\ell:=\log(1/t)$, the natural logarithm of $1/t$.  

\textbf{Pure good luck (normal):} The proof is based on a martingale related to the creation of triangles during the process (of adding edges).  It is very likely that the sum of the conditional variances of the increments of the martingale is of the order of magnitude $t^3n^3$.  It will follow (using Freedman's inequalities) that a deviation of order
\[
\NORMAL(b,t)\, :=\, b^{1/2}t^{3/2}n^{3/2}
\]
has probability $e^{-b}$.  This behaviour is an extension of the central limit theorem, and the behaviour of deviations of the order of the standard deviation.  For this reason we call it the ``normal'' regime.

\textbf{A star:} Observe that a vertex with particularly large degree, $d$, is likely to be contained in $\Theta(d^2 t)$ triangles.  And so a star with degree $d\gg tn$ ``causes'' $\Theta(d^2 t)$ triangles, and occurs with probability approximately $t^d=e^{-d\ell}$.  Clearly this only applies for degrees $d<n$.  We see that this may cause a deviation of order 
\[
\STAR(b,t)\, := \, \frac{b^2t}{\ell^2}\, 1_{b\le n\ell} 
\]
with probability $e^{-b}$.

\textbf{A hub:} The hub is the natural extension of the star once it has reached full degree.  Observe that $k$ vertices with degree of order $n$ will be involved in $\Theta(ktn^2)$ triangles, and this occurs with probability approximately $t^{kn}=e^{-kn\ell}$.  We see that this may cause a deviation of order
\[
\HUB(b,t)\, :=\, \frac{btn}{\ell}\, 1_{b\ge n\ell} 
\]
with probability $e^{-b}$.

\textbf{A clique:} A clique of $k$ vertices creates $\Theta(k^3)$ triangles, and occurs with probability approximately $t^{k^2}=e^{-k^2\ell}$.  We see that this may cause a deviation of order
\[
\CLIQUE(b,t)\, :=\, \frac{b^{3/2}}{\ell^{3/2}}
\]
with probability $e^{-b}$.

Let $M(b,t)$ be defined to be the maximum of these four,
\eq{Mdef}
M(b,t)\, =\, \max\{\NORMAL(b,t),\STAR(b,t),\HUB(b,t),\CLIQUE(b,t)\}\, .
\eqe

%Let us also define $\DEV_{\triangle}(b,t)$ to be the triangle count deviation which has probability only $e^{-b}$.  That is, let $\DEV_{\triangle}(b,t)$ be the minimal value of $a$ such that
%\[
%\pr{N_{\triangle}(G_m)\, >\, \Ex{N_{\triangle}(G_m)}\, +\, a}\, \le \, e^{-b}\, .
%\]

Our main theorem for the $G(n,m)$ model states that the deviation $\DEV_{\triangle}(b,t)$ which has probability $e^{-b}$ is given by $M(b,t)$, up to a multiplicative constant, across a large range of $t$ and $b$.  Note we restrict $t$ to be at most $1/2$, even though the results certainly hold in all cases where the density is bounded away from $1$, see~ \cite{GGS}.

\begin{theorem}\label{thm:mainm} There exist absolute constants $c,C$ such that the following holds.
For all $Cn^{-1/2}(\log{n})^{1/2}\le t\le 1/2$ and $3\log{n}\le b\le tn^2\ell$ we have
\[
cM(b,t)\, \le\, \DEV_{\triangle}(b,t)\, \le \, CM(b,t)\, .
\]
\end{theorem}

\begin{remark} The behaviour of deviations is likely to be different in sparse random graphs, with $t\ll n^{-1/2}$, as in this case most edges are not in any triangles.  This is the so called Poisson regime.  See~\cite{HMS} for large deviation results in this regime and~\cite{GHN} for results in the even sparser regime with constant average degree.  It seems likely that our result should extend down to $t=Cn^{-1/2}$.  The complication is that our results are based on the behaviour of codegrees.  As the expected codegree is around $t^2n$ it becomes more difficult to prove concentration results when $t^2n$ is only a constant.
\end{remark}

Here is another figure that may help us visualise the result.  Here we have set $t=n^{\gamma}$ and $b=n^{\eta}$.

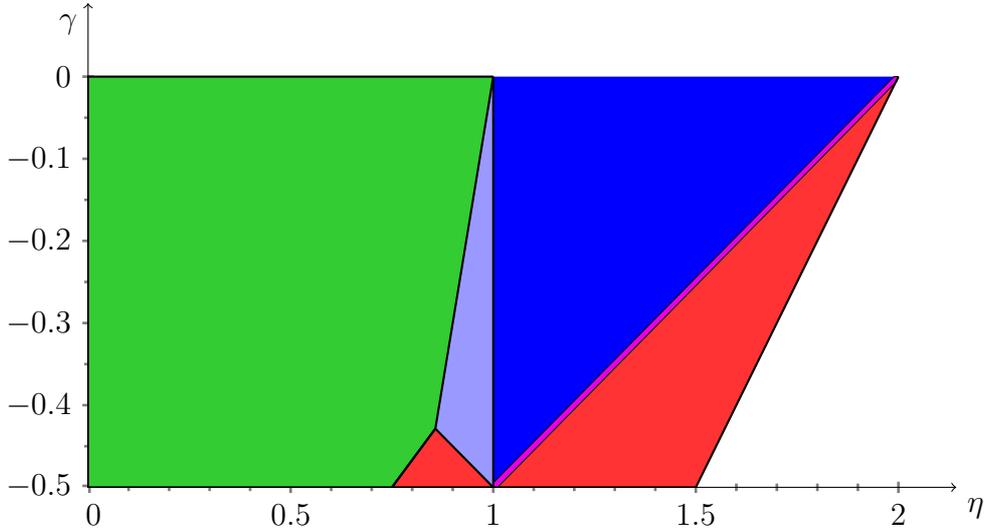
\begin{figure}[H]
\centering

\begin{tikzpicture}
%\addplot[name path=gas,very thick] coordinates {(15,0) (30,30)};
%\addplot[name path=liq,very thick] coordinates {(0,30) (2,22) (30,30)};
%\addplot[name path=help1] coordinates {(30,30) (0,30) (0,0)};
%\addplot[name path=help2] coordinates {(15,0) (30,0) (30,30)};

 %y axis line style={draw opacity=0},    tick style={line width=6pt},

%xtick={0.1,5.6,11.2,16.8,22.4,28,33.6,39.2,44.8,50.4,56,61.6,67.2,72.8,78.4,84,89.6,95.2,100.8,106.4,112,117.6,123.2,126},
 %   xticklabels={$\hspace{1mm} -1.5$, , , , , $-1$, , , , , $-0.5$, , , , , $0$, , , , , $0.5$, , ,$0.75$ },

\begin{axis}[xmin=0,ymin=0,xmax=60,ymax=33,axis lines=middle, 
standard,  axis line style={->},
xlabel=$\eta$,
ylabel=$\gamma$,
                minor xtick={2.8,8.4,19.6,25.2,5.6,11.2,16.8,22.4,30.8,33.6,36.4,39.2,44.8,47.6,50.4,53.2},
                tick style={line width=1pt},
          xtick={0.1,14,28,42,56},
          xticklabels={$\hspace{1mm} 0$, $0.5$, $1$, $1.5$, $2$},
          minor ytick={2.8,8.4,14,19.6,25.2},
     ytick={0.1,5.6,11.2,16.8,22.4,28},
    yticklabels={$\vspace{-2mm}-0.5$,$-0.4$,$-0.3$,$-0.2$,$-0.1$,$0$}]
    
    \addplot[name path=LHS,thick] coordinates {(28,28) (0,28) (0,0) (21,0)};
\addplot[name path=N,thick] coordinates {(28,28) (24,4) (21,0)};
\addplot[name path=NC,thick] coordinates {(24,4) (21,0)};
\addplot[name path=star,thick] coordinates {(28,28) (28,0)};
\addplot[name path=C1,thick] coordinates {(24,4) (28,0) (21,0)};
%\addplot[name path=H,thick] coordinates {(28,28) (56,28) (28,0)};

\addplot[name path=C2,thick] coordinates {(56,28) (42,0) (28,0)};
%\addplot[name path=HC,thick] coordinates {(56,28) (28,0)};

\addplot[name path=HCl] coordinates {(28,0) (28,0.3) (55.7,28) (56, 28)};
\addplot[name path=HCr] coordinates {(56, 28) (55.9,27.7) (28.3,0) (28,0)};

%\node[align=right] at (0,0) {{\tiny $(0,0)$}};

%\addplot[blue] fill between[of=gas and liq];

\addplot[dgreen] fill between[of=LHS and N];
\addplot[lblue] fill between[of= N and star];
\addplot[dred] fill between[of= NC and C1];
\addplot[blue] fill between[of= HCl and star];
\addplot[dred] fill between[of= HCr and C2];
\addplot[purp] fill between[of= HCl and HCr];

\end{axis}

\end{tikzpicture}

\caption{As in the previous figure, the purple line corresponds to the traditional large deviation results.  Each of the other colours represents a different ``most likely cause'' of the corresponding deviation.   In the green region: Good luck without a structural cause.  In the light blue region: A star.  In the dark blue region: A hub.  In the red regions: A clique.} \label{fig:tb}
\end{figure}

Some readers may prefer a version of the result which expresses the probability as a function of the deviation, so we state this as an immediate corollary of Theorem~\ref{thm:mainm}.  Let us define
\[
r(t,a)\, :=\, \min\left\{\frac{a^2}{t^3n^3}\, ,\, \frac{a^{1/2}\ell}{t^{1/2}}\, +\, \frac{a\ell}{tn}\, ,\, a^{2/3}\ell\right\}\, .
\]

\begin{cor}\label{cor:mainm} There exist absolute constants $c,C$ such that the following holds.  For all $Cn^{-1/2}(\log{n})^{1/2}\le t\le 1/2$ and $Ct^{3/2}n^{3/2}\sqrt{\log{n}}\le a\le ct^{3/2}n^3$, we have
\[
\exp(-Cr(t,a))\, \le\, \pr{D_{\triangle}(G_m)\, >\, a}\, \le\, \exp(-cr(t,a))\, .
\]
\end{cor}

We also mention that our methods give some results for the lower tail.  When discussing the lower tail it is more difficult to talk about ``causes'' of deviations.  Although some progress has been made in this direction by Neeman, Radin and Sadun~\cite{NRS}. In the dense setting, with $t$ a constant, they show that the lower tail is closer related to the smallest eigenvalue of the adjacency matrix.  Our methods cannot really deal with this type of structure and so our lower tail result is restricted to the normal regime.

\begin{theorem}\label{thm:mainlower} There exist absolute constants $c,C$ such that the following holds.  For all $Cn^{-1/2}(\log{n})^{1/2}\le t\le 1/2$ and $Ct^{3/2}n^{3/2}\sqrt{\log{n}}\le a\le ct^{3/2}n^2$, we have
\[
\exp\left(\frac{-Ca^2}{t^3n^3}\right)\, \le\, \pr{D_{\triangle}(G_m)\, <\, -a}\, \le\, \exp\left(\frac{-ca^2}{t^3n^3}\right)\, .
\]
\end{theorem}

\subsection{Results for triangle counts in $G(n,p)$}\label{sec:mainp}

Understanding the behaviour of the upper tail for the number of triangles in the Erd\H{o}s-R\'enyi random graph $G(n,p)$ has been a major focus in probabilistic combinatorics in recent decades.    We use the following notation to denote the asymptotic rate associated with a deviation of size $a=\delta p^3 (n)_3$ above the mean:
\[
\rdpn\, :=\, -\log\pr{N_{\triangle}(G_p)\, >\, (1+\delta)p^3(n)_{3}}\, ,
\]
where $G_p\sim G(n,p)$.   This asymptotic rate was considered by Vu~\cite{Vu} and Janson and Ruci\'nski~\cite{JR} who gave a number of methods for bounding $\rdpn$, in the early 2000s. A little later Kim and Vu\cite{KVtri} and Janson, Oleszkiewicz and Ruci\'nski~\cite{JOR} independently arrived at the bounds
\[
c(\delta) p^2n^2\, \le\, \rdpn\, \le\, C(\delta) p^2n^2\log(1/p) \, 
\]
for some constants $c(\delta),C(\delta)$.  In other words, the value of $\rdpn$ was determined up to a logarithmic error factor.  Around 2010, the problem of finding the correct order of magnitude of $\rdpn$ was resolved independently by Chatterjee~\cite{chatterjee} and De Marco and Kahn~\cite{DK}.  They discovered that
\[
c(\delta) p^2n^2\log(1/p)\, \le\, \rdpn\, \le\, C(\delta) p^2n^2\log(1/p) \, .
\]
Attention then turned to the precise asymptotic behaviour of $\rdpn$.

In the case of constant $p\in (0,1)$, Chatterjee and Varadhan~\cite{CV} showed that finding $\rdpn$ reduces to finding the solution of a variational problem for graphons \footnote{See Lubetzky and Zhao~\cite{LZ} and the survey of Chatterjee~\cite{ChatS} for more on the dense case.}.  Chatterjee and Dembo~\cite{CD} extended the framework of~\cite{CV} to the sparse case, provided $p\to 0$ relatively slowly, specifically $p\ge n^{-1/42}\log{n}$.  By solving the variational problem, Lubetzky and Zhao~\cite{LZsparse} obtained
\eq{LZsay}
\rdpn\, =\, (1+o(1)) \min\left\{\frac{\delta^{2/3}}{2},\frac{\delta}{3}\right\} p^2n^2\log(1/p)
\eqe
whenever $n^{-1/42}\log{n}\le p\ll1$.  With subsequent improvements by Eldan~\cite{eldan}, Cook and Dembo~\cite{cookdembo} and Augeri~\cite{augeri} we now know that this framework may be applied, in this case, whenever $p\gg n^{-1/2}(\log{n})^2$.  This shows that \eqr{LZsay} holds for all $n^{-1/2}(\log{n})^2\ll p\ll 1$.

Recently, Harel, Mousset and Samotij~\cite{HMS} essentially answered completely the large deviation upper tail problem for triangles in $G(n,p)$.  Their approach combines ideas from the earlier combinatorial approaches with a new concept of entropic stability which they introduce.  In particular, they confirm that~\eqr{LZsay} holds whenever $p\gg n^{-1/2}$, prove that
\[
\rdpn\, =\, (1+o(1)) \frac{\delta^{2/3}}{2} p^2n^2\log(1/p)
\]
whenever $n^{-1}\ll p\ll n^{-1/2}$ and find the asymptotic value of $\rdpn$ if the remaining case that $p^2 n\to c\in \mathbb{R}$.  The expression in this final case is more complex.  The $\delta/3$ is minimal in~\eqr{LZsay} when the ``easiest'' way to ``add'' $\delta p^2(n)_{3}$ triangles is to add a hub like structure and the $\delta^{2/3}/2$ is minimal when the ``easiest'' way is to add an appropriately sized clique.  In the intermediate case $p^2n\to c$ a hybrid construction is required, see~\cite{HMS} for more details.   The remaining ``missing case'' of very sparse random graphs, with constant average degree, was completed recently by Ganguly, Hiesmayr and Nam~\cite{GHN}.

There is also an extensive literature based around the question of central limit theorems for subgraph counts. Since Ruci\'nski~\cite{Ruc} established the central limit theorem for subgraph counts many articles giving sequentially stronger bounds on convergence have appeared~\cite{BKR,KRT,RR,Rol}.  Furthermore, Janson~\cite{Jan} (building on the earlier articles, Janson~\cite{JanRSA} and Janson and Nowicki~\cite{JN}) proved a functional central limit theorem, and that subgraph counts in $G(n,m)$ are also asymptotically normally distributed.

There has been far less investigation of the probability of moderately large deviations, between the orders of magnitude of the standard deviation and the mean. D\"oring and Eichelsbacher~\cite{DE} proved that
\eq{asy}
\rdpn\, =\, \frac{-\delta_n^2 p n^2}{36 (1-p)}\, +\, o(\delta_n^2 pn^2)\, ,
\eqe
whenever $p^{-1/2}n^{-1}\ll \delta_n \ll p^7$ (see also~\cite{DE2}).  While in the dense case $p\in (0,1)$, F\'eray, M\'eliot and Nikeghbali~\cite{FGN} found the asymptotics of the deviation probability (this is much stronger than knowing the asymptotics of the rate): 
\begin{align}\label{eq:smalldelta}
\pr{N_{\triangle}(G_p)\, >\, (1+\delta)p^3(n)_{3}}\, &  = \phantom{\Bigg{|}}\\ \phantom{\Bigg{|}} (1+o(1))\sqrt{\frac{9(1-p)}{\pi p}} &\exp\left(- \frac{\delta_n^2 pn^2}{36(1-p)}\, +\, \frac{\big(7-8p\big) \delta_n^3 pn^2}{324(1-p)^2}\, -\, \log (n\delta_n)  \right)\nonumber
\end{align}
whenever $n^{-1}\ll \delta_n \ll n^{-1/2}$.

We can see that these results taken together only cover a small part of the space of possible deviations.  Indeed, consider the cases $p=n^{\gamma}$ and $\delta=n^{\theta}$ with $\gamma\in (-1/2,0)$ and $\theta\in (-1,0)$.

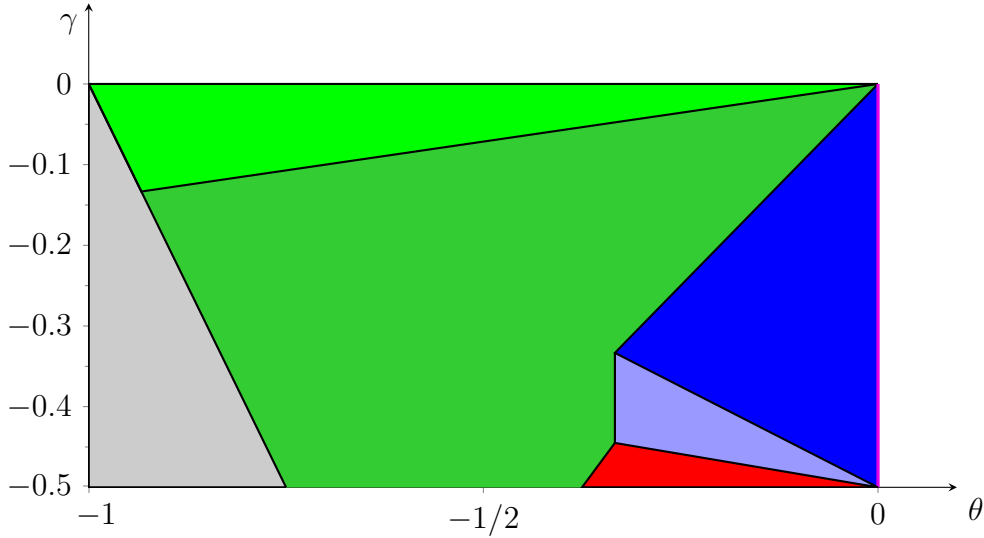
\begin{figure}[H]
\centering

\begin{tikzpicture}

\begin{axis}[xmin=0,ymin=0,xmax=33,ymax=36,axis lines=middle,
    axis line style={->},
   standard,
    xlabel=$\theta$,
    ylabel=$\gamma$,
    xtick={0.01,15,30},
    xticklabels={$\hspace{2mm} -1$,$-1/2$,$0$},
   minor ytick={3,9,15,21,27},
     ytick={0.1,6,12,18,24,30},
    yticklabels={$\vspace{-2mm}-0.5$,$-0.4$,$-0.3$,$-0.2$,$-0.1$,$0$}]
    
\addplot[name path=H,thick] coordinates {(30,0) (20,10) (30,30)};
\addplot[name path=known,thick] coordinates {(0,30) (2,22) (30,30)};
\addplot[name path=left,thick] coordinates {(0,30) (0,0) (7.5,0)};
\addplot[name path=sd,thick] coordinates {(0,30) (7.5,0)};  
\addplot[name path=top,thick] coordinates {(30,30) (0,30)};
\addplot[name path=right, very thick, color=purp] coordinates {(30,0) (30,30)};
\addplot[name path=S,thick] coordinates {(30,0) (20,3.3) (20,10)};
\addplot[name path=C,thick] coordinates {(30,0) (18.75,0) (20,3.3) };

%\addplot[name path=Cbot, thick] coordinates {(18.75,0) (30,0)};

%\node[align=right] at (5.6,21.3) {{\tiny $(-14/15,-2/15)$}};

%\addplot[blue] fill between[of=gas and liq];
\addplot[dgreen] fill between[of=sd and known];
\addplot[dgreen] fill between[of=sd and H];
\addplot[dgreen] fill between[of=sd and C];

\addplot[lgr] fill between[of=left and sd];
\addplot[green] fill between[of=known and top];
\addplot[blue] fill between[of=H and right];
\addplot[lblue] fill between[of=H and S];
\addplot[red] fill between[of=C and S];

\end{axis}

\end{tikzpicture}

\caption{{\small The result of D\"oring and Eichelsbacher~\cite{DE} applies in the light green region.  We extend the same bound to the entire ``normal'' regime, shown here in green, see Theorem~\ref{thm:mainp}.  In the three coloured regions on the right, localisation (described below) occurs.  The three colours correspond to the corresponding type of localisation, light blue for a star, dark blue for a hub and red for a clique.  We remark that the grey region corresponds to deviations smaller than the order of the standard deviation, which have probability $\frac{1}{2}+o(1)$ by the central limit theorem.  (The traditional large deviation results lie in the purple line along the right hand side.)}} \label{fig:p}
\end{figure}

We say that localisation occurs if the most likely way to achieve the deviation involves a small object being present in the random graph, such as a star, hub or clique.  In the non-localised region the deviation is more likely to be caused by there being extra edges in the random graph (but with no specific structure).  This concept is similar to that of replica symmetry and symmetry breaking considered by Lubetzky and Zhao~\cite{LZ} in the dense case.  Since we are working in the sparse case it seems more natural to talk in terms of localisation.  Harel, Mousset and Samotij~\cite{HMS} show in the case of fixed $\delta>0$ that their concept of entropic stability implies localisation.  It seems quite likely that some adaptation of their methods could give the correct asymptotic rate across the regimes of localisation, as the same authors have done recently for the equivalent results for arithmetic progressions~\cite{HMS2}.

One of our main results states that the asymptotic rate~\eqr{asy} given by D\"oring and Eichelsbacher~\cite{DE} in the light green region holds across the whole non-localised region.

\begin{theorem}\label{thm:mainp} Let $n^{-1/2}(\log{n})^{1/2}\ll p\ll 1$ and let $\delta_n$ be a sequence satisfying 
\[
p^{-1/2}n^{-1}\, \ll\, \delta_n\, \ll \, p^{3/4}(\log{n})^{3/4}\, ,\, n^{-1/3}(\log{n})^{2/3}\, +\, p\log(1/p)\, .
\]
Then
\[
\rdn\, =\, (1+o(1))\frac{\delta_n^2 pn^2}{36}\, .
\]
\end{theorem}

\begin{remark} See Proposition~\ref{prop:mainp} for more detail about the true order of magnitude of the error term.\end{remark}

Our techniques are not so well suited to the localised region.  However, we are able to attain the order of magnitude of $\rdpn$ throughout this region.

\begin{theorem}\label{thm:localp} 
Let $n^{-1/2}(\log{n})^{1/2}\ll p\ll 1$ and let $\delta_n$ be a sequence satisfying 
\[
p^{3/4}(\log{n})^{3/4}\, ,\, n^{-1/3}(\log{n})^{2/3}\, +\, p\log(1/p)\, \le\, \delta_n\, \le 1\, .
\]
Then
\[
\rdn\, =\, \Theta(1) \min\{\delta_n^{2/3}p^2n^2\log{n}, \delta_n^{1/2}pn^{3/2}\log{n}\, +\, \delta_n p^2n^2 \log(1/p)\}\,  .
\]
\end{theorem}

\subsection{Our results for cherry counts}

The corresponding problem for cherries (paths of length two) is somewhat easier.  And we would not be surprised if these results on cherries could be obtained by other methods.  However, it is very useful to state and prove these results for cherries as it provides a more elementary application of our method.  We hope this will help the reader understand the essential idea of the method in a simpler setting.
 
Let $N_{\owedge}(G)$ be the number of isomorphic copies of the path of length two in the graph $G$ (this is simply double the count without multiplicity).

Again we will see that the results break into various regimes.  This time there are only three.  The corresponding functions for cherry count deviations are:
\begin{align*}
\NORMAL_{\owedge}(b,t) \, & :=\, b^{1/2}tn^{3/2}\, ,\\
\STAR_{\owedge}(b,t)\, &:=\, \dfrac{b^2}{\ell^2} 1_{b < n\ell} \, ,
\end{align*}
and 
\[
\HUB_{\owedge}(b,t)\, :=\, \dfrac{bn}{\ell}1_{b \geq n \ell}\, . 
\]

Let 
\[
M_{\owedge}(b,t) = \max\{\NORMAL_{\owedge}(b,t), \STAR_{\owedge}(b,t), \HUB_{\owedge}(b,t)\}.
\]
and define $\DEV_{\owedge}(b,t)$ as the minimal value of $a$ such that 
\[
\pr{N_{\owedge}(G_m) > \Ex{N_{\owedge}(G_m)} + a}\,  \leq\, e^{-b}\, .
\]
That is, $\DEV_{\owedge}(b,t)$ describes the size of deviation which has probability $e^{-b}$.

Our main theorem concerning the cherry count deviation determines the order of magnitude of $\DEV_{\owedge}(b,t)$ across (essentially) the entire range of the parameters $b$ and $t$.
%%% Our main theorem on the cherry count deviation determines the order of magnitude of $\DEV_{\owedge}(b,t)$ across essentially the complete range of the parameters $b$ and $t$.

\begin{restatable}{theorem}{cherries}\label{thm:cherries}
There exist absolute constants $c, C$ such that the following holds. Suppose that $2n^{-1}\log{n}\le t \le 1/2$ and that $3 \log n \leq b \leq tn^2 \ell$.  Then
\[
cM_{\owedge}(b,t) \leq \DEV_{\owedge}(b,t) \leq CM_{\owedge}(b,t)\, .
\]
\end{restatable}

We stress that this result is given as a secondary result to introduce methods which will be used in the upper bound part of Theorem~\ref{thm:mainm}.  For this reason we only give a sketch proof of the lower bound, see Section~\ref{sec:cherrieslower}.

\subsection{An informal discussion of our approach}

As we have mentioned, the results for the $G(n,p)$ setting will follow from results proved for the $G(n,m)$ model.  In this subsection we give a quick intuitive idea of the approach to bounding triangle count deviations in $G(n,m)$.

Naturally, the approach has some differences based on the regime in question.  For definiteness, let us discuss the normal regime, i.e., the regime in which $\NORMAL(b,t)=b^{1/2}t^{3/2}n^{3/2}$ is the maximum of the four, and so $M(b,t)=\NORMAL(b,t)$.  In this context our objective is to prove that for some constants $c<C$, we have
\[
\pr{D_{\triangle}(G_m)\, \ge \, cb^{1/2}t^{3/2}n^{3/2}}\, \ge\, \exp(-b)
\]
and
\[
\pr{D_{\triangle}(G_m)\, \ge \, Cb^{1/2}t^{3/2}n^{3/2}}\, \le\, \exp(-b)\, .
\]

As the standard deviation $\sigma$ of $N_{\triangle}(G_m)$ is of order $t^{3/2}n^{3/2}$, this corresponds to an (approximate) extension of central limit theorem.  We recall that more precise bounds were proved in~\cite{GGS}, but these results seem to be difficult to extend to sparser random graphs.

In common with~\cite{GGS} our approach is based on a martingale expression for the deviation $D_{\triangle}(G_m)$ and then using Freedman's martingale inequalities (which we state in the next section) to give both upper and lower bounds on the probability of deviations.  This expression (discussed in more detail in Section~\ref{sec:Mart})
\[
D_{\triangle}(G_m)\, =\, \sum_{i=1}^{m} \left[3\frac{(N-m)_{2}(m-i)}{(N-i)_3}\, X_{\owedge}(G_i)\, +\, \frac{(N-m)_3}{(N-i)_3}\, X_{\triangle}(G_i)\right]\, ,
\]
expresses the triangle count deviation in terms of increments which are related to the created of cherries ($X_{\owedge}(G_i)$) and triangles ($X_{\triangle}(G_i)$) in the $i$th step of the process.

The challenges associated with this approach include:
\begin{itemize}
\item understanding the variance, or more precisely, the likely quadratic variation, of the process, and
\item controling the effect of large increments on the course of the martingale.
\end{itemize}

The term $X_{\owedge}(G_i)$ may cause a large increment if the new edge $e_i=uw$ is such that one of the degrees $d(u)$ or $d(v)$ is far from the average.  The term $X_{\triangle}(G_i)$ may cause a large increment if the new edge $e_i=uw$ is such that the codegree $d(u,w)$ is far from its expected value. 

We therefore require a good understanding of degrees, codegrees, and their deviations.  We use truncations and the so-called divide and conquer martingale approach.  Informally, this means that we use our understanding of degrees and codegrees to control the effect of large increments, while using the martingale approach to control the remaining increments.

Unfortunately, as we shall see, this approach becomes quite technical.  We hope at least that this informal discussion puts the rest of the proof in context.

\subsection{Overview of the article}

In Section~\ref{sec:Aux} we introduce necessary background and inequalities that will be used throughout the article.  We then turn to the properties of degrees in $G(n,m)$, we prove a number of bounds in Section~\ref{sec:degs}.  In Section~\ref{sec:cherries}, we prove (the upper bound part of) Theorem~\ref{thm:cherries} about cherries.

We then begin to prepare ourselves for the results on triangles.  We begin by studying codegrees in $G(n,m)$ in Section~\ref{sec:codegs}.  We then prove the required upper bounds on deviations in Section~\ref{sec:UBtri}, and lower bounds in Section~\ref{sec:LBtri}, which together prove Theorem~\ref{thm:mainm}.  As the lower bound results for cherries are similar, we discuss them in Section~\ref{sec:cherrieslower}.

We then show how the results may be transferred to $G(n,p)$ in Section~\ref{sec:pworld}.  We give some concluding remarks in Section~\ref{sec:ConcRem}.  The article also includes an appendix where we include a few of the proofs which we consider to be less directly relevant.

\section{Important auxiliary results and inequalities}\label{sec:Aux}

In this section we give some important auxiliary results and inequalities which we shall use throughout the article.

We first remark that we write $N$ for $\binom{n}{2}$, and set $t:=m/N$, which is the density of the random graph $G_m\sim G(n,m)$.  We will sometimes consider the Erd\H{o}s-R\'enyi random graph process for generating $G_m$.  The process $G_0,G_1,\dots ,G_m$ may be generated from the empty graph $G_0$ (with $n$ vertices) by adding a uniformly random edge at each step, chosen from amongst the edges not already present.  Equivalently, one may consider $\{e_1,\dots ,e_N\}$, a uniformly random permutation of the edges of $K_n$, and define the sequence of graphs $G_0,\dots ,G_m$ by setting $E(G_i)=\{e_1,\dots,e_i\}$.  In this context we denote by $s$ the density $s:=i/N$ of the graph $G_i$, and we note that $G_i\sim G(n,i)$.

When it is convenient, and not likely to confuse, we ignore the problem of rounding to the nearest integer, and floor signs, and so on.

We now begin stating the required auxiliary results.  

\subsection{Chernoff inequalities and related inequalities} We shall state Chernoff inequalities for binomial or hypergeometric random variables, see Theorems 2.1 and 2.10 of~\cite{JLR}. Note that the last two inequalities follow from~\eqr{h1}.

\begin{lem}\label{lem:Chern} Let $X$ be a binomial or hypergeometric random variable and let $\mu=\Ex{X}$. Then, for all $a>0$, we have
\eq{h1}
\pr{X\, \ge\, \mu +a}\,  \le\, \exp\left(\frac{-a^2}{2\mu+2a/3}\right)
\eqe
and
\eq{h2}
\pr{X\, \le\, \mu -a}\, \le\, \exp\left(\frac{-a^2}{2\mu}\right)\, .
\eqe
For $\theta\ge e$ we have
\eq{h3}
\pr{X\, \ge\, \theta\mu}\, \le\, \exp\left(- \theta \mu \big(\log(\theta)-1\big)\right)\, .
\eqe
Consequently, for $j\ge 3$ and any $\nu\ge \mu$ we have
\eq{h4}
\pr{X\, \ge\, 2^j \nu}\, \le\, \exp\left(-j2^{j-2}\nu \right)\, .
\eqe
\end{lem}

It is sometimes also useful to consider random variables which are sums of binomial or hypergeometric random variables.  If $X$ is a sum of two such random variables then it follows from Lemma~\ref{lem:Chern} that:
\eq{h5}
\pr{|X-\mu|\, \ge\, a}\,  \le\, 4\exp\left(\frac{-a^2}{8\mu+2a}\right)
\eqe
and, for $j\ge 4$, and any $\nu\ge \mu$
\eq{h6}
\pr{X\, \ge\, 2^j \nu}\, \le\, 2\exp\left(-j2^{j-3}\nu \right)\, .
\eqe

In particular, this applies when we count edges between sets with multiplicity.  Given a graph $G$ we may define
\[
e(U,W)\, :=\, \sum_{u\in U}\sum_{w\in W}1_{uw\in E(G)}\, .
\]
Note that setting $X=e(U,W)$, and taking $G$ with distribution $G(n,m)$ or $G(n,p)$ then $X$ is a sum of two binomial or two hypergeometric random variables, and so inequalities~\eqr{h5} and~\eqr{h6} hold.

\subsection{Martingale deviation inequalities}

We begin with the Hoeffding-Azuma inequality~\cite{Hoeff,Azuma}. This inequality states that for a supermartingale $(S_i)_{i=0}^{m}$ with increments $(X_i)_{i=1}^{m}$ we have
\eq{HA}
\pr{S_m-S_0\, >\, a}\, \le \, \exp\left(\frac{-a^2}{2\sum_{i=1}^{m}c_i^2}\right)\, .
\eqe
for all $a>0$, where $c_i:=\|X_i\|_{\infty}$.

When increments are typically much smaller than their supremum the following inequality of Freedman gives better bounds.

\begin{lem}[Freedman's inequality]\label{lem:F} Let $(S_i)_{i=0}^{m}$ be a supermartingale with increments $(X_i)_{i=1}^{m}$ with respect to a filtration $(\F_i)_{i=0}^{m}$, let $R\in \mathbb{R}$ be such that $\max_i |X_i|\le R$ almost surely, and for each $j\ge 1$ let 
\[
V(j):=\sum_{i=1}^{j}\,\, \Ex{\,  X_i^2\, \big|\, \F_{i-1}}\, .
\]  
Then, for every $\alpha,\beta >0$, we have
\[
\mathbb{P}\big(S_j-S_0\, \ge\,  \alpha\quad \text{and}\quad V(j)\le \beta \quad \text{for some } j\big)\, \le\, \exp\left(\frac{-\alpha^2}{2(\beta+R\alpha)}\right)\, .
\]
In particular, 
\[
\mathbb{P}\big(S_j-S_0\, \ge\,  \alpha\big)\, \le\, \exp\left(\frac{-\alpha^2}{2(\beta+R\alpha)}\right)\, +\, \pr{V(j)>\beta} .
\]
\end{lem}

The following converse result shows that the above bound is often close to best possible. Let $\tau_\alpha$ be the least $j$ such that $S_j\ge S_0+\alpha$ (and infinity if this never occurs), and set
\[
T_{\alpha}\, :=\, V(\tau_{\alpha})\, ,
\]
where we interpret $V(\infty)$ as $\infty$.  The above inequality states that
\[
\pr{T_{\alpha}\le \beta}\, \le\, \exp\left(\frac{-\alpha^2}{2(\beta+R\alpha)}\right)\, .
\]
Freedman's converse inequality~\cite{F} is as follows.

\begin{lem}[Converse Freedman inequality]\label{lem:CF}
Let $(S_i)_{i=0}^{m}$ be a martingale with increments $(X_i)_{i=1}^{m}$ with respect to a filtration $(\F_i)_{i=0}^{m}$, let $R$ be such that $\max_i|X_i|\le R$ almost surely, and let $T_{\alpha}$ be as defined above.
Then, for every $\alpha,\beta >0$, we have
\[
\pr{T_{\alpha}\le \beta}\, \ge\, \frac{1}{2} \exp\left(\frac{-\alpha^2(1+4\delta)}{2\beta}\right),
\]
where $\delta>0$ is such that $\beta/\alpha\ge 9R\delta^{-2}$ and $\alpha^2/\beta \ge 16\delta^{-2}\log(64\delta^{-2})$.
\end{lem}

We shall also prove a corollary of Freedman's inequality applied to the $G(n,m)$ setting.  It is in the spirit of McDiarmid's inequality~\cite{McD}, which considers deviation probabilities for functions which satisfy a ``Lipschitz'' property.  However, we work with a weaker ``Lipschitz'' property where the analogue of the ``Lipschitz constant'' depends on the edges in question.

Let $\mathcal{G}_{n,m}$ be the family of graphs with $n$ vertices and $m$ edges.  We may consider an edit distance between graphs $G,G'\in \mathcal{G}_{n,m}$.  We consider two graphs $G,G'$ to be adjacent (distance $1$) if $E(G')=E(G)\setminus \{e\}\cup \{e'\}$ for some pair $e,e'\in E(K_n)$.  Given a function $\psi:E(K_n)\to \mathbb{R}^{+}$, we say that a function $f:\mathcal{G}_{n,m}\to \mathbb{R}$ is \emph{$\psi$-Lipschitz} if for every adjacent pair of graphs $G,G'\in \mathcal{G}_{n,m}$ we have
\[
\big|\, f(G)\, -\, f(G')\, \big|\, \le\, \psi(e)\, +\, \psi(e')\, ,
\]
where $E(G)\bigtriangleup E(G')=\{e, e'\}$.

\begin{prop}\label{cor:F}  Let $G_m\sim G(n,m)$ with $t=m/N$.  Given $\psi:E(K_n)\to \mathbb{R}^+$ and a $\psi$-Lipschitz function $f:\mathcal{G}_{n,m}\to \mathbb{R}$, we have
\[
\pr{f(G_m)\, -\,\Ex{f(G_m)}\, \ge \, a}\, \le\, \exp\left(\frac{-a^2}{24t\|\psi\|^2\, +\, 6a\psi_{\ourmax}}\right)
\]
for all $a\ge 0$, where $\|\psi\|^2:=\sum_{e\in E(K_n)}\psi(e)^2$ and $\psi_{\ourmax}:=\max_{e}\psi(e)$.

Furthermore, the same bound holds for $\pr{f(G_m)\, -\,\Ex{f(G_m)}\, \le \, -a}$.
\end{prop}

\begin{remark} We have not attempt to optimise the constants.\end{remark}

\begin{remark} We have stated this result for $G(n,m)$ as this is the context in which we are working.  It is clear that the result does not actually depend on the graph structure at all.  In particular, if $f:[N]^{(m)}\to \mathbb{R}$ then one may obtain the same result for the selection of a uniformly random $m$-element subset $G\subseteq [N]$.
\end{remark}

We now present the proof of Proposition~\ref{cor:F}.  The proof is based on an application of Freedman's inequality, and will rely on a bound on the increments in terms of the function $\psi$.

\begin{proof}  We do not lose any generality in assuming $m\le N/2$.  Indeed, if $m>N/2$, then simply observe that $f$ is also a function of the complementary graph $G_m^c\sim G_{N-m}$.  The ``Furthermore'' statement follows by applying the inequality to $-f$.

We prove the result by considering the martingale
\[
Z_i\, :=\, \Ex{f(G_m)\, |\, G_i}\, \qquad i=0,\dots,m
\]
where $G_0,\dots ,G_m$ is the Erd\H os-R\'enyi random graph process.  Equivalently, let $e_1,\dots, e_N$ be a uniform random ordering of the edges of $K_n$ and for each $i$ define $G_i$ to be the graph with edges $\{e_1,\dots ,e_i\}$.   We require the following bound on the increments of this martingale
\eq{needinc}
\big|Z_i\, -\, Z_{i-1}\big|\, \le\, \psi(e_i)\, +\, 2\bar{\psi} \, ,
\eqe
where $\bar{\psi}=N^{-1}\sum_{e\in E(K_n)}\psi(e)$.

Let us show that the proposition follows if we assume~\eqr{needinc}.  In this case we may immediately bound the absolute value of increments,
\[
\big|Z_i\, -\, Z_{i-1}\big|\, \le\, \psi(e_i)\, +\, 2\bar{\psi}\, \le\, 3\psi_{\ourmax}\, ,
\]
and we obtain the following bound
\begin{align*}
\Ex{\big(Z_i\, -\, Z_{i-1}\big)^2\, \middle| \, G_{i-1}}\, &\le\, \Ex{\left(\psi(e_i)+\frac{2}{N}\sum_{e\in E(K_n)}\psi(e)\right)^2\, \middle| \,G_{i-1}}\phantom{\bigg|}\\
& \le \, 2 \Ex{\psi(e_i)^2|G_{i-1}}\, +\, \frac{8}{N^2} \left(\sum_{e\in E(K_n)}\psi(e)\right)^2\phantom{\bigg|}\\
& \le\, \frac{4}{N}\sum_{e\in E(K_n)} \psi(e)^2\, +\, \frac{8}{N}\sum_{e\in E(K_n)} \psi(e)^2\phantom{\bigg|}\\
& = \, \frac{12}{N}\|\psi\|_2^2\,\, .\phantom{\bigg|}
\end{align*}
Summing the above bound over $i\le m$ we obtain
\[
\sum_{i=1}^{m} \Ex{\big(Z_i\, -\, Z_{i-1}\big)^2\, \mid\, G_{i-1}}\, \le\, \frac{12m}{N}\|\psi\|_2^2 \, = \, 12t\|\psi\|_2^2\, .
\]
We now apply Freedman's inequality (Lemma~\ref{lem:F}) with $\alpha=a$, $\beta=12t\|\psi\|_2^2$ and $R=3\psi_{\ourmax}$, which yields
\[
\pr{f(G_m)\, -\,\Ex{f(G_m)}\, \ge \, a}\,   =\, \pr{Z_m-Z_0\, \ge\, a}\, \le\, \exp\left(\frac{-a^2}{24t\|\psi\|_2^2\, +\, 6a\psi_{\ourmax}}\right)\, ,
\]
as required.
%\begin{align*}
%\pr{f(G_m)\, -\,\Ex{f(G_m)}\, \ge \, a}\,  & =\, \pr{Z_m-Z_0\, \ge\, a}\\
%& \le\, \exp\left(\frac{-a^2}{24t\|\psi\|_2^2\, +\, 6a\psi_{\ourmax}}\right)\, ,
%\end{align*}

All that remains is to justify~\eqr{needinc}, i.e., we must prove that for every possible sequence $(e_1,\dots,e_i)$ we have that $|Z_i\, -\, Z_{i-1}|\, \le\, \psi(e_i)\, +\, 2\bar{\psi}$.  The conditional expectations $Z_i$ and $Z_{i-1}$ may be expressed explicitly as sums over the choices of the edges up to $e_m$,
\[
Z_i\, =\, \frac{1}{(N-i)_{m-i}}\sum_{e_{i+1},\dots, e_m} f(G_i\cup\{e_{i+1},\dots ,e_m\}) 
\]
and
\[
Z_{i-1}\, =\, \frac{1}{(N-i+1)_{m-i+1}}\sum_{e'_i,e'_{i+1},\dots, e'_m} f(G_{i-1}\cup\{e'_i,e'_{i+1},\dots ,e'_m\}) \, .
\]
Note that the sums are over sequences of distinct edges of $K_n$ disjoint from those already selected.  We would like to pair up the terms in such a way that the graphs $E(G_i)\cup\{e_{i+1},\dots ,e_m\}$ and $E(G_{i-1})\cup\{e'_i,e'_{i+1},\dots ,e'_m\}$ are either equal or adjacent.  The obvious problem is that the sums do not even have the same number of terms.  We introduce a dummy edge $g$ in the first sum, which may be any edge of $E(K_n)\setminus \{e_1,\dots ,e_{i-1}\}$.  We have
\[
Z_i\, =\, \frac{1}{(N-i+1)_{m-i+1}}\sum_{e_{i+1},\dots, e_m; g} f(G_i\cup\{e_{i+1},\dots ,e_m\}) \, .
\]
We may now pair up terms of the two summations on a one-to-one basis.  Let $\S$ be the sequences of edges allowed in the above summation and $\T$ the sequences allowed in the summation for $Z_{i-1}$.  We define a bijection $\phi:\S \to \T$ as follows: 
\[
\phi(e_{i+1},\dots, e_m; g)\, :=\, 
\begin{cases} 
(e_i,e_{i+1},\dots ,e_m)&\text{if } g=e_i\\
(e_{i+1},\dots, e_{j-1},e_i,e_{j+1},\dots ,e_m)\qquad &\text{if } g\in \{e_{i+1},\dots ,e_m\}\\
(g,e_{i+1},\dots ,e_m)&\text{if } g\not\in \{e_i,\dots ,e_m\}\, .
\end{cases}
\]
Note that in all cases the graphs $E(G_{i-1})\cup\{e_i,\dots ,e_m\}$ and $E(G_{i-1})\cup \phi(e_{i+1},\dots ,e_m;g)$ are either equal or adjacent (with symmetric difference $\{e_i,g\}$).  Now, by the triangle inequality,
\begin{align*}
\big|&Z_i\, -\, Z_{i-1}\big|\phantom{\Big|}\\ & \le\, \frac{1}{(N-i+1)_{m-i+1}}\sum_{e_{i+1},\dots,e_m;g} \big| f(G_{i-1}\cup\{e_i,\dots ,e_m\})
\, -\, f(G_{i-1}\cup \phi(e_{i+1},\dots ,e_m;g))\big|\\
& \le \, \frac{1}{(N-i+1)_{m-i+1}}\sum_{e_{i+1},\dots,e_m;g} \psi(e_i)+\psi(g)\\
& \le\, \psi(e_i)\, +\, \frac{1}{N-i+1}\sum_{g\in E(K_n)}\psi(g)\\
& \le \, \psi(e_i)\, +\, 2\bar{\psi}.
\end{align*}
This completes the proof.
\end{proof}

Let us also mention a fairly trivial bound, which is simply a consequence of the triangle inequality and union bound:
\eq{triun}
\pr{X+Y\ge \alpha+\beta}\, \le\, \pr{X\ge \alpha}+\pr{Y\ge \beta}
\eqe
for any  non-negative random variables $X,Y$ and $\alpha,\beta\in \RR$.

\subsection{The martingale representations of the deviations of cherries and triangles}\label{sec:Mart}

In this subsection we recall the martingale expressions for the subgraph count deviations given in~\cite{GGS}.  We first state these results in general and then consider the specific cases, cherries and triangles, which interest us.

We begin by introducing some notation.  For graphs $H$ and $G$ let $N_H(G)$ denote the number of isomorphic copies of $H$ in $G$.  That is, $N_H(G)$ is the number of injective functions from $V(H)$ to $V(G)$ which preserve the edges of $H$.  This counts with multiplicity, so that, for example, $N_{K_3}(K_4)=24$.

%This counts subgraphs with multiplicity.  We use the notation $\binom{G}{H}$ for the number of copies of $H$ in $G$ without multiplicity.  So that, for example: $\binom{K_4}{K_3}=4$, while $N_{K_3}(K_4)=24$. (These counts always differ by a multiple of $|Aut(H)|$.)

If $H$ has $v(H)$ vertices and $e(H)$ edges, then the expected value of $N_H(G_m)$ is 
\[
L_H(m) \, :=\, \Ex{N_{H}(G_m)}\, =\, \frac{(n)_{v(H)}(m)_{e(H)}}{(N)_{e(H)}}\, .
\]
Thus the deviation of $N_{H}(G_m)$ from its mean is
\[
D_H(G_m)\, :=\, N_{H}(G_m)\, -\, L_{H}(m)\, .
\]
We will focus in particular on studying $D_{\owedge}(G_m)$, the deviation for cherries, and $D_{\triangle}(G_m)$, the deviation for triangles.

Let $G_0,\dots G_m$ be a realisation of the Erd\H{o}s-R\'enyi random graph process.  In a slight abuse of notation, when we write $G_m$ we may assume that $G_m$ also includes the information of the order in which the edges were added.

As in~\cite{GGS}, for each fixed graph $F$ (although we only really care about the cases that $F$ is a cherry or a triangle) we may define
\[
A_F(G_m)\, :=\, N_F(G_m)\, -\, N_F(G_{m-1})
\]
to be the number of copies of $F$ created with the addition of the $m$th edge.  We may also define a ``centralised'' version.  Let
\[
X_{F}(G_m)\, :=\, A_F(G_m)\, -\, \Ex{A_F(G_m)\, \middle|\,G_{m-1}}\, =\, N_{F}(G_m)\, -\, \Ex{N_{F}(G_m)\, \middle|\, G_{m-1}}\, .
\]
The general martingale representation for $D_H(G_m)$ given in~\cite{GGS} is
\[
D_{H}(G_m) \, =\, \sum_{i=1}^{m}\, \sum_{F\ssq E(H)}\frac{(N-m)_{e(F)}(m-i)_{e(H)-e(F)}}{(N-i)_{e(H)}}\, X_F(G_i)\, ,
\]
where we sum over all $2^{e(H)}$ subgraphs $F$ with $V(F)=V(H)$ and $E(F)\ssq E(H)$.  Since $X_F(G_i)$ is deterministically $0$ if $F$ has $0$ edges or $1$ edge, we obtain
\[
D_{\owedge}(G_m)\, =\, \sum_{i=1}^{m} \frac{(N-m)_{2}}{(N-i)_2}\, X_{\owedge}(G_i)
\]
and
\[
D_{\triangle}(G_m)\, =\, \sum_{i=1}^{m} \left[3\frac{(N-m)_{2}(m-i)}{(N-i)_3}\, X_{\owedge}(G_i)\, +\, \frac{(N-m)_3}{(N-i)_3}\, X_{\triangle}(G_i)\right]\, .
\]
We say these are martingale representations as the random variables $X_F(G_i)$ behave as martingale increments, in the sense that $\Ex{X_{F}(G_i)\middle|G_{i-1}}=0$.

\subsection{A bound on a certain summation}

The following elementary inequality will be helpful for some estimates.

\begin{prop}
\label{prop:integral}
Let $d\in \mathbb{N}$, let $r\ge d^{1/2}$ and let $\beta\ge 1$. Then, 
\[
  \sum_{x \in \ZZ^d: \norm{x}\ge r} \expb{-\beta \norm{x}^2} \, \leq\, (8\pi)^{d/2} e^{-\beta (r-d^{1/2})^2/2}\, .
\]
\end{prop}

For completeness a proof is given in Appendix~\ref{Ap:A}.

\section{Degrees in $G(n,m)$}\label{sec:degs}

Many papers have studied degree sequences in random graphs and a great deal is known about the typical behaviour of the degrees. Bollob\'as~\cite{Bbook} gives an overview which includes the results of Bollob\'as himself~\cite{B81,B82} and mentions other approaches, such as that of McKay and Wormald~\cite{MW}. See also the recent article of Liebenau and Wormald~\cite{LW} and the references therein.

Since we require results not only about the likely behaviour of degrees, but also the behaviour further into the tail (events with smaller probability), we state in this section the results we shall need.  Most of the proofs are fairly standard and so are given in the appendix, see Appendix~\ref{Ap:B}.

After the results about degrees, we include a short subsection showing what we may deduce about the increments $X_{\owedge}(G_i)$.

We shall work in $G_m\sim G(n,m)$, although we remark that similar results could be easily deduced in $G(n,p)$.  Let $d_u(G)$ be the degree of a vertex $u$ in a graph $G$ and let
\[
D_u(G_m) \, :=\, d_u(G_m) \, -\, \dfrac{2m}{n}
\]
be the deviation of the degree of $u$ (from its mean) in the graph $G_m \sim G(n,m)$.  

The expected degree of each vertex is $2m/n\, =\, t(n-1)\approx tn$.  We shall state bounds related to the number of vertices of large degree (at least $32tn$) and even larger degree (at least $2^jtn$ for each $j\ge 5$).  It is also helpful to state a bound on the maximum degree $\Delta(G_m)$ of $G_m$.  These results are fairly standard.  We also state bounds related to the sum of squares of degree deviations $\sum_{u}D_{u}(G_m)^2$.  

We now introduce some notation in order to state the bounds.  We set $\ell_b:=\log(b/etn)$.   Let $V_j$ denote the set of vertices in $G_m\sim G(n,m)$ with degree at least $2^jtn$.

\begin{restatable}{lem}{largedegs}\label{lem:largedegs} Suppose $t\ge 2n^{-1} \log{n}$ and let $b \geq 4tn$.  Then, except with probability at most $\exp(-b)$, we have
\begin{enumerate}
\item[(i)] $\Delta(G_m)\, \le\, 2b/\ell_b\phantom{\bigg|}$, and
\item[(ii)] $\displaystyle |V_j|\, \le\, \frac{b}{tnj2^{j-6}}\, \, $ for all $j\ge 5$.
\end{enumerate}
\end{restatable}

%\begin{lem}\label{lem:largedegs} Suppose $t\ge 2n^{-1} \log{n}$ and let $b \geq 4tn$.  Then, except with probability at most $\exp(-b)$, we have
%\begin{enumerate}
%\item[(i)] $\Delta(G_m)\, \le\, 2b/\ell_b$, and
%\item[(ii)] $|V_j|\le b/tnj2^{j-6}$ for all $j\ge 5$
%\end{enumerate}
%\end{lem}

With respect to the sum of squares of degree deviations we see three different types of behaviours at different points in the tail.  Let $\kappa(b,t)$ be the function defined by
\[
\kappa(b,t)\, :=\, \begin{cases} tn^2\qquad & 1 \le b<t^{1/2}n\ell\\
 b^2/\ell^2 &t^{1/2}n\ell \le b<n\ell\\
bn/\ell &n\ell \le b\le tn^2\ell \, .
\end{cases}
\]
The three ranges correspond to three different ``causes'' of a large sum of squares of degree deviations.  The first value corresponds to the contribution of order $n$ vertices with degree deviation given by the standard deviation, which is of order $t^{1/2}n^{1/2}$.  The second and third values are related to degree deviations caused by stars and hubs respectively.
We also define
\[
\kappa^{+}(b,t)\, :=\, \begin{cases} b^2/\ell_b^2 \qquad & 32tn\le b< n\ell\\
bn/\ell &n\ell \le b\le tn^2\ell \, .
\end{cases}
\]
We may now state our bounds on the sum of squares of deviations.

\begin{restatable}{prop}{sumsquare}\label{prop:sumsquare} There exists an absolute constant $C>0$ such that the following holds.  Suppose that $t\ge 2n^{-1} \log{n}$ and that $b\ge 32tn$ .  Except with probability at most $\exp(-b)$, we have
\[
\sum_{u}D_{u}(G_i)^2\, \le \, C\kappa(b,t)
\]
for all steps $i\le m$, and
\[
\sum_{u\, :\, d_u(G_m) \ge 32tn}d_{u}(G_m)^2\, \le \, C\kappa^{+}(b,t)\, .
\]
\end{restatable}

%\begin{prop}\label{prop:sumsquare} There exists an absolute constant $C>0$ such that the following holds.  Suppose that $t\ge 2n^{-1} \log{n}$ and that $b\ge tn$ .  Except with probability at most $\exp(-b)$ we have
%\[
%\sum_{u}D_{u}(G_i)^2\, \le \, C\kappa(b,t)
%\]
%and
%\[
%\sum_{u\, :\, |d_u(G_m)| \ge 32tn}D_{u}(G_i)^2\, \le \, C\kappa^{+}(b,t)
%\]
%for all steps $i\le m$.
%\end{prop}

The following lemma, which bounds the contribution of vertices whose degree is not too large, covers the most interesting part of the proof of Proposition~\ref{prop:sumsquare}.  For this reason we give the proof of Lemma~\ref{lem:sd2} here, while the rest of the proof of Proposition~\ref{prop:sumsquare} is given in Appendix~\ref{Ap:B}.

\begin{lem}\label{lem:sd2} There is an absolute constant $C>0$ such that the following holds.  Suppose that $t\ge 2n^{-1} \log{n}$ and that $b \geq n$.  Except with probability at most $\exp(-b)$, we have
\[
\sum_{u:d_u(G_i) \leq 32tn} D_u(G_i)^2\, \le\, Cbtn
\]
for all steps $i\le m$.
\end{lem}

We wish to use a martingale approach to bound the probability that the sum ${\sum_{u:d_u(G_i) \leq 32tn} D_u(G_i)^2}$ is large.  One thing which could complicate the analysis is the non-linearity of this expression.  This following trick allows us to ``linearise''.  If the sum of squares of degree deviations is large then we may encode this using a vector $\sigma$ which mimics the degree deviations, so that $\sum_{u\in V(G_m)}\sigma_u D_{u}(G_i)$ is also large.  As this function is linear, and with certain Lipschitz properties, we may bound the probability it is large using Proposition~\ref{cor:F}.  This idea, combined with a union bound, gives the desired result.

The idea of $\sigma$ is to encode the (renormalised) degree deviation of each vertex up to a factor of two.  So that the reader may think of $\sigma$ as a dyadic rounding of the renormalised degree deviation .  Values that are too small to matter are set to $0$.

\begin{proof}[Proof of Lemma~\ref{lem:sd2}]
Given a vector $\sigma\in \ZZ^{V(G)}$ with entries which are all either $0$ or $\pm 2^j$ for some $j=0,\dots ,\lfloor \frac{1}{2}\log_2{tn}\rfloor$, we define, for each $i\le m$, the event $E_{\sigma,i}$ that
\eq{Esdef}
\sum_{u\in V(G_i)}\sigma_u D_{u}(G_i)\, \ge\, 16\norm{\sigma}^2 t^{1/2}n^{1/2}\, .
\eqe
Let $S$ be the set of such vectors $\sigma$ with $\norm{\sigma}^2\ge 16b$.  We make two claims:

\textbf{Claim 1:} The event that $\sum_{u:d_u(G_i) \leq 32tn} D_u(G_i)^2 > 2^{20}btn$ is contained in the union $\bigcup_{\sigma\in S} E_{\sigma,i}$.

\textbf{Claim 2:} $\pr{E_{\sigma,i}}\, \le \, \exp(-\norm{\sigma}^2)$ for all $\sigma \in S$ and $i\le m$.

Let us observe that the lemma will follow from these two claims.  By Claim 1, and a union bound, and then Claim 2, the probability that $\sum_{u:d_u(G_i) \leq 32tn} D_u(G_i)^2 > 2^{20}btn$ occurs for some $i\le m$, is at most $\sum_{i=1}^{m}\sum_{\sigma\in S}\pr{E_{\sigma,i}}\le m\sum_{\sigma\in S}\exp(-\norm{\sigma}^2)$.  Now, by Corollary~\ref{prop:integral}, with $r=4b^{1/2}$, we have
\begin{align*}
m\sum_{\sigma\in S}\exp(-\norm{\sigma}^2) &\le \, n^2(8\pi)^{n/2} \exp\left(\frac{-(4b^{1/2}-n^{1/2})^2}{2}\right)\\
& \le \, n^2\exp(2n-9b/2)\\
&\le\, \exp(-b)\, ,
\end{align*}
as required.

\textbf{Proof of Claim I:}  If the event $\sum_{u:d_u(G_i) \leq 32tn} D_u(G_i)^2 > 2^{20}btn$ occurs for the graph $G_i$ then we shall prove that $E_{\sigma}$ occurs for the choice of $\sigma$ given below.  Let $U$ denote the set of vertices with $|D_u(G_i)|\in [32t^{1/2}n^{1/2},32tn]$.  The contribution of vertices with deviation less than $32t^{1/2}n^{1/2}$ is at most $2^{10}tn^2\le 2^{10}btn$ and so
\eq{evimp}
\sum_{u\in U} D_u(G_i)^2 \, >\, 2^{19}btn\, .
\eqe
Define $\sigma$ by $\sigma_u=0$ if $u\not\in U$ and for $j\ge 0$, we set
\[
\sigma_u\, :=\, \begin{cases} 
2^j\qquad & 2^{j+5}t^{1/2}n^{1/2} \le D_u(G_i)< \min\{2^{j+6}t^{1/2}n^{1/2}, 32tn\}\\
-2^j\qquad & 2^{j+5}t^{1/2}n^{1/2} \le -D_u(G_i)< \min\{2^{j+6}t^{1/2}n^{1/2}, 32tn\}\, .
\end{cases}
\]
It follows from the definition of $\sigma$ that 
\[
\frac{D_u(G_i)^{2}}{2^{12}tn}\, \le \sigma_u^2\, \le\,  \frac{D_u(G_i)^{2}}{2^{10}tn}\, 
\]
for all $u\in U$.  And so, by~\eqr{evimp}, we have that
\[
16b\, \le\, 2^7b\,\le\,  \|\sigma\|^2\, \le \, \frac{1}{2^{10}tn} \sum_{u\in U}D_u(G_i)^{2}\, .
\]
In particular, the lower bound confirms that $\sigma\in S$.
We also observe that $\sigma_u D_{u}(G_i)\ge D_{u}(G_i)^2/2^6t^{1/2}n^{1/2}$ for all $u\in U$, and so
\begin{align*}
\sum_{u\in V(G_i)}\sigma_u D_{u}(G_i)\, & \ge\, \frac{1}{2^6t^{1/2}n^{1/2}}\sum_{u\in U}D_{u}(G_i)^2\\
&\ge\, 16t^{1/2}n^{1/2}\|\sigma\|^2\, .
\end{align*}
This shows that $E_{\sigma}$ occurs, completing the proof of Claim 1.

\textbf{Proof of Claim 2:}  Let us fix $\sigma\in S$.  Set
\[
f_{\sigma}(G_i)\, :=\, \sum_{u\in V(G_i)}\sigma_u D_{u}(G_i)\, .
\]
So that the event $E_{\sigma}$ is the event that the function $f_{\sigma}(G_i)$ is at least $16t^{1/2}n^{1/2}\norm{\sigma}^2 $.  We use the inequality from Proposition~\ref{cor:F} to bound this probability.  In the language of that proposition our function $f_{\sigma}$ is $\psi$-Lipschitz for the function $\psi(uw)=|\sigma_u|+|\sigma_w|$.  Note that $\Ex{f_{\sigma}(G_i)}=0$ and that $\psi$ satisfies $\norm{\psi}^2\le 2n\norm{\sigma}^2$ and $\psi_{\ourmax}\le 2\sigma_{\ourmax}\le 2t^{1/2}n^{1/2}$.  Applying Proposition~\ref{cor:F} we obtain
\begin{align*}
\pr{E_{\sigma}}\, &= \, \pr{f_{\sigma}(G_i)\, \ge \,16\norm{\sigma}^2 t^{1/2}n^{1/2}}\phantom{\bigg|}\\
& \le\,  \exp\left(\frac{-256\norm{\sigma}^4 tn}{48tn\norm{\sigma}^2\, +\, 192tn\norm{\sigma}^2}\right)\\
& \le\, \exp(-\norm{\sigma}^2)\, . \phantom{\bigg|}
\end{align*}
This completes the proof.
\end{proof}

\subsection{The increments $X_{\wedge}(G_i)$}

We now deduce the following bound on increments from our results on the sum of squares of degree deviations.

\begin{lem}\label{lem:varsch} There exists an absolute constant $C$ such that the following holds.  Suppose that $2n^{-1}\log{n}\le t\le 1/2$, that $b\ge 32tn$ and that $i\le m$.  Then, except with probability at most $\exp(-b)$, we have
\[
\Ex{X_{\owedge}(G_i)^2\middle|G_{i-1}}\,  \le\, \frac{C\kappa(b,t)}{n}\, .
\]
\end{lem}

\begin{proof} Recall that $X_{\owedge}(G_i)$ is defined by $X_{\owedge}(G_i):=A_{\owedge}(G_i)-\Ex{A_{\owedge}(G_i)\middle|G_{i-1}}$, which is the difference between $A_{\owedge}(G_i)$, the number of (isomorphic copies of) paths of length $2$ created with the addition of the $i$th edge, and its expected value given $G_{i-1}$.  It follows that 
\[
\Ex{X_{\owedge}(G_i)^2\middle|G_{i-1}}\, =\, \Var(X_{\owedge}(G_i) | G_{i-1})\, =\, \Var(A_{\owedge}(G_i) |G_{i-1})\, .
\]
We now use that for any constant $c$ we have $\Var(X)\le \Ex{(X-c)^2}$, so that
\[
\Ex{X_{\owedge}(G_i)^2\middle|G_{i-1}}\, \le\, \Ex{\left(A_{\owedge}(G_i)\, -\, \frac{8(i-1)}{n}\right)^2\, \middle|\, G_{i-1}}\, .
\]
By counting the number of isomorphic copies of cherries created, we see that $A_{\owedge}(G_i)=2d_u(G_{i-1})+2d_w(G_{i-1})$ where $uw$ is the $i$th edge included in $G_i$.  Now, since $uw$ is a uniformly selected pair from $E(K_n)\setminus E(G_{i-1})$ we have
\begin{align*}
\mathbb{E}[X_{\owedge}&(G_i)^2|G_{i-1}]\,  \le\, \frac{1}{N-i+1}\,\sum_{uw\not\in E(G_{i-1})} \left(2d_u(G_{i-1})\, +\, 2d_w(G_{i-1}) \,-\, \frac{8(i-1)}{n}\right)^2\\
& \le\,  \frac{8}{N-i+1}\,\sum_{uw\not\in E(G_{i-1})} \left(d_u(G_{i-1})\, -\, \frac{2(i-1)}{n}\right)^2\, +\, \left(d_w(G_{i-1}) \,-\, \frac{2(i-1)}{n}\right)^2\\
& \le\, \frac{32}{n^2}\, \sum_{uw\not\in E(G_{i-1})} D_{u}(G_{i-1})^2\, +\, D_w(G_{i-1})^2\\
& \le\, \frac{32}{n}\, \sum_{u} D_{u}(G_{i-1})^2\, .
\end{align*}
The required inequality now follows immediately from Proposition~\ref{prop:sumsquare}.
\end{proof}

\section{Upper bound on cherry count deviation in $G(n,m)$}\label{sec:cherries}

In this section we prove the upper bound part of Theorem~\ref{thm:cherries}.  This corresponds to Proposition~\ref{prop:cherries}, stated below.  We recall that 
\[
M_{\owedge}(b,t) := \max\{\NORMAL_{\owedge}(b,t), \STAR_{\owedge}(b,t), \HUB_{\owedge}(b,t)\}
\]
where $\NORMAL_{\owedge}(b,t) := b^{1/2}tn^{3/2}, \STAR_{\owedge}(b,t):= b^21_{b < n\ell}/\ell^2$ and
$\HUB(b,t)\, :=\, bn1_{b \geq n \ell}/\ell\,$.  
The following statement corresponds to the upper bound part of Theorem~\ref{thm:cherries}.

\begin{prop}\label{prop:cherries} There exists an absolute constant $C$ such that the following holds. For all $2n^{-1}\log{n}\le t\le 1/2$ and $3 \log n \leq b \leq tn^2 \ell$ we have
\[
\pr{D_{\owedge}(G_m)\, \ge\, CM_{\owedge}(b,t)}\, \le\, \exp(-b)\, .
\]
\end{prop}

The proof is a relatively straightforward divide and conquer martingale argument.

Before giving the proof it is useful to make some observations about $M_{\owedge}(b,t)$.  It is easily checked that, for each choice of $t$, it is the normal term, $\NORMAL_{\owedge}(b,t)$, which is largest when $b\le t^{2/3}n\ell^{4/3}$, the star term, $\STAR_{\owedge}(b,t)$, is largest if $t^{2/3}n\ell^{4/3}<b\le n\ell$ and the hub term, $\HUB(b,t)\, :=\, bn1_{b \geq n \ell}/\ell$, is largest if $b>n\ell$.  One may also easily verify that
\eq{M1}
M_{\owedge}(b,t)\, \ge\, btn
\eqe
for all $b\ge 1$, and
\eq{M2}
M_{\owedge}(b,t)^2\, \ge\, btn\kappa(b,t)
\eqe
for all $b\ge 1$, where $\kappa(b,t)$ is as defined in Section~\ref{sec:degs}.  To see~\eqr{M2}, note that for $1\le b\le t^{1/2} n \ell$ we have $M_{\owedge}(b,t)^2\ge \NORMAL_{\owedge}(b,t)^2=bt^2 n^3=btn\kappa(b,t)$, and for $b \geq t^{1/2} n \ell$ we have $M_{\owedge}(b,t)=\kappa(b,t)$ and so~\eqr{M2} follows from~\eqr{M1}.  

It will also be useful to note that the inequality
\eq{M3}
M_{\owedge}(b,t)\, \ge\, \frac{1}{9}\kappa^{+}(b,t)
\eqe
holds for all $b\ge 32tn$.  We consider three ranges.  For $32tn\le b\le t^{2/3}n\ell_b^{4/3}$ this follows from the fact that
\[
M_{\owedge}(b,t)\,\ge\, \NORMAL_{\owedge}(b,t)\, =\, b^{1/2}tn^{3/2}\, \ge\, \frac{b^2}{\ell_b^2}\, =\,\kappa^{+}(b,t)
\]
in this interval.  It is easily observed that for $b\ge 32tn$ we have $\ell_b^{4/3}\ge e$, and so, for $b>\max\{32tn,t^{2/3}n\ell_b^{4/3}\}$ we have that $\ell_b\ge \ell/3$.  And so in the range $\max\{32tn,t^{2/3}n\ell_b^{4/3}\}\le b\le n\ell$ we have 
\[
M_{\owedge}(b,t)\,\ge\, \STAR_{\owedge}(b,t)\, =\, b^2/\ell^2\, \ge\, \frac{b^2}{9\ell_b^2}\, =\,\frac{1}{9}\kappa^{+}(b,t)
\]
in this interval.  Finally, for $b>n\ell$ we have $M_{\owedge}(b,t)=\HUB_{\owedge}(b,t)=\kappa^{+}(b,t)$.

We remark that, for the star and hub regimes, it would actually be possible to deduce these results more directly from Proposition~\ref{prop:sumsquare} by expressing the cherry count in terms of the sum of squares of degree deviations.  However, we prefer to present a unified proof which works across the regimes.  One motivation for doing so is that the proof gives a good overview of proof techniques we shall use again later in the triangle case.

We recall the martingale expression for $D_{\owedge}(G_m)$ is given by 
\[
D_{\owedge}(G_m) = \sum_{i=1}^m \dfrac{(N-m)_2}{(N-i)_2} X_{\owedge}(G_i)\, .
\]
To avoid any particular increment having too large an effect we consider truncating the increments.  Let
\[
X'_{\owedge}(G_i)\, :=\,  X_{\owedge}(G_i) 1_{X_{\owedge}(G_i) \le 128tn}
\]
and 
\[
Z_{\owedge}(G_i) := X_{\owedge}(G_i) 1_{X_{\owedge}(G_i) > 128tn}\, .
\]
It follows that 
\[
D_{\owedge}(G_m)\, =\, D'_{\owedge}(G_m)\, +\, N^{*}_\owedge(G_m)
\]
where 
\eq{Dprimew}
D'_{\owedge}(G_m)\, :=\, \sum_{i=1}^{m} \dfrac{(N-m)_2}{(N-i)_2} X'_{\owedge}(G_i)
\eqe
and
\eq{Nstarw}
N^{*}_{\owedge}(G_m)\, :=\, \sum_{i=1}^{m} \dfrac{(N-m)_2}{(N-i)_2} Z_{\owedge}(G_i)\, .
\eqe
And so, to prove Proposition~\ref{prop:cherries} it suffices to prove that there exist absolute constants $C_1$ and $C_2$ such that 
\eq{ch1}
\pr{D'_{\owedge}(G_m)\, \ge \, C_1M_{\owedge}(b,t)}\, \le\, \exp(-2b)
\eqe
and
\eq{ch2}
\pr{N^{*}_{\owedge}(G_m)\, \ge\, C_2M_{\owedge}(b,t)}\, \le\, \exp(-2b)
\eqe
We may then take $C=C_1+C_2$ and the required bound follows by the triangle inequality and union bound, as in~\eqr{triun}.

\subsection{Bounding the probability that $D'_{\wedge}(G_m)\, \ge \, C_1M_{\wedge}(b,t)$} \label{sec:boundingDprimecherries}

We shall use Freedman's inequality, Lemma~\ref{lem:F}, to prove~\eqr{ch1}.

As $X'_{\owedge}(G_i)\le X_{\owedge}(G_i)$, the sequence of partial sums of the summation in~\eqr{Dprimew} form a supermartingale, with final value $D'_{\owedge}(G_m)$.    We first investigate the behaviour of the increments $Y_i^{\prime}$,  defined by
\[
Y'_i := \big((N-m)_2/(N-i)_2\big)X_{\owedge}^{\prime}(G_i)\, .
\] 
We observe immediately that, since the coefficient is at most $1$ we have the deterministic upper bound $|Y'_i| \leq 128tn$.

We may also bound the quadratic variation of the process.  Note that $(Y_i^{\prime})^2 \leq (X_{\owedge}^{\prime}(G_i))^2 \leq X_{\owedge}(G_i)^2$. Thus, Lemma~\ref{lem:varsch} applied with ``$b$'' being $\max\{4b,32tn\}$ gives a constant $C_3$ such that, for each $i\le m$, there is probability at most $\exp(-4b)$ that
\[
\Ex{(Y_i^{\prime})^2|G_{i-1}} > C_3 n^{-1} \kappa(b,t) \, .
\]
And so, except with probability at most $\exp(-3b)$, we have 
\[
\sum_{i=1}^m \Ex{(Y_i^{\prime})^2|G_{i-1}} \leq C_3 \kappa(b,t) tn\, .
\]
Set $\beta=C_3 \kappa(b,t) tn$.  We now apply Freedman's inequality, Lemma~\ref{lem:F}.  We obtain that
\[
 \pr{D_{\owedge}^{\prime}(G_m) \geq C_1 M_{\owedge}(b,t)}\, \leq\, \expb{\dfrac{- C_1^2 M_{\owedge} (b,t)^2}{2 \beta + 256 C_1 M_{\owedge}(b,t) tn}} \, +\, \exp(- 3b)\, .
 \]
It now follows immediately from~\eqr{M1} and~\eqr{M2} that the first term is at most $\exp(-3b)$, provided $C_1$ is taken sufficiently large.  As $2\exp(-3b)\le \exp(-2b)$, this completes the proof.

\subsection{Bounding the probability that $N^{*}_{\wedge}(G_m)\, \ge \, C_2M_{\wedge}(b,t)$} \label{sec:boundingNstarcherries}

To prove~\eqr{ch2}, we will bound $N^{*}_{\owedge}(G_m)$ using our results on the behaviour of degrees proved in Section~\ref{sec:degs}.  We begin by relating $N^{*}_{\owedge}(G_m)$ to the degree sequence.

It follows from the definition of $X_{\owedge}(G_i)$ that
\[
X_{\owedge}(G_i)\, \leq\, A_{\owedge}(G_i)\, \le\, 2 \sum_{v \in e_i} d_v(G_{i-1})\, .
\]
In particular, if $X_{\owedge}(G_i)>128tn$ then the vertex of larger degree has degree at least $32tn$.  And so
\[
N^{*}_{\owedge}(G_m) \, =\,  \sum_{i=1}^m X_{\owedge}(G_i) 1_{X_{\owedge}(G_i) > 128tn}\, \leq\, 4\sum_{i=1}^m \sum_{v \in e_i} d_v(G_{m})1_{d_v(G_{m}) > 32tn}\, .
\]
Note that we used that $d_v(G_{i-1}) \leq d_v(G_m)$ in the last inequality.  Now, by double counting we get
\begin{align*}
    N^{*}_{\owedge}(G_m)\, &\leq \, 4\sum_{v \in V} d_v(G_{m})1_{d_v(G_{m}) > 32tn} \sum_{e \in E(G_m)} 1_{v \in e} \\
    &= 4\sum_{v \in V} d_v(G_{m})^2 1_{d_v(G_{m}) > 32tn}\, .
\end{align*}
The quantity on the right hand side is automatically $0$ if $\Delta(G_m)\le 32tn$.  In particular, if $b\le 16tn$ then, by applying Lemma~\ref{lem:largedegs} with ``b'' being $32tn$, we obtain
\[
\pr{\Delta(G_m)>32tn}\, \le\, \exp(-32tn)\, \le\, \exp(-2b)\, .
\]
And so, in this case there is probability at most $\exp(-2b)$ that $N^{*}_{\owedge}(G_m)$ is non-zero, and~ \eqr{ch2} follows immediately.

For $b\ge 16tn$, we apply Proposition~\ref{prop:sumsquare} to obtain that, except with probability at most $\exp(-2b)$, we have 
\[
N^{*}_{\owedge}(G_m)\, \le \, C_2 \kappa^{+}(b,t)\, \le\, 9C_2 M_{\owedge}(b,t)\,  ,
\]
where the final inequality uses~\eqr{M3}.  This completes the proof of~\eqr{ch2}.

\section{Codegrees in $G(n,m)$}\label{sec:codegs}

Our proof for cherries in the previous section was reliant on some control of the degree sequence.   The quadratic variation of the random variables $X_{\triangle}(G_i)$ is a function of the sequence of codegrees.   And so, to extend the approach to the triangle count we also need to prove bounds related to the behaviour of the sequence of codegrees.  This turns out to be a much bigger challenge as we need to engage with the structure of the set of pairs with large codegrees.

After we have proved these results, we state, in Section~\ref{sec:Xtriangle}, what we can deduce about the quadratic variation of the associated increments $X_{\triangle}(G_i)$ from these bounds on the behaviour of codegrees.

For each pair of distinct vertices $u,w\in V(G_m)$ (equivalently, for each $uw\in E(K_n)$) we write $d_{uw}(G_m)$ for their codegree, i.e., $d_{uw}(G_m)=|N_{G_m}(u)\cap N_{G_m}(w)|$, the number of common neighbours of $u$ and $w$ in the graph $G_m$.  Let us also define
\[
D_{uw}(G_m)\, =\, d_{uw}(G_m)\, -\, \frac{m(m-1)(n-2)}{N(N-1)}\, ,
\]
which is the deviation of the codegree from its expected value in $G_m\sim G(n,m)$.  We shall consider two ranges of codegree deviations
\begin{itemize}
\item[(i)] from $\Theta(tn^{1/2})$ to $\Theta(t^2n)$, and
\item[(ii)] larger than $\Theta(t^2n)$.
\end{itemize}
The first range considers deviations from the order of the codegree standard deviation to the order of the mean, while the second range considers genuinely large codegree deviations.  

One observation is that a vertex $u$ with large degree is likely to have large codegree with almost all vertices.  For this reason we focus in this section on pairs $uw$ which have large codegree which \emph{cannot} be explained directly in terms of the degree of one of its vertices.  Corresponding to (i) and (ii) above,
\begin{itemize}
\item[(i)] For $k\in K_1:=\{10,\dots ,\lceil \log_{2}(tn^{1/2})+10\rceil \}$, we define
\[
F_k(G)\, :=\, \{uw\, :\, |D_{uw}(G)|\in [2^{k}tn^{1/2},2^{k+1}tn^{1/2})\, ,\, |D_u(G)|,|D_w(G)|\le 2^{k-5}n^{1/2}\}
\]
and set $f_k(G):=|F_{k}(G)|$.
\item[(ii)] For $k\in K_2:=\{10,\dots, \lfloor 2\log_2(1/t)\rfloor\}$, we define
\[
H_{k}(G)\, :=\, \{uw\, :\, d_{uw}(G)\in [2^{k}t^2n,2^{k+1}t^2n)\, ,\, d_u(G),d_w(G)\le 2^{k-5}tn\}
\]
and set $h_k(G):=|H_k(G)|$.
\end{itemize}
Note that the intervals $K_1$ and $K_2$ have been chosen so that all dyadic intervals of possible codegrees from $2^{10}tn^{1/2}$ up to $n$ are covered.  The alert reader will have noticed that we use the codegree $d_{uw}(G)$ rather than the codegree deviation $D_{uw}(G)$ in the definition of $H_k(G)$.  We do so because the codegree is easier to discuss, and in applications we will use that $|D_{uw}(G_i)|\le d_{uw}(G_i)$, whenever $d_{uw}(G_i)\ge t^2n$.

% we define
%\[
%f_k(G_i)\, :=\, |\{uw\, :\, |D_{uw}|\ge 2^{k}tn^{1/2}\, ,\, |D_u(G_i)|,|D_w(G_i)|\le 2^{k-5}n^{1/2}\}|
%\]
%for $k\in K_1:=\{5,\dots ,\lceil \log_{2}(tn^{1/2})+10\rceil \}$, and
%\[
%g_k(G_i)\, :=\, |\{uw\, :\, |d_{uw}|\ge 2^{k}t^2n\, ,\, d_u(G_i),d_w(G_i)\le 2^{k-5}tn\}|\, 
%\]
%for $k\in K_2:=\{10,\dots, \lfloor 2\log_2(1/t)\rfloor\}$.  

We shall prove the following bounds.

\begin{prop}\label{prop:codegs} There exists an absolute constant $C$ such that the following holds.  Suppose that $t\ge Cn^{-1/2}(\log{n})^{1/2}$ and that $b\ge n$.  Then, except with probability at most $\exp(-b)$, we have 
\[
f_k(G_i)\, \le \, \frac{Cb^2}{2^{4k}}\phantom{\Bigg|}  \qquad\qquad \text{for all } k\in K_1\,\text{and } i\le m\, .
\]
For $b\ge 3\log{n}$, except with probability at most $\exp(-b)$, we have 
\[
h_k(G_i)\, \le\, \frac{Cb^2}{k^2 2^{2k}t^4n^2}\phantom{\Bigg|} \qquad\qquad \text{for all } k\in K_2\,\text{and } i\le m\, .
\]
\end{prop}

For the application to controlling the triangle count, and, in particular, the quadratic variation of the associated process, we will use the following corollary.

\begin{cor}\label{cor:codegs}
There exists an absolute constant $C$ such that the following holds.  Suppose that $t\ge Cn^{-1/2}(\log{n})^{1/2}$ and that $b\ge n$.  Then, except with probability at most $\exp(-b)$, we have
\[
\sum_{uw}D_{uw}(G_i)^2\, \le \, C\max\{bt^2n^2,b^2\}
\]
for all $i\le m$.
\end{cor}

\begin{proof}  We may fix $i\le m$ and prove the result holds with failure probability at most $4\exp(-2b)$.  The corollary then follows by a union bound as $4m\exp(-2b)\le 4n^2\exp(-2b)\le \exp(-b)$.  

We will consider four types of pairs which contribute to the sum $\sum_{uw}D_{uw}(G_i)^2$.   

First, pairs with $|D_{uw}(G_i)|\le 2^{10}tn^{1/2}$ make total contribution at most $2^{20}t^2n^3\le 2^{20}bt^2n^2$, so these may be safely ignored.

Second, the contribution of pairs with $|D_{uw}(G_i)|\le 64t\max\{|D_{u}(G_i)|,|D_w(G_i)\}$, may be controlled by our bound on the sum of squares of degree deviations.  Indeed, by Proposition~\ref{prop:sumsquare}, except with probability at most $\exp(-32b)$, the total contribution of these terms is at most
\[
2^{12}t^2n\sum_{u}D_u(G_i)^2\, \le\, C_1t^2n\kappa(32b,t)\, \le\, \frac{C'_1bt^2n^2}{\ell}
\]
for some constants $C_1$ and $C'_1$.

All remaining pairs appear in $\bigcup_{k\in K_1} F_k(G_i)$ or $\bigcup_{k\in K_2}H_k(G_i)$.  Except with probability at most $\exp(-2b)$, we have
\begin{align*}
f_k(G_i)\, &\le \, \frac{C_2 b^2}{2^{4k}}\phantom{\Bigg|}   &\text{for all } &k\in K_1 ,\, \text{and}\\
h_k(G_i)\, &\le\, \frac{C_2 b^2}{k^2 2^{2k}t^4n^2}\phantom{\Bigg|} \qquad\qquad &\text{for all } &k\in K_2\, ,
\end{align*}
for some constant $C_2$.  We will assume these bounds in what follows.

Let us first consider pairs in $\bigcup_{k\in K_1} F_k(G_i)$.  If $uw\in F_k(G_i)$ then $|D_{uw}(G_i)|\le  2^{k+1}tn^{1/2}$ and so we may bound the total contribution of pairs in $F_{k}(G_i)$ by $2^{2k+2}t^2 n f_k(G_i)$.

We consider two cases.  If $2^{2k}\le b/n$ then we use the trivial bound $f_k(G_i)\le n^2$, so that the total contribution of pairs in $F_{k}(G_i)$ is at most
\[
2^{2k+2}t^2nf_k(G_i)\,\le \, 2^{2k+2}t^2n^3\, .
\]
Summing these contributions gives at most $8bt^2n^2$.  For $k$ with $2^{2k}> b/n$ we use the above bound on $f_k(G_i)$, which gives us that the total contribution of pairs in $F_{k}(G_i)$ is at most 
\[
2^{2k+2}t^2nf_k(G_i)\,\le \, \frac{4C_2 b^2t^2n}{2^{2k}}\, .
\]
Summing these contributions gives at most $8C_2bt^2n^2$.  And so
\[
\sum_{k\in K_1}\sum_{uw\in F_k(G_i)}D_{uw}(G_i)^2\, \le\, 8(C_2+1)bt^2n^2\, .
\]

We now consider pairs in $\bigcup_{k\in K_2} H_k(G_i)$.  If $uw\in H_k(G_i)$ then $D_{uw}(G_i)\le d_{uw}(G_i)\le 2^{k+1}t^2n$ and so we may bound the total contribution of pairs in $H_{k}(G_i)$ by $2^{2k+2}t^4n^2 h_k(G_i)$.  Summing these contributions, over $k\in K_2$, gives at most
\[
\sum_{k\in K_2} 2^{2k+2}t^4n^2 h_k(G_i)\, \le\, \sum_{k\ge 10} \frac{4C_2 b^2}{k^2}\, \le\, C_2 b^2\, .
\]

Summing all the contributions of all types gives that the summation is at most 
\[
\big(2^{10}+C_1+9C_2+8\big)\, \max\{bt^2n^2,b^2\}\, ,
\]
as required.
\end{proof}

In the remainder of the section we prove Proposition~\ref{prop:codegs}.  The proof is somewhat different for $f_k(G_i)$ and $h_k(G_i)$, so we dedicate a subsection to each.  In fact, it is better to start with $h_k(G_i)$.

\subsection{Controlling $h_k(G_i)$}\label{sec:h}

We shall now prove the part of Proposition~\ref{prop:codegs} related to $h_k(G_i)$.  

%In fact it will be sufficient to prove the following.  

%\begin{prop}\label{prop:h} There exists an absolute constant $C$ such that the following holds.  Suppose that $t\ge Cn^{-1/2}(\log{n})^{1/2}$, that $p\in (0,t)$, that $b\ge 4tn$ and that $k\in K_2$.  Let $G\sim G(n,p)$.  Then, except with probability at most $\exp(-b)$, we have 
%\[
%h_k(G)\, \le\, \frac{Cb^2}{k^2 2^{2k}t^4n^2}\phantom{\Bigg|}
%\]
%\end{prop}
%
%Let us first see that this implies the required result.  Fix $i\le m$ and $k\in K_2$.  We apply Proposition~\ref{prop:h} with ``$b$'' being twice the value of $b$ from Proposition~\ref{prop:codegs} and $p=s=i/N$.  As the event $Bin(N,s)=i$ has probability at least $n^{-2}$, we have $\pr{G_i\in \P}\le n^2\pr{G\in \P}$, where $G\sim G(n,s)$.  We obtain that, except with probability at most $n^2\exp(-2b)$, we have $h_k(G_i)\, \le\, 4Cb^2/k^2 2^{2k}t^4n^2$.

%The probability the required bound fails for some $i\le m$ and $k\in K_2$ is therefore at most $m|K_2|n^2\exp(-2b)\le n^5\exp(-2b)\le \exp(-b)$, as required.

%\begin{prop}\label{prop:h} There exists an absolute constant $C$ such that the following holds.  Suppose that $t\ge Cn^{-1/2}(\log{n})^{1/2}$, that $b\ge 4tn$, that $i\le m$, and that $k\in K_2$.  Then, except with probability at most $\exp(-b)$, we have 
%\[
%h_k(G_i)\, \le\, \frac{Cb^2}{k^2 2^{2k}t^4n^2}\phantom{\Bigg|}
%\]
%\end{prop}

We remark that the bound is trivial if $k2^k\le b/t^2n^2$, as $h_k(G)\le n^2$.  So let us fix $i\le m$ and $k\in K_2$ such that $k2^k> b/t^2n^2$. 

We may think of $H_k(G)$ as an auxiliary graph associated with the graph $G$.  Rather than work directly with $H_k(G)$ it is useful to partition $H_k(G)$ depending on the degrees of $u,w$ in the pair $uw$. 
We define
\[
H_{k,0}(G)\, :=\, \{uw\, :\, d_{uw}(G)\in [2^{k}t^2n,2^{k+1}t^2n)\, ,\, \max\{d_u(G),d_w(G)\}\le 2^{k/2}tn\}\phantom{\bigg|}
\]
and set $h_{k,0}(G):=|H_{k,0}(G)|$.  And for $k/2<j\le \min\{k-6,\log_2(1/t)\}$, we define
\[
H_{k,j}(G)\, :=\, \{uw\, :\, d_{uw}(G)\in [2^{k}t^2n,2^{k+1}t^2n)\, ,\, \max\{d_u(G),d_w(G)\} \in [2^jtn, 2^{j+1}tn)\}\phantom{\bigg|}
\]
and set $h_{k,j}(G):=|H_{k,j}(G)|$.  

Let $J=\{0\}\cup\{j:k/2<j\le \min\{k-6,\log_2(1/t)\}\}$ be the set of indices.  As every pair in $H_k(G)$ occurs in some $H_{k,j}(G)$ with $j\in J$, we have
\eq{hbound}
h_{k}(G)\, =\, \sum_{j\in J}h_{k,j}(G)\, .
\eqe
The following lemma bounds the size of certain structures within the auxiliary graphs $H_{k,0}(G)$ and $H_{k,j}(G)$.  In fact, the lemma will be stated in the context of a graph $G\sim G(n,p)$ for $p\in (0,t)$.  This makes the proof slightly smoother as we avoid some issues related to conditioning.

\begin{lem}\label{lem:starmatch} Suppose that $t\ge n^{-1/2}(\log{n})^{1/2}$, that $p\in (0,t)$, that $b\ge 3\log{n}$.  Let $10\le k\le 2\log_2(1/t)$ and $k/2<j\le \min\{k-6,\log_2(1/t)\}$ and let $G\sim G(n,p)$.  With probability at least $1-\exp(-b)$:
\begin{enumerate}
\item[(i)] $H_{k,0}(G)$ contains no star with degree at least $32b/k2^k t^2n$,
\item[(ii)] $H_{k,j}(G)$ contains no star with degree at least $32b/(k-j)2^k t^2n$,
\item[(iii)] $H_{k,0}(G)$ contains no matching with at least $32b/k2^k t^2n$ edges.
\end{enumerate}
\end{lem}

\begin{proof} A quick observation is that we may assume the above values, $32b/k2^k t^2n$ for example, are integer, as for non-integers the result follows by increasing $b$ until $32b/k2^k t^2n$ reaches the next integer.

We begin with the proof of (i).  For the duration of the proof of (i) let us write $j$ for $k/2$.  We first describe the event that $H_{k,0}(G)$ contains a large star as a union of certain events.  The event depends on
\begin{itemize}
\item a vertex $u\in V$, 
\item a subset $W\subseteq V\setminus \{u\}$ with $|W|=32 b/k2^k t^2n$, and
\item a subset $\Gamma \subseteq V\setminus\{u\}$ with $|\Gamma |\le 2^{j}tn$.
\end{itemize}
For each choice of the above, we set $F(u,W,\Gamma)$ to be the event that $\Gamma$ is the neighbourhood of $u$ in $G\sim G(n,p)$, and that
\[
e(W,\Gamma)\, \ge\, 2^{k}t^2n|W|\, .
\]
Observe that the event described in (i) may only fail if $F(u,W,\Gamma)$ occurs for some trio $u,W,\Gamma$.  Simply take $u$ to be the centre of the star, $\Gamma$ to be its neighbourhood, and $W$ be $32 b/k2^k t^2n$ neighbours of $u$ in $H_{k,0}(G)$.  

For a fixed choice of $W$ and $\Gamma$ we have that $\Ex{e(W,\Gamma)}\le 2^{j}t^2n|W|=:\nu$.  The event $F(u,W,\Gamma)$ implies that $e(W,\Gamma)\ge 2^{k-j}\nu$.  And so, by~\eqr{h6}, we have that
\begin{align*}
\pr{F(u,W,\Gamma)|N(u)=\Gamma}\, &\le\, 2\exp\left(-(k-j)2^{k-j-3}\nu\right)\phantom{\Big|}\\ 
&=\, 2\exp(-(k-j)2^{k-3}t^2n|W|)\phantom{\Big|}\\
&\le \, 2\exp(-2b)\, .\phantom{\Big|}
\end{align*}
For the last line, we used the definition of $|W|$ and the fact that $j=k/2$.  We will now complete the proof with a union bound.  Note that there are $n$ choices of the vertex $u$ and at most 
\[
n^{|W|}\,=\, \exp\left(\frac{32b\log{n}}{k2^k t^2n}\right)\,\le\, \exp\left(\frac{32 b}{k2^{k}}\right)\,\le\, \exp(b/2)
\]
choices of $W$ (we used $k\ge 10$).  And so
\begin{align*}
\pr{\bigcup_{u,W,\Gamma}F(u,W,\Gamma)}\, &\le\, \sum_{u,W,\Gamma}\pr{N(u)=\Gamma}\pr{F(u,W,\Gamma)|N(u)=\Gamma}\\
&\le \, 2\exp(-2b) \sum_{u,W}\sum_{\Gamma}\pr{N(u)=\Gamma}\phantom{\Big|}\\
&\le \, 2n\exp(-3b/2)\phantom{\Big|}\\
&\le\, \exp(-b)\, ,\phantom{\Big|}
\end{align*}
as required (we used $b\ge 3\log{n}$).

The bound for (ii) works in exactly the same way.  The only difference is that $j$ is now taken in the interval $(k/2, \min\{k-6,\log_2(1/t)\}]$, and the $(k-j)$ term cancels with the equivalent term in $|W|$.

For (iii), the argument is similar, though a little more involved as we now need to consider a set of disjoint pairs (a matching) $u_1w_1,\dots u_{f}w_{f}$, where $f:=32b/k2^k t^2n$.  Let us write $W$ for $\{w_1,\dots ,w_f\}$.  We also need to consider a sequence of sets $\Gamma_1,\dots ,\Gamma_f$ rather than a single set $\Gamma$.  Let $F(u_1,\dots ,u_f,w_1,\dots ,w_f,\Gamma_1,\dots ,\Gamma_f)$ be the event that $N(u_g)=\Gamma_g$ for all $g=1,\dots f$, and
\[
\sum_{g=1}^{f}|N(w_g)\cap \Gamma_g|\, \ge\, 2^{k}t^2nf\, .
\]
It is clear that the event that $h_{k,0}(G)\ge f$ is contained in the union of the events $F(u_1,\dots ,u_f,w_1,\dots ,w_f,\Gamma_1,\dots ,\Gamma_f)$, where the union is taken over matchings $u_1w_1,\dots ,u_f w_f$ and sets $\Gamma_1,\dots ,\Gamma_f$ each with cardinality at most $2^{k/2}tn$.  

We abbreviate $F(u_1,\dots ,u_f,w_1,\dots ,w_f,\Gamma_1,\dots ,\Gamma_f)$ to $F(\mathbf{u},\mathbf{w},\mathbf{\Gamma})$ and write $N(\mathbf{u})=\mathbf{\Gamma}$ for the event that $N(u_g)=\Gamma_g$ for all $g=1,\dots ,f$.

For each such choice, $\sum_{g=1}^{f}|N(w_g)\cap \Gamma_g|$ is a sum of indicator functions representing whether certain edges are included in $G$.  As edges may occur twice, it might not necessarily be binomial, but can be written as a sum of two binomials.  And so, just as in the ``star'' case above, by~\eqr{h6} we have
\[
\pr{F(\mathbf{u},\mathbf{w},\mathbf{\Gamma})|N(\mathbf{u})=\mathbf{\Gamma}}\, \le\, 2\exp(-2b)\, .
\]
The union bound argument is also similar to that given in (i).  There are at most 
\[
n^{2f}\, \le\, \exp\left(\frac{32b\log{n}}{k2^k t^2n}\right)\, \le\,  \exp(b/2)
\]
choices of the matching (we used $t^2n\ge \log{n}$ and $k\ge 10$).  And so
\begin{align*}
\pr{\bigcup_{\mathbf{u},\mathbf{w},\mathbf{\Gamma}}F(\mathbf{u},\mathbf{w},\mathbf{\Gamma})}\, &\le\, \sum_{\mathbf{u},\mathbf{w},\mathbf{\Gamma}}\pr{N(\mathbf{u})=\mathbf{\Gamma}}\pr{F(\mathbf{u},\mathbf{w},\mathbf{\Gamma})|N(\mathbf{u})=\mathbf{\Gamma}}\phantom{\Big|}\\
&\le \, 2\exp(-2b)\sum_{\mathbf{u},\mathbf{w}}\sum_{\mathbf{\Gamma}}\pr{N(\mathbf{u})=\mathbf{\Gamma}}\phantom{\Big|}\\
&\le\, 2\exp(-3b/2)\phantom{\Big|}\\
&\le\, \exp(-b)\, ,\phantom{\Big|}
\end{align*}
as required.
\end{proof}

It is now quite straightforward to deduce the part of Proposition~\ref{prop:codegs} related to $h_k(G_i)$.

We first bound $h_{k,0}(G_i)$.  By applying Lemma~\ref{lem:starmatch} with value of ``$b$'' being $2b$, and with $p=s=i/N$, we have that there is a constant $C_1$ such that, in $G\sim G(n,s)$, there is probability at most $\exp(-2b)$ that $H_{k,0}(G)$ contains a star or matching with $C_1b/k2^k t^2n$ edges.  There is probability at least $n^{-2}$ that $Bin(N,s)=i$ and so this event has probability at most $n^2\exp(-2b)$ in $G_i$.  And so, except with probability at most $n^2\exp(-2b)$ the graph $H_{k,0}(G_i)$ contains no star or matching of size $C_1b/k2^k t^2n$.  In this case, degrees in $H_{k,0}(G_i)$ are bounded by $C_1b/k2^k t^2n$ and at most $2C_1b/k2^k t^2n$ vertices have positive degree, and so it follows from the handshaking lemma that
\[
h_{k,0}(G_i)\, \le\, \frac{C_1^2 b^2}{k^2 2^{2k}t^4n^2}\, .
\]

We now bound $h_{k,j}(G_i)$, for $k/2<j\le \min\{k-6,\log_2(1/t)\}$.  Arguing as above, and using part (ii) of Lemma~\ref{lem:starmatch}, except with probability at most $n^2\exp(-2b)$, we have that $H_{k,j}(G_i)$ contains no star of size $2C_1b/(k-j)2^k t^2n$.  Also, every edge of $H_{k,j}(G_i)$ is incident to $V_j$, and, except with probability at most $\exp(-2b)$, we have $|V_j|\le 2^7 b/tnj2^{j}$, by Lemma~\ref{lem:largedegs}.  (This assumes $b\ge 2tn$, but if $b<2tn$ one could instead use the first part of Lemma~\ref{lem:largedegs} to obtain $\Delta(G_i)\le 4tn$, and so, in particular, $V_j=\emptyset$ for all $j\ge 5$.)  It follows that, except with probability at most $(n^2+1)\exp(-2b)$, we have 
\[
h_{k,j}(G_i)\, \le\, |V_j|\, \Delta(H_{k,j}(G_i))\, \le\, \left(\frac{2^7b}{tnj2^{j}}\right)\left(\frac{2C_1b}{(k-j)2^k t^2n}\right)\, =\, \frac{2^8 C_1 b^2}{j(k-j)2^{k+j}t^{3}n^2}\, .
\]

Now, using these bounds and~\eqr{hbound}, except with probability at most $n^2\log_2(1/t)\exp(-2b)\le \exp(-b)$, we have
\begin{align*}
h_{k}(G)\, &=\, \sum_{j\in J}h_{k,j}(G)\\
&\le\, \frac{2C_1^2 b^2}{k^2 2^{2k}t^4n^2}\, +\, \sum_{j=\lceil k/2\rceil}^{\min\{k-6,\log_2(1/t)\}}\frac{2^8 C_1 b^2}{j(k-j)2^{k+j}t^{3}n^2}\, .
\end{align*}
The final summation decays at least as fast as a geometric sequence with ratio $3/4$, and so is at most $4$ times its first term.  So, except with probability at most $\exp(-b)$, we have,
\[
h_{k}(G)\, \le\,  \frac{2C_1^2 b^2}{k^2 2^{2k}t^4n^2}\, +\, \frac{2^{10} C_1 b^2}{\lfloor k/2\rfloor \lceil k/2\rceil 2^{3k/2}t^{3}n^2}\, .
\]
As the range of $k$ goes up to $2\log_2(1/t)$, we have $t^{-1}\ge 2^{k/2}$, and so we have the bound
\[
h_{k}(G)\, \le\,  \frac{2^{13}C_1^2 b^2}{k^2 2^{2k}t^4n^2}\, ,
\]
as required.

\subsection{Controlling $f_k(G_i)$}\label{sec:f}

We shall now prove the part of Proposition~\ref{prop:codegs} related to $f_k(G_i)$.  

The approach we used to control $h_{k,0}(G_i)$ in the previous subsection may also be applied here.  That approach, bounding the size of stars and matchings in the relevant graph ($H_{k,0}(G)$ in that case, and $F_{k}(G)$ now) will give the required result for almost all values of $k\in K_1$.  

However, when $k$ is small (when $2^k\le 2^8\sqrt{\log{n}}$) the approach does not work.  The problem is that the union bound argument breaks down when considering a large matching, as there are $\exp(\Theta(f\log{n}))$ ways to choose a matching with $f$ edges, and, as we shall see below, the probability bound for a fixed matching is $\exp(-\Theta(2^{2k}f))$.  For this reason we give a slightly different argument for $k\in K_1^{-} :=\{k\in K_1:2^k\le 2^8\sqrt{\log{n}} \}$.  

We first give the argument for $k$ such that $2^k\ge 2^8\sqrt{\log{n}}$.  As in Section~\ref{sec:h} we shall bound the size of stars and matchings.  Again we state the result in $G(n,p)$.

\begin{lem}\label{lem:starmatch2} There exists an absolute constant $C$ such that the following holds.  Suppose that $t\ge Cn^{-1/2}(\log{n})^{1/2}$, that $p\in (0,t)$, and that $b\ge 4tn$.  Let $k\in K_1$ be such that $2^k\ge 2^8\sqrt{\log{n}}$ and let $G\sim G(n,p)$.  With probability at least $1-\exp(-b)$:
\begin{enumerate}
\item[(i)] $F_{k}(G)$ contains no star with degree $Cb/2^{2k}$,
\item[(ii)] $F_{k}(G)$ contains no matching with $Cb/2^{2k}$ edges.
\end{enumerate}
\end{lem} 

\begin{proof} We begin with the proof of (i).  We include $C$ in our argument and fix its value later.

We first describe the event that $F_{k}(G)$ contains a large star as a union of certain events.  The event depends on
\begin{itemize}
\item a vertex $u\in V$, 
\item a subset $W\subseteq V\setminus \{u\}$ with $|W|=Cb/2^{2k+1}$, and
\item a subset $\Gamma\subseteq V\setminus\{u\}$ with $\big||\Gamma|-pn\big|\le 2^{k-4}n^{1/2}$.
\end{itemize}
For each choice of the above, we set $E(u,W,\Gamma)$ to be the event that $\Gamma$ is the neighbourhood of $u$ in $G\sim G(n,p)$, and that
\[
\big|e(W,\Gamma)-p^2n|W| \big|\, \ge\, 2^{k}tn^{1/2}|W|\, .
\]
In fact, if (i) fails then $E(u,W,\Gamma)$ occurs for some trio $u,W,\Gamma$.  Simply take $u$ to be the centre of the star, $\Gamma$ to be its neighbourhood in $G$, and $W$ be $Cb/2^{2k+1}$ neighbours of $u$ in $F_k(G)$, chosen so that $D_{uw}(G)$ has the same sign for all $w\in W$.

For a fixed choice of $W$ and $\Gamma$ let $\mu=\Ex{e(W,\Gamma)}=p|W||\Gamma|$, and note that, by the condition on $|\Gamma|$ we have 
\[
\big|\mu\, -\, p^2n|W|\big|\, \le\, 2^{k-4}pn^{1/2}|W|\, \le\, 2^{k-4}tn^{1/2}|W|\, .
\]
And so, the event $E(u,W,\Gamma)$ implies that  
\[
\big|e(W,\Gamma)-\mu\big|\, \ge\, 2^{k-1}tn^{1/2}|W|\, .
\]
We now bound the probability of $E(u,W,\Gamma)$ using~\eqr{h5}, we may of course use that $\mu\le 2p^2n|W|\le 2t^2n|W|$, and that for $k\in K_1$, we have $2^ktn^{1/2}\le 2^{10}t^2n$.  So that
\begin{align*}
\pr{E(u,W,\Gamma)|N(u)=\Gamma}\, &\le\, \exp\left(\frac{-2^{2k-2}t^2n|W|^2}{4t^2n|W|\, +\, 2^{k}tn^{1/2}|W|}\right)\\
&\le\, \exp\big(-2^{2k-13}|W|\big)\, .\phantom{\Big|}
\end{align*}
The union bound argument is similar to that in Section~\ref{sec:h}.  There are at most $n^{|W|+1}=\exp((|W|+1)\log{n})\le\exp(2|W|\log{n})$ ways to choose the pair $u,W$.  So that
\begin{align*}
\pr{\bigcup_{u,W,\Gamma}E(u,W,\Gamma)}\, &\le\, \sum_{u,W,\Gamma}\pr{N(u)=\Gamma}\pr{E(u,W,\Gamma)\, |\, N(u)=\Gamma}\phantom{\Big|}\\
&\le \, \exp\big(-2^{2k-13}|W|\big) \sum_{u,W}\sum_{\Gamma}\pr{N(u)=\Gamma}\phantom{\Big|}\\
&\le \, \exp\big((2\log{n}-2^{2k-13})|W|\big)\phantom{\Big|}\\
&\le\, \exp\big(-2^{2k-14}|W|\big)\phantom{\Big|}\\
&\le\, \exp\big(-Cb/2^{14}\big)\, .\phantom{\Big|}
\end{align*}
Note that we used the bound, $2^k\ge 2^8\sqrt{\log{n}}$, for the penultimate inequality.  This proves the required bound provided $C\ge 2^{14}$.

The matching argument follows by a similar argument.  Just as we saw in the proof of Lemma~\ref{lem:starmatch}.
\end{proof}

For the remaining cases, with $2^{k}\le 2^8\sqrt{\log{n}}$, we use a slightly different lemma.  Instead of considering stars and matchings we now consider a union of disjoint stars.  

In the remaining cases, the graph $F_{k}(G_i)$ is potentially very dense.  The upper bound we must prove on $f_k(G_i)$ is $Cb^2/2^{4k}$, which, provided $C\ge 2^{16}$, is at least $b^2/\log{n}\ge n^2/\log{n}$.  The following lemma states that in reasonably dense graphs we may find a small number of stars which cover a relatively large number of vertices.

\begin{lem}\label{lem:multi} Let $G$ be a graph on $n$ vertices with at least $n^2/r^2$ edges.  Then there exist $r$ vertices $v_1,\dots, v_r$ of $G$ 
such that
\[
\left|\bigcup_{i=1}^{r}N(v_i)\right|\, \ge\, \frac{n}{r}\, .
\]
\end{lem}

\begin{remark} While this lemma is incredibly elementary, we didn't immediately find a reference for it.  If it hasn't already been studied, it could be of some interest to find ``best possible'' results along the same lines.  This is discussed further in the concluding remarks.
\end{remark}

\begin{proof}Clearly the result is trivial if some vertex has degree at least $n/r$ so we may assume all degrees are less than $n/r$.

Consider the digraph obtained from $G$ by replacing each edge by two oriented edges (one in each direction).  It clearly suffices to find $v_1,\dots, v_r$ such that the union of the out-neighbourhoods in $D$ of these vertices has cardinality at least $n/r$.  We note that all in and out degrees in $D$ are at most $n/r$ and that $e(D)=2e(G)\ge 2n^2/r^2$. 

We may find $v_1,\dots ,v_r$ greedily.   Let $v_1$ be a vertex of maximum out-degree and let us write $d_1$ for this degree and let $S_1:=N^{+}(v_1)$ be the set of out-neighbours of $v_1$.  We now remove from the digraph all edges into $S_1$.  In the remaining digraph we find a vertex of maximum out-degree $v_2$ with out-degree $d_2$, set $S_2=S_1\cup N^{+}(v_2)$, and remove any other edges into $S_2$.  We continue.

At step $i$, we may assume the current set $S_i$ has cardinality at most $n/r$ (else we are already done) and so the total number of removed edges so far is at most $n^2/r^2$.  It follows that at least $n^2/r^2$ edges remain and so $d_i\ge n/r^2$.

It is clear this process terminates in at most $r$ steps, as $|S_i|=d_1+\dots +d_i\ge in/r^2$.
\end{proof}

The required result bounding $f_k(G_i)$ in the remaining range, $k\in K_1^-$, will follow from the above statement and the following bound on unions of stars in $F_k(G)$.

\begin{lem}\label{lem:starmatch3} There exists an absolute constant $C$ such that the following holds.  Suppose that $t\ge Cn^{-1/2}(\log{n})^{1/2}$, that $p\in (0,t)$, and that $b\ge n$.  Let $k\in K_1$ be such that $2^k\le 2^8\sqrt{\log{n}}$ and let $G\sim G(n,p)$.  With probability at least $1-\exp(-b)$ we have that any union of $r:=2^{2k}$ stars in 
$F_{k}(G)$ contains less than $Cb/2^{2k}$ vertices.
\end{lem} 

Let us see that the bound $f_k(G_i)\le Cb^2/2^{4k}$ follows from these two lemmas.  By transferring the result of Lemma~\ref{lem:starmatch3} to $G_i\sim G(n,i)$, as we have done previously (using $n^2\exp(-2b)\le\exp(-b)$), except with probability at most $\exp(-b)$ we have that any union of $r:=2^{2k}$ stars in 
$F_{k}(G_i)$ contains less than $C_1b/2^{2k}$ vertices, for some constant $C_1\ge 1$.  We claim that this implies a bound of the form $f_k(G_i)\le Cb^2/2^{4k}$.  Indeed, suppose for contradiction that 
\[
f_k(G_i)\, >\, \frac{C_1^2 b^2}{2^{4k}}\, =\, \frac{C_1^2 n^2(b/n)^2}{2^{4k}}.
\]
It would then follow from Lemma~\ref{lem:multi} that $F_{k}(G_i)$ contains $r':=n2^{2k}/C_1 b$ vertices whose neighbourhoods cover at least $C_1b/2^{2k}$ vertices.  This is a contradiction, as $r'\le 2^{2k}$.

%
%Let $C_1\ge 1$ be at least twice the constant of Lemma~\ref{lem:starmatch3}, so that, except with probability at most $\exp(-2b)$, we have that any union of $r:=2^{2k}$ stars in $F_{k}(G)$ contains at most $C_1b/2^{2k}$ vertices.  As usual, we may transfer this result to $G_i\sim G(n,i)$.  The probability the event fails in $G_i$ is at most $n^2\exp(-2b)\le\exp(-b)$.  And in this case $f_k(G_i)\le C_1^2b^2/2^{4k}$
%
%
%
%First, let $C_1$ be twice the constant of Lemma~\ref{lem:starmatch3}.  We may assume $C_1\ge 1$.  Set $C=C_1^2$.  Now, by Lemma~\ref{lem:starmatch3}, there is probability at most $\exp(-2b)$ that $F_k(G)$ contains $r=\sqrt{log{n}}$ stars covering $C_1b/2^{2k}$ vertices.  
%
%As we have argued previously, this probability is at most $n^2$ times larger in $G_i$, and so, except with probability at most $n^2\exp(-2b)\le \exp(-b)$ we have that $F_k(G_i)$ does not contain $r=\sqrt{log{n}}$ stars covering $C_1b/2^{2k}$ vertices.  We claim that this implies the required bound on $f_k(G_i)$.
%
%Suppose for contradiction that 
%\[
%f_k(G_i)\, >\, \frac{C_1^2 b^2}{2^{4k}}\, =\, \frac{C_1^2 n^2(b/n)^2}{2^{4k}}.
%\]
%It would then follow from Lemma~\ref{lem:multi} that $F_{k}(G_i)$ contains $r':=n2^{2k}/C_1 b$ vertices whose neighbourhoods cover at least $C_1b/2^{2k}$ vertices.  This is a contradiction, as $r'\le r$.

Now that we have explained how the required bound follows from the lemmas, we complete the section with a proof of Lemma~\ref{lem:starmatch3}.

The proof is similar to those given for Lemmas~\ref{lem:starmatch} and~\ref{lem:starmatch2}.  The event considered is sort of an amalgam of the star and matching cases.  It will be useful to use the following notation.  Given two sequences of sets $\mathbf{A}=(A_1,\dots, A_r)$ and $\mathbf{B}=(B_1,\dots ,B_r)$, we define
\[
|\mathbf{A}|\, :=\, \sum_{j=1}^{r}|A_j| \qquad \text{and} \qquad \mathbf{A}\cdot \mathbf{B}\, :=\, \sum_{j=1}^{r}|A_j||B_j|\, .
\]
We recall, from Section~\ref{sec:Aux}, that for two sets of vertices $U,W$, we let $e(U,W)$ count the number of edges with multiplicity.  We extend this definition to sequences of vertex subsets. We define $\mathbf{A}\cap\mathbf{B}:=(A_1\cap B_1,\dots ,A_r\cap B_r)$ and $e(\mathbf{A},\mathbf{B}):=\sum_{j=1}^{r}e(A_j,B_j)$.  For example, in the random graph $G\sim G(n,p)$, we have
\[
\Ex{e(\mathbf{A},\mathbf{B})}\, =\, p\mathbf{A}\cdot\mathbf{B}\, -\, p|\mathbf{A}\cap\mathbf{B}|\, =\,  p\mathbf{A}\cdot\mathbf{B}\, +\, O(p|\mathbf{A}|)\, .
\]

\begin{proof}[Proof of Lemma~\ref{lem:starmatch3}] 
We include $C$ in our argument and fix its value later.  We continue to use $r:=2^{2k}$.  As before we express the event in question as a union.  This time we take a union over
\begin{itemize}
\item a sequence of vertices $\mathbf{u}=(u_1,\dots ,u_r)$, and sequences of
\item vertex subsets $\mathbf{W}=(W_1,\dots,W_r)$ such that $|\mathbf{W}|=Cb/2^{2k+1}$, and
\item vertex subsets $\mathbf{\Gamma}=(\Gamma_1,\dots ,\Gamma_r)$ with $\big||\Gamma_j|-pn\big|\le 2^{k-4}n^{1/2}$ for all $j=1,\dots, r$.
\end{itemize}
For each choice of the above, we set $E(\mathbf{u},\mathbf{W},\mathbf{\Gamma})$ to be the event that $\Gamma_j$ is the neighbourhood of $u_j$ in $G\sim G(n,p)$ for all $j=1,\dots, r$, and that
\[
\left| e(\mathbf{W},\mathbf{\Gamma})-p^2n|\mathbf{W}|\right|\, \ge\, 2^{k}tn^{1/2}|\mathbf{W}|\, .
\]
If the event of the lemma fails, i.e., if there are $r=2^{2k}$ stars in $F_k(G)$ which cover at least $Cb/2^{2k}$ vertices, then $E(\mathbf{u},\mathbf{W},\mathbf{\Gamma})$ occurs for some trio $\mathbf{u},\mathbf{W},\mathbf{\Gamma}$.  Furthermore, the sets in $\mathbf{W}$ may be taken to be disjoint and not intersect $\mathbf{u}$.  Let us justify this assertion.

First note that, in this case, at least $Cb/2^{2k+1}$ of these pairs have deviations of the same sign.  Now, simply take $\mathbf{u}$ to be the centres of the stars, $\mathbf{\Gamma}$ to be their neighbourhoods in $G$, and define $\mathbf{W}=(W_1,\dots , W_r)$ by taking the sets $W_j$ to be disjoint and with $W_j$ chosen among the neighbours of $u_j$ in $F_k(G)$ with the favoured sign.  

We now bound the probability of the event $E(\mathbf{u},\mathbf{W},\mathbf{\Gamma})$.  For a fixed choice of $\mathbf{W}$ and $\mathbf{\Gamma}$ let $\mu=\Ex{e(\mathbf{W},\mathbf{\Gamma})}=p \mathbf{W}\cdot \mathbf{\Gamma}+O(p|\mathbf{W}|)$.  By the triangle inequality, and the bounds on the $|\Gamma_j|$, we have 
\[
\big|\mu\, -\, p^2n|\mathbf{W}|\big|\, \le\, 2^{k-3}pn^{1/2}|\mathbf{W}|\, \le\, 2^{k-3}tn^{1/2}|\mathbf{W}|\, .
\]
And so, the event $E(\mathbf{u},\mathbf{W},\mathbf{\Gamma})$ implies that  
\[
\big|e(\mathbf{W},\mathbf{\Gamma})-\mu\big|\, \ge\, 2^{k-1}tn^{1/2}|\mathbf{W}|\, .
\]
We now bound the probability of $E(\mathbf{u},\mathbf{W},\mathbf{\Gamma})$ using~\eqr{h5}, we may of course use that $\mu\le 2p^2n|\mathbf{W}|\le 2t^2n|\mathbf{W}|$, and that for $k\in K_1$, we have $2^k tn^{1/2}\le 2^{10}t^2n$.  So that
\begin{align*}
\pr{E(\mathbf{u},\mathbf{W},\mathbf{\Gamma})|N(\mathbf{u})=\mathbf{\Gamma}}\, &\le\, \exp\left(\frac{-2^{2k-2}t^2n|\mathbf{W}|^2}{4t^2n|\mathbf{W}|\, +\, 2^{k}tn^{1/2}|\mathbf{W}|}\right)\\
&\le\, \exp\big(-2^{2k-13}|\mathbf{W}|\big)\, . \phantom{\Big|}
\end{align*}
The union bound argument is similar to those used before, although we need to be a little more careful this time.  There are at most $n^{r}\le \exp((\log{n})^2)$ ways to choose $\mathbf{u}$, at most $\binom{n}{|\mathbf{W}|}\le \exp\big(|\mathbf{W}|\log(en/|\mathbf{W}|)\big)$ ways to choose the elements of $\mathbf{W}$ and $r^{|\mathbf{W}|}$ ways to assign these elements to the sets $W_1,\dots ,W_r$.  The total number of choices of the pair $\mathbf{u},\mathbf{W}$ is therefore at most
\[
\exp\left((\log{n})^2\, +\, |\mathbf{W}|\log\left(\frac{enr}{|\mathbf{W}|}\right)\right)\, \le\, \exp\left((\log{n})^2\, +\, 4k|\mathbf{W}|\right)\, .
\]
So that
\begin{align*}
\pr{\bigcup_{\mathbf{u},\mathbf{W},\mathbf{\Gamma}}E(\mathbf{u},\mathbf{W},\mathbf{\Gamma})}\, &\le\, \sum_{\mathbf{u},\mathbf{W},\mathbf{\Gamma}}\phantom{\Big|}\pr{N(\mathbf{u})=\mathbf{\Gamma}}\pr{E(\mathbf{u},\mathbf{W},\mathbf{\Gamma})|N(\mathbf{u})=\mathbf{\Gamma}}\phantom{\Big|}\\
&\le \, \exp\big(-2^{2k-13}|\mathbf{W}|\big) \sum_{u,W}\sum_{\mathbf{\Gamma}}\pr{N(\mathbf{u})=\mathbf{\Gamma}}\phantom{\Big|}\\
&\le \, \exp\big((\log{n})^2\, +\, 4k|\mathbf{W}| \, -\, 2^{2k-13}|\mathbf{W}|\big)\phantom{\Big|}\\
&\le\, \exp\big(-2^{2k-14}|\mathbf{W}|\big)\phantom{\Big|}\\
&\le\, \exp\big(-Cb/2^{14}\big)\, . \phantom{\Big|}
\end{align*}
This proves the required bound provided $C\ge 2^{14}$.
\end{proof}

\subsection{The increments $X_{\triangle}(G_i)$}\label{sec:Xtriangle}

We now deduce the following bound on increments.

\begin{lem}\label{lem:varstri} There is an absolute constant $C$ such that the following holds.  Let $Cn^{-1/2}(\log{n})^{1/2}\le t\le 1/2$.  For all $n\le b\le t^2n^2$ and $i\le m$, we have
\[
\pr{\Ex{X_{\triangle}(G_i)^2|G_{i-1}}\, \ge \, Cbt^2}\, \le\, \exp(-b)\, \phantom{\Big|}
\]
and, for $b>t^2n^2$ we have
\[
\pr{\Ex{X_{\triangle}(G_i)^2|G_{i-1}}\, \ge \, \frac{Cb^2}{n^2}}\, \le\, \exp(-b)\, .\phantom{\Big|}
\]
\end{lem}

\begin{proof}
The proof relies on relating $X_{\triangle}(G_i)$ to codegree deviations.  From the proof of Lemma 5.4 of~\cite{GGS} we have
\eq{Xtricodeg}
X_{\triangle}(G_i)\,  =\, 6 D_{uw}(G_{i-1})\, -\, \Ex{6 D_{e_i}(G_{i-1})\, \mid\, G_{i-1}}\, 
\eqe
where $e_i=uw$ is pair chosen as the $i$th edge.  As $X_{\triangle}(G_i)$ has conditional mean $0$ given $G_{i-1}$, we have 
\begin{align*}
\Ex{X_{\triangle}(G_i)^2\mid G_{i-1}}\, &=\, \Var(X_{\triangle}(G_i)\mid G_{i-1})\\ 
&=\, 36\Var(D_{uw}(G_{i-1})\mid G_{i-1})\, \le\, 36\Ex{D_{uw}(G_{i-1})^2\mid G_{i-1}}\, .
\end{align*}
And so
\begin{align*}
&\Ex{X_{\triangle}(G_i)^2|G_{i-1}}\,  \le\, 36\frac{1}{N-i+1}\,\sum_{uw\not\in G_{i-1}} D_{uw}(G_{i-1})^2\\
& \le\, \frac{144}{n^2}\, \sum_{u} D_{u}(G_{i-1})^2\, .
\end{align*}
The two inequalities now follow immediately from Corollary~\ref{cor:codegs}.
\end{proof}

\section{Triangle counts in $G(n,m)$ -- upper bounds on deviations}\label{sec:UBtri}

In this section we prove upper bounds on the probability of triangle count deviations in $G(n,m)$.  That is, we prove the upper bound part of Theorem~\ref{thm:mainm}.  We state below (Proposition~\ref{prop:mainm}) a result which implies this part of Theorem~\ref{thm:mainm}.  

In order to state Proposition~\ref{prop:mainm} we begin by recalling the functions $\NORMAL(b,t)$, $\STAR(b,t)$, $\HUB(b,t)$, $\CLIQUE(b,t)$ and $M(b,t)$ defined in the introduction.  We recall that
\begin{align*}
\NORMAL(b,t)\, &:=\, b^{1/2}t^{3/2}n^{3/2}\, ,\phantom{\bigg|}\\
\STAR(b,t)\, & := \, \frac{b^2t}{\ell^2}\, 1_{b\le n\ell}\, , \phantom{\bigg|}\\
\HUB(b,t)\, &:=\, \frac{btn}{\ell}\, 1_{b\ge n\ell}\, , \phantom{\bigg|}\\
\CLIQUE(b,t)\, & :=\, \frac{b^{3/2}}{\ell^{3/2}}\quad \text{and}\phantom{\bigg|}\\
M(b,t)\, &:=\, \max\{\NORMAL(b,t),\STAR(b,t),\HUB(b,t),\CLIQUE(b,t)\}\, .\phantom{\bigg|}
\end{align*}

We shall present some lower bounds on $M(b,t)$ which will be useful later on in this section. First, we claim that
\eq{Mbt3/2n}
M(b,t)\, \ge \, bt^{3/2}n \,.
\eqe
for all $b \ge 1$. This is easily verified, as we have $\NORMAL(b,t) \ge bt^{3/2}n$ for $b \le n$, $\STAR(b,t) \ge bt^{3/2}n$ for $n < b \le n\ell$ and $\HUB(b,t) \ge bt^{3/2}n$ for $n\ell < b \le tn^2\ell$. In particular, \eqr{Mbt3/2n} implies that
\eq{Mbt2n}
M(b,t)\, \ge \, bt^2n\,.
\eqe

We also claim that
\eq{tkM}
M(b,t)\, \ge \, t\kappa(b,t) \,
\eqe
for all $b \ge 32tn$, where $\kappa(b,t)$ is as defined in Section~\ref{sec:degs}. To verify this, note that $\NORMAL(b,t) = b^{1/2}t^{3/2}n^{3/2} \ge t^2n^2 = t\kappa(b,t)$ whenever $32tn \le b < t^{1/2}n\ell$. Also, $\STAR(b,t) = t\kappa(b,t)$ whenever $t^{1/2}n\ell \le b < n\ell$ and $\HUB(b,t) = t\kappa(b,t)$ if $n\ell < b \le tn^2\ell$.

We recall from Section~\ref{sec:Mart} the martingale representation for the triangle count deviation, $D_{\triangle}(G_m)$:
\[
D_{\triangle}(G_m)\, =\, \sum_{i=1}^{m} \left[ 3 \, \frac{(N-m)_{2}(m-i)}{(N-i)_3}\, X_{\owedge}(G_i)\, +\, \frac{(N-m)_3}{(N-i)_3}\, X_{\triangle}(G_i)\right]\, .
\]

We may now state the upper bound result which we prove in this section.

\begin{prop}\label{prop:mainm}
There exists an absolute constant $C$ such that the following holds.
For all $C n^{-1/2}(\log{n})^{1/2} \le t\le 1/2$ and $3\log{n}\le b\le tn^2\ell$ we have
\[
\pr{D_{\triangle}(G_m)\, \ge \, CM(b,t)}\, \le\, \exp(-b)\, .
\]
\end{prop}

We use a divide and conquer martingale approach to prove Proposition~\ref{prop:mainm}.  The curious reader may wonder why we cannot simply use a direct martingale approach in which we use Freedman's inequality to bound the probability of such a deviation.  The problem is that during the process there may occasionally be increments which are very large.  If we let these increments occur within the martingale argument then we get weaker bounds.  For this reason we consider large increments separately.

We now introduce truncated versions $X'_{\owedge}(G_i)$ and $X'_{\triangle}(G_i)$ of the increments $X_{\owedge}(G_i)$ and $X_{\triangle}(G_i)$.  The value at which we truncate these increments will be
\[
K_{\owedge}\, =\, K_{\owedge}(b,t)\, :=\, 2^8 \, \frac{M(b,t)}{bt}
\]
and
\[
K_{\triangle}\, =\, K_{\triangle}(b,t)\, :=\,  2^{16} \, \frac{M(b,t)}{b}
\]
respectively.  We set
\[
X'_{\owedge}(G_i)\, =\, X_{\owedge}(G_i)1_{X_{\owedge}(G_i)\le K_{\owedge}}\qquad \text{and} \qquad X'_{\triangle}(G_i)\, =\, X_{\triangle}(G_i)1_{X_{\triangle}(G_i)\le K_{\triangle}}\, .\phantom{\Big|}
\]
We will also consider the random variables
\[
Z_{\owedge}(G_i)\, =\, X_{\owedge}(G_i)1_{X_{\owedge}(G_i)> K_{\owedge}}\qquad \text{and} \qquad Z_{\triangle}(G_i)\, =\, X_{\triangle}(G_i)1_{X_{\triangle}(G_i)> K_{\triangle}}\, .\phantom{\Big|}
\]
Note that $X_{\owedge}(G_i)=X'_{\owedge}(G_i)+Z_{\owedge}(G_i)$ and $X_{\triangle}(G_i)=X'_{\triangle}(G_i)+Z_{\triangle}(G_i)$.

We may now express the deviation $D_{\triangle}(G_m)$ as $D_{\triangle}(G_m)= D'_{\triangle}(G_m) +N^{*}_{\triangle}(G_m)$ where
\eq{D'def}
D'_{\triangle}(G_m)\, =\, \sum_{i=1}^{m} \left[ 3 \, \frac{(N-m)_{2}(m-i)}{(N-i)_3}\, X'_{\owedge}(G_i)\, +\, \frac{(N-m)_3}{(N-i)_3}\, X'_{\triangle}(G_i)\right]\phantom{\Big|}
\eqe
and
\[
N^{*}_{\triangle}(G_m)\, =\, \sum_{i=1}^{m} \left[ 3 \, \frac{(N-m)_{2}(m-i)}{(N-i)_3}\,Z_{\owedge}(G_i)\, +\, \frac{(N-m)_3}{(N-i)_3}\, Z_{\triangle}(G_i)\right]\, .\phantom{\Big|}
\]
Since the $Z$ random variables are non-negative and the coefficients are at most $3t$ and $1$ respectively, we have
\eq{Nstarle}
N^{*}_{\triangle}(G_m)\, \le\, \sum_{i=1}^{m} \left[ 3 t \,  Z_{\owedge}(G_i)\, +\, Z_{\triangle}(G_i)\right]\, .
\eqe

We have expressed $D_{\triangle}(G_m)$ as a sum of $D'_{\triangle}(G_m)$, which is a supermartingale (as the truncation may only reduce the value), and $N^{*}_{\triangle}(G_m)$ which counts the contribution of ``large'' increments.  In order to complete the proof of Proposition~\ref{prop:mainm} it clearly suffices to prove that there exists absolute constants $C_1$ and $C_2$ such that for all $\max\{C_1,C_2\} n^{- 1/2}(\log{n})^{1/2}\le t\le 1/2$ and $3\log{n}\le b\le tn^2\ell$ we have
\eq{D'}
\pr{D'_{\triangle}(G_m)\, \ge \, C_1 M(b,t)}\, \le\, \exp(-2b) 
\eqe
and
\eq{Nstar}
\pr{N^{*}_{\triangle}(G_m)\, \ge \, C_2 M(b,t)}\, \le\, \exp(-2b) \, .
\eqe
We prove these two statements in the following subsections.  

\subsection{Controlling $D'_{\triangle}(G_m)$}

In this subsection we prove~\eqr{D'}.  Since~\eqr{D'def} expresses $D'_{\triangle}(G_m)$ as the final value of supermartingale with increments
\[
\X'_i\, :=\, 3 \, \frac{(N-m)_{2}(m-i)}{(N-i)_3}\, X'_{\owedge}(G_i)\, +\, \frac{(N-m)_3}{(N-i)_3}\, X'_{\triangle}(G_i)\, ,\phantom{\Bigg|}
\]
we may use Freedman's inequality to bound the probability that $D'_{\triangle}(G_m)\ge C_1M(b,t)$.  All we need is some control of the increments
 $\X'_i$.  Since the coefficients are at most $3t$ and $1$ respectively, we have 
 \eq{Xprimed}
 |\X'_i|\, \le\,  3t|X'_{\owedge}(G_i)|\, +\, |X'_{\triangle}(G_i)|
 \eqe
deterministically.  As the truncation limits these absolute values, we have that $|\X'_i|\le 2^{17}M(b,t)/b$ deterministically.  
 
On the other hand, using~\eqr{Xprimed} and the fact that $(x+y)^2\le 2x^2+2y^2$, we have
\eq{X2}
|\X'_i|^2\, \le\, 18t^2\, |X'_{\owedge}(G_i)|^2\, +\, 2|X'_{\triangle}(G_i)|^2\, .\phantom{\Big|}
\eqe
This bound, together with the following lemma, will allow to control the quadratic variation of the process.  It will be proved using~\eqr{X2} and Lemmas~\ref{lem:varsch} and~\ref{lem:varstri}, which bound the conditional second moment of the increments $X_{\owedge}(G_i)$ and $X_{\triangle}(G_i)$ respectively.

\begin{lem}\label{lem:vars} There is an absolute constant $C$ such that the following holds.  Let $Cn^{-1/2}(\log{n})^{1/2}\le t\le 1/2$.  For $n\le b\le t^2n^2$ and $i\le m$, we have
\[
\pr{\Ex{(\X'_i)^2|G_{i-1}}\, \ge \, Cbt^2}\, \le\, \exp(-b)\,
\]
and, for $b>t^2n^2$ we have
\[
\pr{\Ex{(\X'_i)^2|G_{i-1}}\, \ge \, \frac{Cb^2}{n^2}}\, \le\, \exp(-b)\, .
\]
Furthermore, for all $b\ge n$,
\[
\pr{\sum_{i=1}^{m}\Ex{(\X'_i)^2|G_{i-1}}\, \ge \, C(bt^3n^2+b^2t)}\, \le\, \exp(-b)\,.
\]
\end{lem}

\begin{proof} We first remark that the bounds of Lemmas~\ref{lem:varsch} and~\ref{lem:varstri} also hold for $X'_{\owedge}(G_i)$ and $X'_{\triangle}(G_i)$ since a truncation may only reduce variance.

The first two bounds now follow immediately from~\eqr{X2}, Lemmas~\ref{lem:varsch} and~\ref{lem:varstri}, and the fact that
\[
t^2\kappa(b,t)\, \le\, \max\{bt^2, b^2/n^2\}
\]
throughout the range $b\ge n$.

The furthermore statement follows easily by a union bound.  Note that the bound is obtained as the sum of the other two, multiplied by the number of steps $m\le tn^2$.
\end{proof}

We are now ready to prove~\eqr{D'}.  We shall apply Freedman's inequality to the supermartingale~\eqr{D'def} with final value $D'_{\triangle}(G_m)$.  We have already established that the increments $\X'_i$ satisfy $|\X'_i|\le 2^{17}M(b,t)/b$ deterministically. There are now two cases depending on the value of $b$. 

For $3\log{n}\le b\le n$ we set $\beta=8C't^3n^3$, where $C'\ge 1$ is the constant of Lemma~\ref{lem:vars}, so that $\pr{\sum_{i=1}^m \Ex{(\X'_i)^2|G_{i-1}}\, \ge\, \beta}\, \le \, \exp(-2n)\, \le\, \exp(-2b)$. Note also that we may choose a constant $C_1 \ge 16$ such that $\beta\le C_1\NORMAL(b,t)^2/b\le C_1 M(b,t)^2/b$.

It follows from Freedman's inequality (Lemma~\ref{lem:F}) that
\begin{align*}
&\pr{D'_{\triangle}(G_m)\, \ge\, C_1 M(b,t)}\,  \le\, \exp\left(\frac{-C_1^2M(b,t)^2}{2\beta\, +\, 2C_1 M(b,t)(2M(b,t)/b)}\right)\, +\, \exp(-2b)\phantom{Bigg|}\\
&\qquad\qquad \qquad\,\,\, \le \, \exp\left(\frac{-C_1^2 M(b,t)^2}{2C_1 M(b,t)^2/b\, +\, 2C_1 M(b,t)(2M(b,t)/b)}\right)\, +\, \exp(-2b)\phantom{Bigg|}\\
&\qquad\qquad \qquad\,\,\,\le\,  \exp\left(\frac{-C_1 b}{6}\right)\, +\, \exp(-2b)\phantom{Bigg|}\\
&\qquad\qquad \qquad\,\,\, \le\, 2\exp(-2b)\phantom{bigg|}\\
&\qquad\qquad \qquad\,\,\,\le \, \exp(-b)\, .\phantom{bigg|} 
\end{align*}

The argument is similar in the case $b>n$.  We now set $\beta=4C'(bt^3n^2+b^2t)$, so that $\pr{\sum_{i=1}^m \Ex{(\X'_i)^2|G_{i-1}}\, \ge\, \beta}\, \le \, \exp(-2b)$.  Note also that we may choose a constant $C_1 \ge 16$ such that $\beta\le C_1 M(b,t)^2/b$ for all $b\ge n$.  The required bound now follows from Freedman's inequality by exactly the same sequence of inequalities given above.  This completes our proof of~\eqr{D'}.

\subsection{Controlling $N^{*}_{\triangle}(G_m)$}

In this subsection we complete the proof of Propostion~\ref{prop:mainm} by proving the bound~\eqr{Nstar}.  That is, we show that for some constant $C_2$ the inequality
\[
\pr{N^{*}_{\triangle}(G_m)\, \ge \, C_2 M(b,t)}\, \le\, \exp(-2b) \, 
\]
holds for all $C n^{-1/2}(\log{n})^{1/2} \le t\le 1/2$ and $3\log{n}\le b\le tn^2\ell$.  As  we saw above,~\eqr{Nstarle}, we have
\[
N^{*}_{\triangle}(G_m)\, \le\, 3\sum_{i=1}^{m}  t \,  Z_{\owedge}(G_i)\, +\, \sum_{i=1}^{m}   Z_{\triangle}(G_i)\, .
\]
And so, by~\eqr{triun}, i.e., triangle inequality and union bound, it suffices to show that there exist absolute constants $C_3$ and $C_4$ such that
\eq{Zwedge}
\pr{\sum_{i=1}^m t \, Z_{\owedge}(G_i)\, \ge\, C_3 \, M_{\triangle}(b,t)}\, \le\, \exp(-3b)
\eqe
and
\eq{Ztri}
\pr{\sum_{i=1}^m Z_{\triangle}(G_i)\, \ge\, C_4 \, M_{\triangle}(b,t)}\, \le\, 2\exp(-3b)\, .
\eqe
We prove~\eqr{Zwedge} and~\eqr{Ztri} in the following subsections.

\subsubsection{Controlling $\sum_{i=1}^m t \, Z_{\owedge}(G_i)$}\label{sec:Zwedge}

In this subsection we prove~\eqr{Zwedge}.  The proof is similar to the argument of Section~\ref{sec:boundingNstarcherries}.  Again the main objective is to relate the summation to the sum of squares of degree deviations, we may then deduce~\eqr{Zwedge} from Proposition~ \ref{prop:sumsquare}.  First, for $i \le m$, we have that
\[
X_{\owedge}(G_i)\, \le\, A_{\owedge}(G_i)\, =\, 2 \sum_{v \in e_i} d_v(G_{i-1}) \, \le \, 2\sum_{v \in e_i} d_v(G_{m})\, \le \, 4\max_{v\in e_i}d_v(G_i)\, .
\]
Thus, the event $X_{\owedge}(G_i) > K_{\owedge}$ may only occur if $d_v(G_{m}) > K_{\owedge}/4$ for one of the two vertices $v\in e$.  It follows that
\eq{ZP2expression}
Z_{\owedge}(G_i) \, = \, X_{\owedge}(G_i)1_{X_{\owedge}(G_i) > K_{\owedge}}\, \le\, 4 \sum_{v \in e_i} d_v(G_{m})1_{d_v(G_{m}) > K_{\owedge}/4} \,.
\eqe
By summing this expression over $i \le m$, we obtain the upper bound
\begin{align}
	\sum_{i=1}^m Z_{\owedge}(G_i)\, &\le \, 4\sum_{v \in V} d_v(G_{m})1_{d_v(G_{m}) > K_{\owedge}/4} \sum_{e \in E(G_m)} 1_{v \in e}\nonumber \\
	&= 4\sum_{v \in V} d_v(G_{m})^2 1_{d_v(G_{m}) > K_{\owedge}/4}\, .\label{eq:Zsumle}
\end{align}

Recall, from~\eqr{Mbt2n}, that by definition $K_{\owedge} = 2^8M(b,t)/(bt) \ge 2^8tn$.  In particular, $K_{\owedge}/4\ge 64tn$, and so the above summation above is $0$ if $\Delta(G_m) \le 64tn$.

Note that, by applying Lemma~\ref{lem:largedegs} with ``$b$'' being $96tn$ we have that
\eq{D64}
\pr{\Delta(G_m)>64tn}\, \le\, \pr{\Delta(G_m)>\frac{96tn}{\log(96/e)}}\, \le\, \exp(-96tn)\, .
\eqe
In particular, if $b\le 32tn$ then there is probability at most $\exp(-3b)$ that $\sum_{i=1}^m Z_{\owedge}(G_i)$ is non-zero, and~\eqr{Zwedge} follows in this case.

Now assume $b\ge 32tn$.  Note that $d_v(G_m) \le 2D_v(G_m)$ for all vertices $v$ such that $d_v(G_m) \ge 32tn$, and so, by~ \eqr{Zsumle}, we have
\[
\sum_{i=1}^m t Z_{\owedge}(G_i)\, \le\, 16t \sum_{v \in V} D_v(G_m)^2\,
\]
deterministically.  We now use Proposition \ref{prop:sumsquare} to bound the probability that this sum is large.  We recall,~\eqr{tkM} , that $t \kappa(b,t)\le  M(b,t)$.  And so, by Proposition \ref{prop:sumsquare}, there exists a constant $C_3$ such that
\begin{align*}
\pr{\sum_{i=1}^m \ tZ_{\owedge}(G_i)\,\ge C_3 \, M(b,t) }\, &\le\,\pr{\sum_{i=1}^m \ tZ_{\owedge}(G_i)\,\ge C_3 \, t\kappa(b,t) }\\
&\le \, \exp(-3b)\, .
\end{align*}
This completes our proof of~\eqr{Zwedge}.

\subsubsection{Controlling $\sum_{i=1}^m Z_{\triangle}(G_i)$}\label{sec:ZK3}

In this section we prove~\eqr{Ztri}.  We recall that $Z_{\triangle}(G_i)$ is defined as
\[
Z_{\triangle}(G_i)\, =\, X_{\triangle}(G_i)1_{X_{\triangle}(G_i)> K_{\triangle}}
\]
where $K_{\triangle}=2^{16}M(b,t)/ b$.  And that our objective is to prove that for some constant $C_4$, we have
\[
\pr{\sum_{i=1}^m Z_{\triangle}(G_i)\, >\, C_4 M_{\triangle}(b,t)}\, \le\, 2\exp(-3b) \,
\]
for all $C n^{-1/2}(\log{n})^{1/2} \le t\le 1/2$ and $3\log{n}\le b\le tn^2\ell$.

For each $i \le m$ we may control $X_{\triangle}(G_i)$ in terms of the codegree of the edge chosen at step $i$.  Specifically,
\[
X_{\triangle}(G_i)\, \le \, A_{\triangle}(G_i)\, = \, 6 d_{e_i}(G_{i-1}) \,.
\]
And so, for the event $X_{\triangle}(G_i) > K_{\triangle}$ to occur, it is necessary that $d_{e_i}(G_{i-1}) > K_{\triangle}/6$. It follows that
\eq{ZK3expression}
\sum_{i=1}^m Z_{\triangle}(G_i)\, \le \, 6 \sum_{i=1}^m d_{e_i}(G_{i-1}) 1_{\{d_{e_i}(G_{i-1}) > K_{\triangle}/6\}} \,. 
\eqe

We consider two different types of contribution to the sum, depending on whether the codegree is ``explained'' by degree deviations, or not.  For each edge $e_i = uv$ added during the process, which has $d_{e_i}(G_{i-1}) > K_{\triangle}/6$, we classify $e_i$ as \emph{good} if there exists a time $i-1 \le j \le m$ such that
\[
d_{e_i}(G_{j})\, \le\, 64t(d_u(G_{j}) + d_v(G_{j}))\, .
\]
Otherwise $e_i$ is \emph{bad}.

%We now have to deal with the sum of codegrees of edges which have large codegree. As we mentioned in Chapter 6, we may divide these edges into two categories: those who have its large codegree explained by a large degree of one of its vertices and those who have this large codegree explained by the behaviour of the codegree itself. We call the first type good, as we may bound their contribution to the sum by the sum of their degrees, just as we have seen in the previous section. For each $i \le m$, we say that $e_i = uv$ is \textbf{good} if $d_{e_i}(G_{i-1}) > K_{\triangle}/6$ and there exists a time $i \le j \le m$ such that
%\[
%d_{e_i}(G_{j-1})\, \le\, 64t(d_u(G_{j-1}) + d_v(G_{j-1}))\,.
%\]
%We also say that an edge $e_i$ is \textbf{bad} if $e_i$ is not good and $d_{e_i}(G_{i-1}) > K_{\triangle}/6$. 

As
\eq{badgood}
\sum_{i=1}^m Z_{\triangle}(G_i)\, \le \, 6 \sum_{i=1}^m d_{e_i}(G_{i-1})1_{e_i \text{ is good}} + 6 \sum_{i=1}^md_{e_i}(G_{i-1})1_{e_i \text{ is bad}}\, ,
\eqe
we have reduced our objective to proving the following two lemmas, modulo~\eqr{triun}, i.e., the usual union bound and triangle inequality.

\begin{lem}\label{lem:good}
	There exists an absolute constant $C$ such that the following holds. Suppose that $Cn^{-1/2}(\log{n})^{1/2}\le t\le 1/2$ and $3\log{n}\le b\le tn^2\ell$. Then, except with probability at most $\exp(-3b)$,
	\[
	\sum_{i=1}^m d_{e_i}(G_{i-1}) 1_{\{e_i \text{ is good}\}}\, \le\ CM(b,t)\,.
	\]
\end{lem}

\begin{lem}\label{lem:bad}
	There exists an absolute constant $C$ such that the following holds. Suppose that $Cn^{-1/2}(\log{n})^{1/2}\le t\le 1/2$ and $3\log{n}\le b\le tn^2\ell$. Then, except with probability at most $\exp(-3b)$,
	\[
	\sum_{i=1}^m d_{e_i}(G_{i-1}) 1_{\{e_i \text{ is bad}\}}\, \le\ CM(b,t)\, .
	\]
\end{lem}

We begin with Lemma~\ref{lem:good}.

\begin{proof}[Proof of Lemma~\ref{lem:good}]
As degrees may only increase during the process, for each good edge $e_i$, we have $K_{\triangle}/6 \le d_{e_i}(G_{i-1}) \le 64t(d_u(G_{m}) + d_v(G_{m}))$.   

In particular if $e_i$ is good then $d_v(G_m)\ge K_{\triangle}/2^{10}t$ for one of the two vertices $v\in e_i$. It follows that
	\eq{etov}
	d_{e_i}(G_{i-1}) 1_{\{e_i \text{ is good}\}}\, \le\, 128t \sum_{v \in e_i}d_v(G_m)1_{d_v(G_m) > K_{\triangle}/2^{10}t} \,.
	\eqe
	And so
	\begin{align*}
		\sum_{i=1}^m  d_{e_i}(G_{i-1}) 1_{\{e_i \text{ is good}\}} &\le \, 128t \sum_{v \in V} d_v(G_m) 1_{d_v(G_{m}) > K_{\triangle}/2^{10}t} \sum_{e \in E(G_m)}1_{v \in e}  \\
		&= 128t \sum_{v \in V} d_v(G_m)^2 1_{d_v(G_{m}) > K_{\triangle}/2^{10}t}.
	\end{align*}
	
	We also recall, from~\eqr{Mbt2n}, that $K_{\triangle}/2^{10}t = 64M(b,t)/bt \ge 64tn$. In particular, the last summation above is automatically $0$ if $\Delta(G_m) \le 64tn$.  As we saw above, in~\eqr{D64}, we have $\pr{\Delta(G_m)>64tn}\le \exp(-96tn)$.  In particular, if $b < 32tn$, then there is most probability $\exp(-3b)$ that $\sum_{i=1}^m  d_{e_i}(G_{i-1}) 1_{\{e_i \text{ is good}\}} $ is non-zero.
	
We may now assume that $b \ge 32tn$.  As $d_v(G_m) \le 2D_v(G_m)$ for all $v$ such that $d_v(G_m) \ge 64tn$, it follows that
	\[
	128t \sum_{v \in V} d_v(G_m)^2 1_{d_v(G_{m}) > K_{\triangle}/2^9t}\, \le\, 256t \sum_{v \in V} D_v(G_m)^2\, .
	\]
	
We shall now use Proposition \ref{prop:sumsquare} to bound the probability that this sum is large.  We recall~\eqr{tkM}, which states that $ t \kappa(b,t) \le CM(b,t)$.  Now, taking $C$ to be $2^{10}$ times the constant of Proposition \ref{prop:sumsquare}, we have
	\begin{align*}
	\pr{\sum_{i=1}^m  d_{e_i}(G_{i-1}) 1_{\{e_i \text{ is good}\}}\, >\, CM(b,t)}\, &\le \, \pr{\sum_{i=1}^m  d_{e_i}(G_{i-1}) 1_{\{e_i \text{ is good}\}}\, >\,  C t \kappa(b,t)}\phantom{\Bigg|}\\
	& \le\, \pr{256t \sum_{v \in V} D_v(G_m)^2\, >\, Ct\kappa(b,t)}\phantom{\Bigg|}\\
	&\le\,  \pr{\sum_{v \in V} D_v(G_m)^2\, >\, \frac{C}{256}\kappa(b,t)}\phantom{\Bigg|}\\
	&\le\, \exp(-3b)\, .\phantom{\bigg|}
	\end{align*}
\end{proof}

We shall now bound the sum of codegrees of bad edges.  We do so by showing that if $e_i$ is bad then it belongs to one of the graphs $H_k(G_m)$ of pairs with large codegree, defined in Section~\ref{sec:codegs}.  This already gives us an upper bound on the number of such edges, using the bounds proved in Section~\ref{sec:codegs}.  This is sufficient for a large range of possible codegrees.

For codegrees just a little larger than $K_{\triangle}$, up to $K_{\triangle}\ell$ in fact, a more careful analysis is required.  We use the bounds mentioned, but we also need to consider the probability that a large number of such edges are selected during the process.  This is achieved via an argument which corresponds to a coupling with a process with independent increments, which is then binomially distributed.

\begin{proof}[Proof of Lemma~\ref{lem:bad}]
Recall that we want to find an absolute constant $C$ so that for all ${Cn^{-1/2}(\log n)^{1/2}\le t\le 1/2}$ and all $3 \log n \le b \le tn^2\ell$ we have
\eq{bad}
\pr{\sum_{i=1}^m d_{e_i}(G_{i-1}) 1_{\{e_i \text{ is bad}\}}\, >\, CM(b,t)}\, \le\, \exp(-3b)\, .
\eqe

As we mentioned in the discussion before the proof, we shall consider separately contributions from codegrees between $K_{\triangle}/6$ and $K_{\triangle}\ell$, and codegrees larger than $K_{\triangle}\ell$.  We shall reduce the proof to the statement that, for some constant $C>0$, the following two claims hold:
\eq{kbig}
\pr{\sum_{i=1}^m d_{e_i}(G_{i-1}) 1_{\{e_i \text{ is bad}, d_{e_i}(G_{i-1})>K_{\triangle}\ell\}}\, >\, CM(b,t)}\, \le\, \exp(-4b)
\eqe
and
\eq{ksmall}
\pr{\sum_{i=1}^m d_{e_i}(G_{i-1}) 1_{\{e_i \text{ is bad}, d_{e_i}(G_{i-1})\in [K_{\triangle}/6,K_{\triangle}\ell)\}}\, >\, CM(b,t)}\, \le\, 2\exp(-4b) \, .
\eqe
Indeed, if both~\eqr{kbig} and~\eqr{ksmall} hold then, except with probability at most $3\exp(-4b)\le \exp(-3b)$, we clearly have 
\[
\sum_{i=1}^m d_{e_i}(G_{i-1}) 1_{\{e_i \text{ is bad}\}}\, \le\, 2CM(b,t)\, .
\]

Proving~\eqr{kbig} is relatively straightforward.  We shall bound the number pairs of large codegree using results from Section~\ref{sec:codegs}.  These bounds are sufficiently strong that even if every edge with codegree this large is selected the contribution is still at most $CM(b,t)$.  

Recall from Section \ref{sec:codegs} that we defined the sets 
\[
H_{k}(G)\, :=\, \{uw\, :\, d_{uw}(G)\in [2^{k}t^2n,2^{k+1}t^2n)\, ,\, d_u(G),d_w(G)\le 2^{k-5}tn\}
\]
for $k\in K_2:=\{10,\dots, \lfloor 2\log_2(1/t)\rfloor\}$.  And we set $h_{k}(G)=|H_k(G)|$. 

By~\eqr{Mbt2n} and the definition of $K_{\triangle}$ as $2^{16}M(b,t)/b$ we have $K_{\triangle}/6\, \ge\, 2^{10}t^2n$.  Let $k_0\ge 10$ be the maximum $k$ such that $2^{k}t^2n\le K_{\triangle}/6$, and let $k_1$ be the maximum $k$ such that $2^{k}t^2n\le K_{\triangle}\ell$.  It is also helpful to note that, as $K_{\triangle}=2^{16}M(b,t)/b\ge 2^{16}t^{3/2}n$ by~\eqr{Mbt3/2n}, we have
\eq{kell}
k_1\, \ge\, k_0\, \ge \, k_0-12\, \ge \, \log_2(t^{-1/2})\, \ge \, \ell/2\, .
\eqe

\textbf{Claim:} If $e_i$ is bad and $d_{e_i}(G_{i-1})\, \ge\, K_{\triangle}\ell$ then $e_i\in \bigcup_{k=k_1}^{2\log_{2}(1/t)}H_{k}(G_m)$.

\textbf{Proof of Claim:} First, as codegree is non-decreasing through the process
\[
d_{e_i}(G_m)\, \ge\, d_{e_i}(G_{i-1})\, \ge\, K_{\triangle}\ell \, \ge\, 2^{k_1}t^2n\, .
\]
And so $d_{e_i}(G_{m})$ is in the range of one of the sets $H_{k}(G)$ with $k\ge k_1$.

All that remains is to prove that the second condition in the definition of $H_{k}(G)$ is satisfied.  For this it suffices that $d_{e_i}(G_{m})\, >\, 64t(d_u(G_{m}) + d_v(G_{m}))$, and this follows immediately from the definition of ``bad''.  This completes the proof of the claim.

By Proposition~\ref{prop:codegs}, except with probability at most $\exp(-4b)$, we have, for some constant $C_1$, that
\[
h_k(G_i)\, \le\, \frac{C_1 b^2}{k^2 2^{2k}t^4n^2}\phantom{\Bigg|} \qquad\qquad \text{for all } k\in K_2\,\text{and } i\le m\, .
\]
Let us assume these bounds hold and prove the required bound on the sum in~\eqr{kbig}.  As $d_{e_i}(G_{i-1})\le d_{e_i}(G_{m})\le 2^{k+1}t^2n$ for edges $e_i\in H_k(G_{m})$ we have, except with probability at most $\exp(-4b)$,
\begin{align}
\sum_{i=1}^m d_{e_i}(G_{i-1}) 1_{\{e_i \text{ is bad}, d_{e_i}(G_{i-1})>K_{\triangle}\ell\}} \, &\le\, \sum_{k=k_1}^{2\log_2(1/t)} 2^{k+1}t^2n  \frac{C_1 b^2}{k^2 2^{2k}t^4n^2}\phantom{\Bigg|}\nonumber \\ 
&\le\, \sum_{k=k_1}^{2\log_2(1/t)}   \frac{2C_1 b^2}{k^2 2^{k}t^2n}\phantom{\Bigg|}\nonumber\\
&\le\, \frac{4C_1 b^2}{k_1^2 2^{k_1}t^2n}\nonumber\\
&\le\, \frac{8C_1 b ^2}{k_1^2 K_{\triangle}\ell}\nonumber\\
&\le\, \frac{2^{21}C_1 b^3}{\ell^3 M(b,t)}\label{eq:k1was}\\
&\le\, 2^{21} C_1 M(b,t)\, .\phantom{\Big|} \nonumber
\end{align}
Note that the penultimate inequality used the definition of $K_{\triangle}$ and the inequality $k_1\ge \ell/2$ observed above.  The final inequality used that $M(b,t)^2\ge \CLIQUE(b,t)^2 =b^3/\ell^3$.  This completes the proof of~\eqr{kbig}.

We now consider~\eqr{ksmall}.  This time a more subtle argument is needed.  We will still use bounds on the cardinalities $h_k(G_{i-1})$, but we must also consider how many of the edges of large codegree are actually selected during the process.

Let $H^{*}(G_i):=\bigcup_{k=k_0}^{k_1}H_{k}(G_i)$ and note that, if $e_i$ is bad and $d_{e_i}(G_{i-1})\in [K_{\triangle}/6,K_{\triangle}\ell)$ then $e_i\in H^*(G_{i-1})$.  Let $a^*(G_m)$ be the number of edges $e_i$ selected in the process such that $e_i\in H^{*}(G_{i-1})$.  As 
\begin{align*}
\sum_{i=1}^m d_{e_i}(G_{i-1}) 1_{\{e_i \text{ is bad}, d_{e_i}(G_{i-1})\in [K_{\triangle}/6,K_{\triangle}\ell)\}}\, &\le\, K_{\triangle}\ell a^{*}(G_m)\\
& \le \, \frac{2^{16}M(b,t)\ell a^{*}(G_m)}{b}\, ,
\end{align*}
it suffices to prove, for some constant $C$ that
\[
\pr{a^{*}(G_m)\, >\, Cb/\ell}\, \le\, 2\exp(-4b)\, .
\]

Let $E$ be the event that
\[
h_k(G_i)\, \le\, \frac{C_1 b^2}{k^2 2^{2k}t^4n^2}\phantom{\Bigg|} \qquad\qquad \text{for all } k_0\le k\le k_1 \,\text{and } i\le m\, .
\]
By Proposition~\ref{prop:codegs}, $\pr{E^c}\le \exp(-4b)$.  We now prove that, if $C$ is sufficiently large, then $\pr{a^{*}(G_m)\, >\, Cb/\ell\, \text{and}\, E}\le \exp(-4b)$.  This will complete the proof of the lemma.

On the event $E$, at each step we have
\begin{align}
|H^*(G_{i-1})|\, & \le\, \sum_{k=k_0}^{k_1} \frac{C_1 b^2}{k^2 2^{2k}t^4n^2} \nonumber \\
& \le\, \frac{2C_1 b^2}{k_{0}^2 2^{2k_0}t^4n^2}\nonumber \\
& \le\, \frac{C_2 b^2}{\ell^2 K_{\triangle}^2}\nonumber \\
&\le\, \frac{C_2 b^4}{\ell^2 M(b,t)^2}\, , \label{eq:Hstar}
\end{align}
for some constant $C_2$, where we used the inequality $k_0\ge \ell/2$ from~\eqr{kell}.  And so, the event ``$a^{*}(G_m)\, >\, Cb/\ell\, \text{and}\, E$'' implies that 
\begin{enumerate}
\item[(i)] at each step there are at most $C_2b^4/\ell^2M(b,t)^2$ pairs in $H^{*}(G_{i-1})$, and
\item[(ii)] in at least $Cb/\ell$ steps we select $e_i\in H^{*}(G_{i-1})$.
\end{enumerate}

By (i), we have that, the conditional probability (given $G_{i-1}$) that $e_i$ is selected in $H^{*}(G_{i-1})$ is always at most 
\[
\frac{C_2 b^4}{\ell^2M(b,t)^2 (N-i+1)}\, \le\, \frac{2C_2 b^4}{\ell^2M(b,t)^2 N}\, .
\]
Therefore, given any sequence $i_1<i_2<\dots <i_j$, with $j=Cb/\ell$, the probability that $e_i$ is selected in $H^{*}(G_{i-1})$ at all these steps is at most
\[
\left(\frac{2C_2 b^4}{\ell^2M(b,t)^2 N}\right)^j\, .
\] 
By a union bound, over at most $\binom{m}{j}\le (em/j)^j$ choices of the sequence, it follows that
\[
\pr{a^{*}(G_m)\, >\, Cb/\ell\, \text{and}\, E}\, \le\, \left(\frac{2emC_2 b^4}{j \ell^2M(b,t)^2 N}\right)^j\, .
\]
We now use that $m/N=t$, that $M(b,t)^2\ge \CLIQUE(b,t)^2=b^3/\ell^3$, and the value of $j=Cb/\ell$ to deduce that
\[
\pr{a^{*}(G_m)\, >\, Cb/\ell\, \text{and}\, E}\, \le\, \left(\frac{2eC_2 t \ell^2}{C}\right)^{Cb/\ell}\, .
\]
If $C$ is taken sufficiently large then the expression in brackets is at most $t^{1/2}=\exp(-\ell/2)$, and so 
\[
\pr{a^{*}(G_m)\, >\, Cb/\ell\, \text{and}\, E}\, \le\, \exp(-Cb/2)\, \le \, \exp(-4b)\, .
\]
\end{proof}

\section{Triangle counts in $G(n,m)$ -- lower bounds on deviations}\label{sec:LBtri}

We now prove the lower bound part of Theorem~\ref{thm:mainm}.  That is, in each of the four regimes, we prove that a deviation of the corresponding order (for example, at least $cN(b,t)$ in the normal regime) has probability at least $\exp(-b)$.  It is perhaps surprising that this is most challenging in the normal regime.

In the normal regime (Section~\ref{sec:normalLB}) we prove this result using the converse Freedman inequality (Lemma~\ref{lem:CF}).  In the other three regimes (star, hub and clique) we provide an explicit structure (graph) which is present in $G(n,m)$ with probability at least $\exp(-b)$ and is directly responsible for the triangle count deviation.  The details are given in Section~\ref{sec:otherLBs}.

\subsection{The normal regime}\label{sec:normalLB}

In this subsection we use the converse Freedman inequality (Lemma~\ref{lem:CF}) to prove that, for some absolute constant $c>0$, we have
\eq{lbN}
\pr{D_{\triangle}(G_m)\, \ge\, cb^{1/2}t^{3/2}n^{3/2}}\, \ge \, \exp(-b)
\eqe
for all pairs $(b,t)$ in the normal regime.  This regime is displayed in green in Figure~\ref{fig:tb}.  If we fix the value of $t\ge Cn^{-1/2}(\log{n})^{1/2}$ then this range corresponds to values of $b$ such that
\[
3\log{n}\, \le\, b\, \le\, 
\min\{t^{1/3}n\ell^{4/3}\, ,\, t^{3/2}n^{3/2}\ell^{3/2}\}\, .
%\min\{tn\ell^{4/3}\, ,\, t^{3/2}n^{3/2}\ell^{3/2}\}\, .
\]
We also remark that in this regime, the truncation values are
\[
K_{\triangle}\, =\, 2^{16}b^{-1/2}t^{3/2}n^{3/2}\quad\text{and} K_{\owedge}\, =\, 2^{8}b^{-1/2}t^{1/2}n^{3/2}\, .
\]

In the previous section we considered a process with truncated increments, with final value $D'_{\triangle}(G_m)$, introduced in~\eqr{D'def}.  Let us recall that 
\[
D'_{\triangle}(G_m)\, =\, \sum_{i=1}^{m} \left[ 3 \, \frac{(N-m)_{2}(m-i)}{(N-i)_3}\, X'_{\owedge}(G_i)\, +\, \frac{(N-m)_3}{(N-i)_3}\, X'_{\triangle}(G_i)\right]\phantom{\Big|}
\]
uses the truncated random variables $X'_{\owedge}(G_i)\, =\, X_{\owedge}(G_i)1_{X_{\owedge}(G_i)\le K_{\owedge}}$ and $X'_{\triangle}(G_i)\, =\, X_{\triangle}(G_i)1_{X_{\triangle}(G_i)\le K_{\triangle}}$ where $K_{\owedge}= 2^8 M(b,t)/bt$ and $K_{\triangle}= 2^{16} M(b,t)/b$.

Unfortunately, the truncations cause the process to be a supermartingale rather than a martingale and we cannot apply the converse Freedman inequality, Lemma~\ref{lem:CF}.  This leaves us with a difficult decision, as we still want to have a bound on the size of increments.

We shall therefore``repair" the martingale property.  Let us define $D''_{\triangle}(G_m)=\sum_{i=1}^{m}\X''_i$ as the sum of the rebalanced increments
\[
\X''_i\, :=\, 3\frac{(N-m)_{2}(m-i)}{(N-i)_3}\, X''_{\owedge}(G_i)\, +\, \frac{(N-m)_3}{(N-i)_3}\, X''_{\triangle}(G_i)\, ,
\]
where $X''_{\owedge}(G_i)$ and $X''_{\triangle}(G_i)$, are defined by
\[
X''_{\owedge}(G_i)\, =\, X'_{\owedge}(G_i)\, +\, \Ex{X_{\owedge}(G_i)1_{X_{\owedge}(G_i)> K_{\owedge}}|G_{i-1}}
\]
and
\[ 
 X''_{\triangle}(G_i)\, =\, X'_{\triangle}(G_i)\, +\, \Ex{X_{\triangle}(G_i)1_{X_{\triangle}(G_i)> K_{\triangle}}|G_{i-1}} \, .
\]
Note that these random variables have been rebalanced so that $\Ex{\X''_i|G_{i-1}}=0$.  That is, $D''_{\triangle}(G_m)=\sum_{i=1}^{m}\X''_i$ is a martingale representation of $D''_{\triangle}(G_m)$.  Let us now state a lower bound for $D''_{\triangle}(G_m)$.

\begin{prop}\label{prop:D''} There is an absolute constant $c>0$ such that the following holds.  Let $n^{-1/2}(\log{n})^{1/2}\le t\le 1/2$, and suppose that $3\log{n}\, \le\, b\, \le\, \min\{t^{1/3}n\ell^{4/3}\, ,\, t^{3/2}n^{3/2}\ell^{3/2}\}$. %\min\{tn\ell^{4/3}\, ,\, t^{3/2}n^{3/2}\ell^{3/2}\}$ 
Then
\[
\pr{D''_{\triangle}(G_m)\, \ge\, cb^{1/2}t^{3/2}n^{3/2}}\, \ge \, \exp(-b)\, .
\]
\end{prop}

We also prove the following lemma which shows that $D''_{\triangle}(G_m)$ is likely to be very close to $D_{\triangle}(G_m)$.

\begin{lem}\label{lem:DD''}  There is an absolute constant $c_0>0$ such that for all $c<c_0$ the following holds.  Let $n^{-1/2}(\log{n})^{1/2}\le t\le 1/2$, and suppose that $3\log{n}\, \le\, b\, \le\,c^{3}\min\{t^{1/3}n\ell^{4/3}\, ,\, t^{3/2}n^{3/2}\ell^{3/2}\}$.  
%\min\{tn\ell^{4/3}\, ,\, t^{3/2}n^{3/2}\ell^{3/2}\}$.  
Then
\[
\pr{D''_{\triangle}(G_m)-D_{\triangle}(G_m)\, \ge\, \frac{cb^{1/2}t^{3/2}n^{3/2}}{2}}\, \le \, \exp(-2b)\, ,
\]
for all sufficiently large $n$.
\end{lem}

It is clear that the lower bound for the normal regime,~\eqr{lbN}, follows from Proposition~\ref{prop:D''} and Lemma~\ref{lem:DD''}.  We now turn to the proofs of these two results.  We begin with Lemma~\ref{lem:DD''}.

We continue to use the notation $k_0$ for the maximum integer $k$ such that $2^kt^2n\le K_{\triangle}/6$.  It may also be checked that $2^{k_0-8} tn\le K_{\owedge}$.  

\begin{proof} First we recall that $D'_{\triangle}(G_m)\le D_{\triangle}(G_m)$ deterministically (as only positive increments are truncated), and that 
\begin{align}
&D''_{\triangle}(G_m) \, -\, D'_{\triangle}(G_m) \, \nonumber \le\\
&\qquad \quad  \sum_{i=1}^{m}\,  3t \Ex{X_{\owedge}(G_i)1_{X_{\owedge}(G_i)> K_{\owedge}}\mid G_{i-1}}\, +\, \Ex{X_{\triangle}(G_i)1_{X_{\triangle}(G_i)> K_{\triangle}}\mid G_{i-1}} \, .\label{eq:sums}
\end{align}
The proof will be based on proving that, if $C$ is a sufficiently large constant, then each of the events
\eq{wedgepp}
\exists i\le m \quad \Ex{X_{\owedge}(G_i)1_{X_{\owedge}(G_i)> K_{\owedge}}|G_{i-1}}\, >\, \frac{Cb}{n}
\eqe
and 
\eq{tripp}
\exists i\le m\quad \Ex{X_{\triangle}(G_i)1_{X_{\triangle}(G_i)> K_{\triangle}}|G_{i-1}}\, >\,  \frac{Cbt}{n}\, +\, \frac{Cb^{5/2}}{ t^{3/2}n^{7/2}\ell^2}\eqe
have probability at most $2\exp(-3b)$.  Let us show that this is sufficient.  Suppose that the events of~\eqr{wedgepp} and~\eqr{tripp} do not occur.  Then, since the summations are over $m\le tn^2$ terms, the right hand side of~\eqr{sums} is at most
\[
3Cbt^2n\, +\, Cbt^2n\, +\, \frac{Cb^{5/2}t^{1/2}}{n^{3/2}\ell^2}\, .
\]
We now use that $bt\le c^3 t^{4/3}n\ell^{4/3}\le c^3n$, and $b^2\le c^6t^3n^3\ell^3\le c^6tn^3\ell^2$, and we see that this is at most
\[
\big(4Cc^{3/2}\, +\, Cc^6 \big)b^{1/2}t^{3/2}n^{3/2}\, \le\, cb^{1/2}t^{3/2}n^{3/2}/2
\]
provided $c$ is sufficiently small, as required.

%Therefore, it suffices to prove that each of the events
%\eq{wedgepp}
%\sum_{i=1}^{m} 3t \Ex{X_{\owedge}(G_i)1_{X_{\owedge}(G_i)> K_{\owedge}}|G_{i-1}}\, >\, cb^{1/2}t^{3/2}n^{3/2}/4
%\eqe
%and 
%\eq{tripp}
%\sum_{i=1}^{m}\Ex{X_{\triangle}(G_i)1_{X_{\triangle}(G_i)> K_{\triangle}}|G_{i-1}}\, >\, cb^{1/2}t^{3/2}n^{3/2}/4
%\eqe
%have probability at most $2\exp(-3b)$.

We begin with~\eqr{wedgepp}.  We remark that the expected value is $0$ here for all $i\le m$ is $\Delta(G_m)\le 64tn$, and this happens except with probability at most $\exp(-96tn)$, see~\eqr{D64}. Suppose now that $b\ge 32tn$. The argument is in some ways similar to that of Section~\ref{sec:Zwedge}.  Again we use that $X_{\owedge}(G_i)\, \le\, A_{\owedge}(G_i)\, =\, 4 \max_{v \in e_i} d_v(G_{i-1})$, and that the condition $X_{\owedge}(G_i)>K_{\owedge}$ implies that this maximum is at least $K_{\owedge}/4> 2^{k_0-8}tn$.  We now just use that each vertex has probability at most $4/n$ to be selected as part of the next edge, and so
\begin{align*}
\Ex{X_{\owedge}(G_i)1_{X_{\owedge}(G_i)> K_{\owedge}}\, \mid\, G_{i-1}}\, & \le\, \frac{16}{n}\sum_{u}d_u(G_{i-1})1_{d_u(G_{i-1})>2^{k_0-8}tn}\\
&\le\, \frac{16}{n}\sum_{j=k_0-8}^{\log_{2}(1/t)} 2^{j+1}tn\, |V_j|\, .
\end{align*}
Now, by Lemma~\ref{lem:largedegs}, we have, except with probability at most $\exp(-3b)$ that $|V_j|\, \le\, b/tnj2^{j-8}$ for all $j\ge 5$.  In this case we have
\begin{align}
\Ex{X_{\owedge}(G_i)1_{X_{\owedge}(G_i)> K_{\owedge}}\, \mid\, G_{i-1}}\, & \le\, \frac{2^{13}}{n}\sum_{j=k_0-8}^{\log_{2}(1/t)}\frac{b}{j}\nonumber\\ 
& \le\, \frac{2^{15}b}{n}\phantom{\Big|}\label{eq:215bn}
\end{align}
for all $i\le m$, where the last inequality used from~\eqr{kell} that $k_0-8\ge \ell/2$, and that there are at most $2\ell$ terms in the summation.  This completes the bound on the probability of~\eqr{wedgepp}

Now we turn to~\eqr{tripp}.  As in Section~\ref{sec:ZK3}, we use that $X_{\triangle}(G_i)\, \le \, 6 d_{e_i}(G_{i-1})$, and split the terms depending on whether each edge is good or bad.  That is, we shall use that
\begin{align*}
&\Ex{X_{\triangle}(G_i)1_{X_{\triangle}(G_i)> K_{\triangle}}|G_{i-1}}\\ &\qquad \le\, \Ex{X_{\triangle}(G_i)1_{e_i\, \text{good}}1_{X_{\triangle}(G_i)> K_{\triangle}}|G_{i-1}}\, +\, \Ex{X_{\triangle}(G_i)1_{e_i\, \text{bad}}1_{X_{\triangle}(G_i)> K_{\triangle}}|G_{i-1}}\, .
\end{align*}
For good edges, the argument is essentially identical to that given above.  We first recall, from the proof of Lemma~\ref{lem:good}, see~\eqr{etov}, that for good $e_i$ we have $d_{e_i}(G_{i-1}) \le 128t \sum_{v \in e_i}d_v(G_m)1_{d_v(G_m) > K_{\triangle}/2^{10}t}$.  Using that each vertex has probability at most $4/n$ to be included in the next edge we obtain
\begin{align*}
 \Ex{X_{\triangle}(G_i)1_{e_i \text{ good}}1_{X_{\triangle}(G_i)> K_{\triangle}}|G_{i-1}}\, &\le\, \frac{2^{9}t}{n}\sum_{v}d_v(G_m)1_{d_v(G_m) > K_{\triangle}/2^{10}t} \\
& \le\,  \frac{2^{9}t}{n}\sum_{j\ge k_0-8}2^{j+1}tn|V_j|\, .
\end{align*}
We may now argue exactly as above in~\eqr{215bn}, to obtain
\[
 \Ex{X_{\triangle}(G_i)1_{e_i \text{ good}}1_{X_{\triangle}(G_i)> K_{\triangle}}|G_{i-1}}\, \le\,\frac{2^{20}bt}{n}\, .
 \]

For bad edges, we recall that $X_{\triangle}(G_i)\le 6d_{e_i}(G_{i-1})$.  We use an argument similar to~\eqr{k1was}, but with $k_0-3$ in place of $k_1$. We obtain that, except with probability at most $\exp(-3b)$, we have
\begin{align*}
%\sum_{i=1}^m d_{e}(G_{i-1}) 1_{\{e \text{ is bad}, d_{e}(G_{i-1})>K_{\triangle}/6\}}
\sum_{e\in E(K_n)\setminus E(G_{i-1})} d_{e}(G_{i-1}) 1_{\{e \text{ is bad}, d_{e}(G_{i-1})>K_{\triangle}/6\}}
 \, &\le\, \sum_{k=k_0-3}^{2\log_2(1/t)} 2^{k+1}t^2n  \frac{C_1 b^2}{k^2 2^{2k}t^4n^2}\phantom{\Bigg|} \\ 
&\le\, \sum_{k=k_0-3}^{2\log_2(1/t)}   \frac{2C_1 b^2}{k^2 2^{k}t^2n}\phantom{\Bigg|}\\
&\le\, \frac{4C_1 b^2}{(k_0-3)^2 2^{k_0-3}t^2n}\\
&\le\, \frac{32C_1 b ^2}{(k_0-3)^2 K_{\triangle}}\\
&\le\, \frac{2^{23}C_1 b^3}{\ell^2 M(b,t)}\\ 
&=\, \frac{2^{23}C_1 b^{3/2}}{t^{3/2}n^{3/2}\ell^2}\, ,
\end{align*}
where we used that $k_0-3\ge \ell/2$, see~\eqr{kell}, and that $M(b,t)=b^{1/2}t^{3/2}n^{3/2}$ in this regime.  As each pair is included as the next edge with probability at most $4/n^2$ we obtain
\[
\Ex{X_{\triangle}(G_i)1_{\{e_i \text{ is bad}\}}1_{X_{\triangle}(G_i)> K_{\triangle}}|G_{i-1}}\, \le\,\frac{C b^{3/2}}{t^{3/2}n^{7/2}\ell^2}\, , 
\]
for the constant $C=2^{23}C_1$.  This completes the bound on the probability of~\eqr{tripp}, and so completes the proof.\end{proof}

We now turn to the proof of Proposition~\ref{prop:D''}.  As discussed above, the idea is to use Freedman's lower bound inequality.  In order to do so we require some lower bounds on the variance of the process.

In Lemma~\ref{lem:vars}, with $b=n$, we saw that except with probability at most $e^{-n}$ we have an upper bound of $Ct^2n$ on the variance of the $i$th increment.  The following statement gives a lower bound of the same order.  This is not surprising: as codegrees typically vary by order $tn^{1/2}$, one would expect the increments to typically be of this order.

\begin{restatable}{lem}{LBvar}\label{lem:LBvar}
There exists an absolute constant $c_1>0$ such that the following holds.  Let $n^{-1/2}/c_1 \le t\le 1/2$.  For all $m/2\le i\le m$, we have
\[
\pr{\Ex{(\X''_i)^2|G_{i-1}}\, \le \, c_1 t^2n}\, \le\, \exp(-c_1 n)\,
\]
for all sufficiently large $n$.
\end{restatable}

As the proof of the lemma is long but relatively elementary it appear in Appendix~\ref{sec:LBvar}.  We now complete the proof of Propositon~\ref{prop:D''}.

\begin{proof}[Proof of Proposition~\ref{prop:D''}]  Let us fix the choice of $t$ in the range $n^{1/2}(\log{n})^{1/2}\le t\le 1/2$, and fix $3\log{n}\, \le\, b\, \le\, \min\{t^{1/3}n\ell^{4/3}\, ,\, t^{3/2}n^{3/2}\ell^{3/2}\}$.
%\min\{tn\ell^{4/3}\, ,\, t^{3/2}n^{3/2}\ell^{3/2}\}$.  
In fact, we may assume $b$ is at most a small constant times $\min\{t^{1/3}n\ell^{4/3}\, ,\, t^{3/2}n^{3/2}\ell^{3/2}\}$, otherwise simply use the lower bound given by that value of $b$ and adjust the constant.  In particular, we assume $b\le c^3\min\{t^{1/3}n\ell^{4/3}\, ,\, t^{3/2}n^{3/2}\ell^{3/2}\}\le c^2 n$, where $c\le 2^{-36}$ is the minimum of constants of Lemmas~\ref{lem:DD''} and~\ref{lem:LBvar}.   We shall prove a lower bound on the probability
\[
\pr{D''_{\triangle}(G_m)\, \ge\, cb^{1/2}t^{3/2}n^{3/2}}
\]
using the converse Freedman inequality (Lemma~\ref{lem:CF}).  Let us set $\alpha'=cb^{1/2}t^{3/2}n^{3/2}$ throughout the proof.

There are two small issues.  One is that the Lemma~\ref{lem:CF} gives a lower bound on the probability that the martingale exceed a certain value, $\alpha'$ say, before a certain time $m$, rather than \emph{at} step $m$.  This is easy to deal with by instead proving there is a good probability we reach $2\alpha'$, and then use our upper bound result to control the probability the process falls thereafter.  The same issue was also considered in~\cite{GGS}.

The other issue is that we cannot bound $\|\X''_i\|_{\infty}$ as well as we would like, as is required to use Lemma~\ref{lem:CF}.  So instead we consider a variation of the process which is very likely to be equal, but may behave differently on a rare event.  Let $E_i$ be the event that one of the events~\eqr{wedgepp} or~\eqr{tripp} holds for a particular value of $i$.  Let $E^+_{i-1}$ be the event that there exists a pair $e\in E(K_n)\setminus E(G_{i-1})$ such that if $e_i=e$ then $E_i$ will occur.  And let $E^{+}=\bigcup_{i=1}^{m}E^{+}_{i-1}$.  As each pair $e\in E(K_n)\setminus E(G_{i-1})$ has probability at least $n^{-2}$ we have that 
\eq{Eplus}
\pr{E^{+}}\, \le\, \sum_{i=1}^{m}\pr{E^{+}_{i-1}}\, \le\, n^2\sum_{i=1}^{m}\pr{E_{i}}\, \le\, 4n^2\exp(-3b)\, ,
\eqe
where the final inequality relies on the bounds of the probability of the events~\eqr{wedgepp} and~\eqr{tripp} given in the proof of Lemma~\ref{lem:DD''}.  Let us define $X^*_i$ by
\[
X^*_i\, :=\, \X''_i 1_{(E^{+}_{i-1})^c}\, .
\]
As the event is $G_{i-1}$-measurable we have that $X^*_i$ still behaves as a martingale increment in the sense that $\Ex{X^*_i\mid G_{i-1}}=0$.  Let $D^{*}_{\triangle}(G_m)=\sum_{i=1}^{m}X^*_i$.  We observe that
\[
\pr{D''_{\triangle}(G_m)\neq D^{*}_{\triangle}(G_m)}\,\le \, 4n^2\exp(-3b)\, \le\, \exp(-2b)\, .
\]
Also, $X^*_i$ is only non-zero if $E^{+}_{i-1}$ does not occur in which case the expected values in~\eqr{wedgepp} and~\eqr{tripp} are at most $C_1b/n$ and $C_1bt/n\, +\, C_1b^{5/2}/t^{3/2}n^{7/2}\ell^2$ respectively, for some constant $C_1$.  On this event we have 
\[
|\X''_i-\X'_i|\, \le\, \frac{C_2bt}{n}\, +\, \frac{C_2 b^{5/2}}{t^{3/2}n^{7/2}\ell^2}\, \le\,  \frac{C_3\NORMAL(b,t)}{b}
\]
for constants $C_2,C_3$, where the last bound follows as $b^3\le tn^3\ell^4\le t^3n^5\ell^2\le tn^5$.  As 
\[
|\X'_i|\le 3t|X'_{\owedge}(G_i)|+|X'_{\triangle}(G_i)|\, \le\, \frac{C_4 M(b,t)}{b} \, =\, \frac{C_4\NORMAL(b,t)}{b}
\]
in this regime, we therefore have an absolute bound
\begin{align*}
|X^*_i|\, &\le\, |\X''_i| 1_{(E^{+}_{i-1})^c}\\
&\le\, |\X'_i|\, +\, |\X''_i-\X'_i|1_{(E^{+}_{i-1})^c}\\
&\le\, \frac{C\NORMAL(b,t)}{b}\, ,
\end{align*}
for some constant $C$.

%We have that $D''_{\triangle}(G_m)=\sum_{i=1}^{m}\X''_i$ is a martingale representation of $D''_{\triangle}(G_m)$.  One small issue with using Lemma~\ref{lem:CF} is that it actually gives a lower bound on the probability $\sum_{i=1}^{m'}\X''_i\ge \alpha'$ for some $m'\le m$.  We get around this by instead asking for the partial sum to be twice as big $\sum_{i=1}^{m'}\X''_i\ge 2\alpha'$ for some $m'\le m$, and there is then probability at least $1/2$ (in fact very close to $1$) that the final sum is also at least $\alpha'$.  This is similar to the argument given in~\cite{GGS}.  In our setting this may be achieved by applying our upper bound technique to the remaining part of the martingale.

%It is also useful to replace the increments $\X''_i$ by increments $X^*_i$ defined as follows: 
%\[
%X^*_i\, := \X''_i 1_{\|\X''_i\mid G_{i-1}\|_{\infty}\le 2M(b,t)/b
%\]

We also have, by Lemma~\ref{lem:LBvar} that, except with probability at most $n^2\exp(-cn)$, all the increments satisfy
\[
\Ex{(\X''_i)^2\, |\, G_{i-1}}\, \ge\, c t^2n\, .
\]
It follows that, except with probability at most $n^2\exp(-c n)\le \exp(-2b)$, we have
\[
\sum_{i=1}^{m}\Ex{(\X''_i)^2\, |\, G_{i-1}}\, \ge\, c t^3n^3/10\, .
\]
As there is probability at most $\exp(-2b)$ that $X^*_i\neq \X''_i$ occurs for any $i$, we have that
\[
V(m)\, :=\,\sum_{i=1}^{m}\Ex{(X^{*}_i)^2\, |\, G_{i-1}}\, \ge\, c t^3n^3/10\, ,
\]
except with probability at most $\exp(-b)$.

%Also, by the choice of the truncation used to define $X'_{\owedge}(G_i)$ and $X'_{\triangle}(G_i)$ we have 
%\[
%|\X''_i|\, \le\, \frac{2M(b,t)}{b}\, \le\, \frac{2\NORMAL(b,t)}{b}\, .
%\]
%
%\nj{Note that when we `repair' the process, we increase the former by a mean of large variables. From the computations above, we only can infer that the $G_{i-1}$-event (say $\mathcal{E}_{i-1}$) defined by
%\[
%|\X''_i|\, \le\, \frac{2M(b,t)}{b}\, \le\, \frac{2\NORMAL(b,t)}{b}\, .
%\]
%occurs, except with probability at most $\exp(-b)$ (due to the `repaired' term). Now, to apply Freedman inequality as desired, we could modify our process slightly by defining a new martingale:
%\[\sum_{i=1}^{m}{ \X''_i 1_{\mathcal{E}_{i-1}}}
%\]
%Now, we can proceed (essentially) as bellow (using Freedman's Inequality). Together with the rare events $\mathcal{E}^c_{i-1}$, we are done. 
%}

We shall apply Lemma~\ref{lem:CF} with $\alpha=2\alpha'=2cb^{1/2}t^{3/2}n^{3/2}$, $\beta=c t^3n^3/10$ and $R=2\NORMAL(b,t)/b=b^{-1/2}t^{3/2}n^{3/2}$.  We may take $\delta=1$.  We obtain
\begin{align*}
\pr{\sum_{i=1}^{m'}X^*_i\ge \alpha\, \text{for some }m'\le m}\, & \ge\, \frac{1}{2} \exp\left(\frac{-3\alpha^2}{\beta}\right)\, -\, \exp(-b)\\ 
& \ge \, \exp(-120cb)\, -\, \exp(-b)\\
& \ge\, \exp(-b/2)\, -\, \exp(-b)\\
&\ge\, 4\exp(-b)\, .
\end{align*}
Let $\tau$ be the hitting time associated with the event $\sum_{i=1}^{m'}X^{*}_i\ge \alpha$, i.e., the least $m'$ such that this occurs.  We now prove that 
\[
\pr{D^{*}_{\triangle} (G_m)\ge \alpha'\, \mid\, \tau\le m}\, \ge\, \frac{1}{2}\, .
\]
By standard arguments, it suffices to prove the conditional probability is at least $1/2$ for any fixed value $\tau=m'<m$.  Now, on this event, it is actually very likely that $D^{*}_{\triangle}(G_m) =\sum_{i=1}^{m} X^*_i\ge \alpha'=\alpha/2$.  The conditional probability that this fails is at most
\[
\pr{\left|\sum_{i=m'+1}^{m} X^{*}_i\right|\, \ge\, \alpha'\, \mid\, \tau=m'}
\]
which may easily seen to be at most $\exp(-cb)\le 1/2$ by arguing as in Section~\ref{sec:UBtri}.  This shows that $D^{*}_{\triangle} (G_m)\ge \alpha$ with probability at least $2\exp(-b)$.  Finally, as they differ with probability at most $\exp(-2b)\le \exp(-b)$ we have that $D''_{\triangle}(G_m)\ge \alpha$ with probability at least $\exp(-b)$.
\end{proof}

\subsection{Lower bounds in the remaining regimes}\label{sec:otherLBs}

We recall that we must prove in these regimes that, for some constant $c>0$, we have
\[
N_{\triangle}(G_m)\, >\, \Ex{N_{\triangle}(G_m)}\, +\, cM(b,t)
\]
with probability at most $\exp(-b)$.  In these regimes, we always have $b\ge 6\log{n}$.

In each of the other three regimes we give a graph $G_{*}$ containing at most $b/10\ell$ edges and such that
\eq{condex}
\Ex{N_{\triangle}(G_{m-b/10\ell}\cup G_*)}\, >\, \Ex{N_{\triangle}(G_m)}\, +\, 2cM(b,t)\, .
\eqe
Note that this implies that
\[
\Ex{N_{\triangle}(G_{m})|G_m\supseteq G_*}\, >\, \Ex{N_{\triangle}(G_m)}\, +\, 2cM(b,t)\, ,
\]
as one may couple $G_m$ (conditioned on $G_m\supseteq G_*$) to contain $G_{m-b/10\ell}\cup G_*$.
We shall see that this is sufficient to prove the lower bound.  First note that 
\[
\pr{G_*\subseteq G_m}\, \ge\,
\frac{(m)_{b/10\ell}}{(N)_{b/10\ell}}\, \ge \, (t/2)^{b/10\ell}\, \ge\, \exp(-b/2)\, .
\]
Also 
\begin{align*}
&\Ex{N_{\triangle}(G_{m})\, \mid\phantom{\big|} G_m\supseteq G_*}\le \\ & \, \Ex{N_{\triangle}(G_m)}\, +\, cM(b,t)\, +\, n^3\pr{N_{\triangle}(G_m)\, >\, \Ex{N_{\triangle}(G_m)} + cM(b,t)\, \bigm\vert\phantom{\Big|} G_m\supseteq G_*}\, ,
\end{align*}
from which it follows that
\[
\pr{N_{\triangle}(G_m)\, >\, \Ex{N_{\triangle}(G_m)}\, +\, cM(b,t)\, |\, G_m\supseteq G_*}\, \ge\, n^{-3}
\]
and so 
\[
\pr{N_{\triangle}(G_m)\, >\, \Ex{N_{\triangle}(G_m)}\, +\, cM(b,t)}\, \ge\, n^{-3}\exp(-b/2)\, \ge\, \exp(-b)\, ,
\]
as required.

We now present the three structures (subgraphs) which give the required lower bounds in the three regimes.

\textbf{The star regime:} In the star regime we may take $G_*$ to be a star of degree $b/10\ell$.  

Recall that $\STAR(b,t)\, := \, b^2t\ell^{-2} 1_{b\le n\ell}$.   The expected number of triangles in $G_{m-b/10\ell}\cup G_*$ involving the vertex at the centre of the star $G_*$ is of the order $b^2t\ell^{-2}$.  Since a typical edge is in order $t^2n$ triangles, $\Ex{N_{\triangle}(G_m)}-\Ex{N_{\triangle}(G_{m-b/10\ell})}$ As the expected number of  star of degree $b/10\ell$ will

\textbf{The hub regime:} In the hub regime we may take $G_*$ to be a hub which consists of $b/10n\ell$ vertices of full degree, i.e., of degree $n-1$.

\textbf{The clique regime:} In the clique regime we take $G_*$ to be a clique with $b^{1/2}/4\ell^{1/2}$ vertices.

\section{Lower bounds for cherries}\label{sec:cherrieslower}

We now observe the corresponding lower bounds for cherries, i.e., the lower bound part of Theorem~\ref{thm:cherries}.  As we mentioned in the introduction, we do so at the level of a sketch proof.

As we noted in Section~\ref{sec:cherries}, for each fixed value of $t\ge 2n^{-1}\log{n}$ the normal regime corresponds to $3\log{n}\le b\le t^{2/3}n\ell^{4/3}$, the star regime corresponds to $t^{2/3}n\ell^{4/3}<b\le n\ell$ and the hub regime corresponds to $b>n\ell$.

The star and hub regimes are straightforward.  In the star case, suppose that $t^{2/3}n\ell^{4/3}<b\le n\ell$, we must prove that a deviation at least $cb^2/\ell^2$ has probability at least $\exp(-b)$.   There is probability at least $t^{(1+o(1))b/2\ell}\ge \exp(-b)$ that the first vertex of the graph has degree at least  $b/2\ell$.  In this case, this vertex alone is in $\Omega(b^2/\ell^2)$ cherries.  This argument may be made rigorous by conditioning on the corresponding event as we did for triangles in the previous section.

In the hub regime, suppose $b\ge n\ell$, we must prove that a deviation at least $cbn/\ell$ has probability at least $\exp(-b)$.  There is probability at least $t^{(1+o(1))bn/2n\ell}\ge \exp(-b)$ that there are $b/2n\ell$ vertices which have degree $n-1$.  These vertices alone are in $\Omega(bn/\ell)$ cherries.  Again, this argument may be made rigorous using conditioning.

In the normal regime, suppose $3\log{n}\le b\le t^{2/3}n\ell^{4/3}$, we must prove that a deviation of size at least $cb^{1/2}tn^{3/2}$ has probability at least $\exp(-b)$.  By analogy with the previous section one may define $D''_{\owedge}(G_m)$ as the sum of the ``repaired'' martingale increments $X''_{\owedge}(G_i)$.  One may easily prove that there is probability at most $\exp(-2b)$ that $|D_{\owedge}(G_m)-D''_{\owedge}(G_m)|=\Omega(b^{1/2}tn^{3/2})$, and so it suffices to prove the bound for $D''_{\owedge}(G_m)$.

By an argument similar to that in the previous section one may show that it is very unlikely, with probability at most $\exp(-\Omega(n))$, that all degrees are concentrated close to the average, and do not vary on the scale $t^{1/2}n^{1/2}$.  This allows us to show that, with very high probability, $\Ex{X_{\owedge}(G_i)^2\mid G_{i-1}}\ge c'tn$ throughout the process, and likewise for the repaired process, $\Ex{X''_{\owedge}(G_i)^2\mid G_{i-1}}\ge c'tn$.  The required lower bound now follows\footnote{technically it may be necessary to introduce an $X^{*}$ process as we did in the proof of Propositon~\ref{prop:D''}} from the converse Freedman inequality (Lemma~\ref{lem:CF}) using $R=2^{10}tn$ and $\beta=c't^2n^3/2$, just as in the proof of Proposition~\ref{prop:D''}. 

\section{Triangles in $G(n,p)$}\label{sec:pworld}

Here we deduce our results in $G(n,p)$, Theorem~\ref{thm:mainp} and Theorem~\ref{thm:localp}.

We obtain results in $G(n,p)$ from our $G(n,m)$ results, as has been done in~\cite{GGS} and other articles.  In particular, we use Theorem~\ref{thm:mainm}, which control deviation probabilities of $N_{\triangle}(G_m)$, in combination with the expression\footnote{Note that~\eqr{pfromm} may be easily proved by simply considering conditioning on the number of edges in $G(n,p)$.}
\eq{pfromm}
\pr{N_{\triangle}(G_p)\, >\, \delta_n p^3(n)_{3}}\, =\, \sum_{m=0}^{N}b_N(m)\, \pr{N_{\triangle}(G_m)\, >\, (1+\delta_n)p^3 (n)_{3}}\, ,
\eqe
where $G_p\sim G(n,p)$ and $b_{N}(m):=\pr{Bin(N,p)=m}$.  

It is therefore useful to have the following estimates for $b_N(k)$ and 
\[
B_N(k) := \pr{\mathrm{Bin}(N,p) \ge k} \, .
\]
The results, stated in terms of $x_N = \frac{k-pN}{\sqrt{pqN}}$, are valid for $p\in (0,1)$ a constant or $p=p_N$ a function.  This version of the result was stated in~\cite{GGS}, it follows easily from Theorem 2 of Bahadur~\cite{Bah}.  It includes the expression
\[
E(x,N)\, :=\, \sum_{i=1}^{\infty} \frac{(p^{i+1} + (-1)^i q^{i+1})x^{i+2}}{(i+1)(i+2) p^{i/2} q^{i/2} N^{i/2}}\, ,
\]
where $q$ denotes $1-p$, as it shall in the statements which follow.  

%We write $E(x,N,J)$ for the partial sum up to $i=J$.  

\begin{theorem} \label{thm:bah}
Suppose that $(x_N)$ is a sequence such that $1\ll x_N\ll \sqrt{pqN}$. Then
\[
b_N(\fl{pN+x_N\sqrt{pqN}})\,  =\, (1 + o(1))\frac{1}{\sqrt{2 \pi pqN}} \exp \left(-\frac{x_N^2}{2} - E(x_N,N) \right)
\]
and
\[
B_N(pN + x_N \sqrt{pqN})\, =\, (1 + o(1)) \frac{1}{x_N \sqrt{2\pi}} \exp \left(- \frac{x_N^2}{2} - E(x_N,N) \right)\, .
\]
%Furthermore, if $1\ll x_N\ll (pqN)^{1/2} (pqN)^{-1/(J+3)}$ then the infinite sum $E(x_N,N)$ may be replaced by the finite sum $E(x_N,N,J)$ in both expressions. 
\end{theorem}

In this section we prove Theorem~\ref{thm:mainp} and Theorem~\ref{thm:localp}.  We begin with Theorem~\ref{thm:mainp}, which gives the asymptotic value of $\rdn$ for a certain range of the parameters $\delta,p$.  We first state a more precise version of this result.  Let us define
\[
\tilde{M}(\delta,p)\, :=\, C M(\delta^2 pn^2,2p)
\]
where $M$ is as defined in the introduction, see~\eqr{Mdef}, and $C$ is the constant of Theorem~\ref{thm:mainm}.  We also define 
\[
x_*\, :=\, \left[(1+\delta_n)^{1/3}\, -\, 1\right]\, \sqrt{\frac{pN}{q}}\, .
\]

\begin{prop}\label{prop:mainp}
Let $n^{-1/2}(\log{n})^{1/2}\ll p\ll 1$ and let $\delta_n$ be a sequence satisfying 
\[
p^{-1/2}n^{-1}\, \ll\, \delta_n\, \ll \, p^{3/4}(\log{n})^{3/4}\, ,\, n^{-1/3}(\log{n})^{2/3}\, +\, p\log(1/p)\, .
\]
Then
\[
\rdn\, =\,  \frac{x_*^2}{2}\, +\, E(x_*,N)\, +\, \log{x_{*}}\, +\, O\left(\frac{\delta_n \tilde{M}(\delta_n,p)}{p^2 n}\right)\, +\, O(1)\, .
\]
\end{prop}

It is easily observed that Proposition~\ref{prop:mainp} implies Theorem~\ref{thm:mainp}.  First observe that $x_*=(1+o(1))\delta_n p^{1/2}n/3\sqrt{2}$, that $x_*\ll \sqrt{pqN}$ from which is follows that $E(x_*,N)\ll x_*^2$.  Now we will check that
\eq{tMle}
\tilde{M}(\delta_n,p)\,\ll\, \delta_n p^3n^3\, 
\eqe
in the given range of $\delta_n$.  To see this, we check for the four functions used to define $M(b,t)$.  Note that $\NORMAL(\delta^2 pn^2,2p)\ll \delta p^3n^3$ for all $\delta$, and that $\CLIQUE(\delta^2 pn^2,2p)\ll \delta p^3n^3$ provided $\delta\ll p^{3/4}(\log{n})^{3/4}$.  For the star and hub regimes the corresponding upper bounds are $\delta\ll n^{-1/3}(\log{n})^{2/3}$ and $\delta\ll p\log(1/p)$.  However, we note that the border between these regimes occurs when $\delta^2 pn^2=n\ell$, that is $\delta=p^{-1/2}n^{-1/2}\ell^{1/2}$, and that it is the maximum of the these two functions which is required as the upper bound in the corresponding regime.  For this reason the bound $\delta\ll n^{-1/3}(\log{n})^{2/3}\, +\, p\log(1/p)$ suffices.

From the point of view of its application, the intuition behind~\eqr{tMle} is that the contribution to the deviation made by $G(n,m)$ is small.  In fact, the quantity $\tilde{M}(\delta_n,p)$ behaves as an upper bound on this $G(n,m)$ contribution.  Indeed, it follows from Theorem~\ref{thm:mainm} that
\eq{fromm-}
\pr{N_{\triangle}(G_m)\, >\, \Ex{N_{\triangle}(G_m)}\, +\, \tilde{M}(\delta_n,p)}\, \le\, \exp(-\delta^2 p n^2)\, 
\eqe
for all $m\le 2pN$.  Since $\delta_n^2pn^2\ge x_*^2$, in this regime, this probability is negligible.

We now prove Proposition~\ref{prop:mainp}.  This will require us to prove an upper and lower bound on the probability $\pr{N_{\triangle}(G_p)\ge (1+\delta_n)p^3(n)_{3}}$.  We define 
\[
m_*\, :=\, pN(1+\delta_n)^{1/3}\, 
\]
which is (approximately) the value of $m$ at which $\Ex{N_{\triangle}(G_m)}=(1+\delta_n)p^3(n)_{3}$.  In other words, if $G_p$ has at least this many edges then no $G(n,m)$ deviation is required in order to achieve the stated deviation.  It will also be useful to consider
\[
m_-\, :=\, m_*\, -\, 2p^{-2}n^{-1}\tilde{M}(\delta_n,p)\, .
\]
%\[
%m_-\, :=\, m_*\, -\, 2p^{-2}n^{-1}\tilde{M}(\delta_n^2 pN,2p)\, .
%\]
Let $x_-:= (m_--pN)/\sqrt{pqN}$ be the corresponding value of $x$.  Note that $x_-=x_{*}-\Theta(\tilde{M}(\delta_n,p)/p^{5/2}n^2)$, and that $x_*=\Theta(\delta_n p^{1/2}n)$.

\begin{proof}[Proof of Proposition~\ref{prop:mainp}] The lower bound on $\pr{N_{\triangle}(G_p)\ge (1+\delta_n)p^3(n)_{3}}$ is relatively straightforward.  Using~\eqr{pfromm} and monotonicity we have that
\[
\pr{N_{\triangle}(G_p)\,\ge \,(1+\delta_n)p^3(n)_{3}}\, \ge \, B_N(m_*)\pr{N_{\triangle}(G_{m_*})\ge (1+\delta_n)p^3(n)_{3}}\, .
\]
We also note that the expected number of triangles in $G(n,m_*)$ is $(1+\delta_n)p^3(n)_{3}+O(p^2n)$ and so the difference from $(1+\delta_n)p^3(n)_{3}$ is much less than $t^{3/2}n^{3/2}$, the order of the standard deviation of this random variable in $G(n,m)$.  It follows that the latter probability above is $\frac{1}{2}+o(1)=\exp(O(1))$, by the central limit theorem.  It follows that
\[
\pr{N_{\triangle}(G_p)\,\ge \,(1+\delta_n)p^3(n)_{3}}\, \ge \, \exp\left(\frac{-x_*^2}{2}\, -\, E(x_*,N)\, -\, \log{x_{*}}\,  +\, O(1)\right)
\]
which gives the required bound, which is the upper bound on the rate $\rdn$.

We now prove an upper bound on $\pr{N_{\triangle}(G_p)\,\ge \,(1+\delta_n)p^3(n)_{3}}$.  Again we use~\eqr{pfromm}.  This time we simply consider the contribution above $m_-$ and below $m_-$ respectively.  The contribution from above $m_-$ is at most
\begin{align*}
B_N(m_-)\, & \le\, \exp \left(\frac{-x_-^2}{2} \, -\, E(x_-,N)\, -\, \log{x_-} \, +\, O(1) \right)\\
&\le\, \exp\left(\frac{-x_{*}^2}{2}\, +\, O\left(\frac{x_{*}\tilde{M}(\delta_n,p)}{p^{5/2} n^2}\right) -\, E(x_*,N)\, -\, \log{x_*}\, +\, O(1)\right)\\
&\le\, \exp\left(\frac{-x_{*}^2}{2}\, +\, O\left(\frac{\delta_n \tilde{M}(\delta_n,p)}{p^2 n}\right) -\, E(x_*,N)\, -\, \log{x_*}\, +\, O(1)\right)
\end{align*}
as required.

Now for $m\le m_-$.  We claim that 
\eq{mdev}
\pr{N_{\triangle}(G_{m_-})\ge (1+\delta_n)p^3(n)_{3}}\, \le\, \exp(-\delta_n^2 pN)\, .
\eqe
Since this probability is much less than the required bound, and holds for $m\le m_-$ by monotonicity, this will complete the proof.

To prove~\eqr{mdev} we observe that $N_{\triangle}(G_{m_-})\ge (1+\delta_n)p^3(n)_{3}$ corresponds to a large positive deviation for the triangle count in $G(n,m)$.  Indeed
\begin{align*}
\Ex{N_{\triangle}(G_{m_-})}\, & \le \, \left(\frac{m_-}{N}\right)^3(n)_{3}\\
& =\, \left(\frac{m_*\, -\, 2p^{-2}n^{-1}\tilde{M}(\delta_n,p)}{N}\right)^3 (n)_{3}\\
& \le\, \left(\frac{m_*}{N}\right)^3 (n)_{3}\, -\, \tilde{M}(\delta_n,p)\\
& =\, (1+\delta_n)p^3(n)_{3} \, -\, \tilde{M}(\delta_n,p)\, .
\end{align*}
So we see that having $(1+\delta_n)p^3(n)_{3}$ triangles in $G_{m_-}$ would represent a large positive deviation, at least $\tilde{M}(\delta_n,p)$.  As we observed above,~\eqr{fromm-}, this has probability at most $\exp(-\delta_n^2 pN)$.  This completes the proof of~\eqr{mdev} and the proof of the proposition.
\end{proof}

We now prove our result for the localised regime, Theorem~\ref{thm:localp}.  We recall that this theorem states that for the regime
\[
p^{3/4}(\log{n})^{3/4}\, ,\, n^{-1/3}(\log{n})^{2/3}\, +\, p\log(1/p)\, \le\, \delta_n\, \le \, 1\,
\]
we have
\eq{rdnis}
\rdn\, =\, \Theta(1) \min\{\delta_n^{2/3}p^2n^2\log{n}, \delta_n^{1/2}pn^{3/2}\log{n}\, +\, \delta_n p^2n^2 \log(1/p)\}\,  .
\eqe
In this regime the easiest way to achieve the required deviation is by $G_p$ containing a certain structure (a star, a hub or a clique).  These examples will give us the lower bound.

For the upper bound we shall show that one cannot get far simply by having extra edges without any structure.  Let $m_0=pN+\delta pn^2/10$.  It will be useful to note that 
\eq{binm0}
B_N(m_0)\, \le\, \exp\left(\frac{-\delta_n^2 pn^2}{200}\right)\, .
\eqe
We are somewhat informal with some of the (standard) lower bounds.

\begin{proof}[Proof of Theorem~\ref{thm:localp}] 
The proof will be similar in each of the three parts, corresponding to the most likely cause of extra triangles: hub, star or clique.  In the Hub region,
\[
p\log(1/p)\, , \, \frac{1}{p^2 n} \le\, \delta_n \, \le\, 1
\]
the minimum in~\eqr{rdnis} is $\Theta(\delta_n p^2n^2 \log(1/p))$.  We note that in this region (as $\delta_n\ge p\log(1/p)$) we have from~\eqr{binm0} that $\pr{e(G_p)>m_0}\, \le\, \exp(-\Omega(\delta_n p^2n^2 \log(1/p)))$.

It is now straightforward to give the upper and lower bound on the probability $\pr{N_{\triangle}(G_p)\ge (1+\delta_n)p^3(n)_{3}}$ in this region.  The lower bound follows by asking that some $\Theta(\delta_n p^2n)$ vertices have degree $n-1$.  This has probability $\exp(-O(\delta_n p^2n^2 \log(1/p)))$ and produces an extra $\delta_n p^3n^3$ triangles, as required. %\footnote{It is easily checked in each of these cases that it is likely the rest of the graph produces at least $(1-\delta_n)p^3(n)_{3}$ triangles}.
For the upper bound we note that
\begin{align*}
&\pr{N_{\triangle}(G_p)\ge (1+\delta_n)p^3(n)_{3}}\\ \phantom{\Big|}
&\qquad \le\, \pr{e(G_p)>m_0}\, +\, \pr{N_{\triangle}(G_p)\ge (1+\delta_n)p^3(n)_{3}\mid e(G_p)\le m_0}\phantom{\Big|}\\
&\qquad \le\, \exp(-\Omega(\delta_n p^2n^2 \log(1/p)))\, +\, \pr{N_{\triangle}(G_{m_0})\ge (1+\delta_n)p^3(n)_{3}}\, .\phantom{\Big|}
\end{align*}
It therefore suffices to prove the same bound for the second probability.  It is easily checked that $\Ex{N_{\triangle}(G_{m_0})}\le (1+9\delta_n/10)p^3(n)_{3}$, so that the event $N_{\triangle}(G_{m_0})\ge (1+\delta_n)p^3(n)_{3}$ represents a positive deviation $\Theta(\delta_n p^3n^3)$ in $N_{\triangle}(G_{m_0})$.  This has probability at most $\exp(-\Omega(\delta_n p^2n^2 \log(1/p)))$ by Theorem~\ref{thm:mainm} (in this region $\HUB(b,t)$ is the maximiser), as required.

The proofs in the other two regions are very similar.  In the star region, 
\[
\frac{1}{p^6 n^3}\, , \, n^{-1/3}(\log{n})^{2/3}\, \le\,  \delta_n\, \le\, \frac{1}{p^2 n}
\]
the minimum in~\eqr{rdnis} is $\Theta(\delta_n^{1/2}pn^{3/2}\log{n})$, and, similarly, in this region (as $\delta_n\ge n^{-1/3}(\log{n})^{2/3}$), we have from~\eqr{binm0} that $\pr{e(G_p)>m_0}\, \le\, \exp(-\Omega(\delta_n^{1/2}pn^{3/2}\log{n}))$.

The lower bound follows by asking that some vertex have degree at least $\delta_n^{1/2}pn^{3/2}$.  This has probability $\exp(-O(\delta_n^{1/2} p n^{3/2} \log{n}))$\footnote{Note that $\log{n}$ and $\log(1/p)$ are equivalent (of the same order) in this regime} and produces an extra $\delta_n p^3n^3$ triangles, as required.  The upper bound is exactly as in the hub case except that the maximiser in Theorem~\ref{thm:mainm} is now $\STAR(b,t)$, which gives the bound $\exp(-\Omega(\delta_n^{1/2}pn^{3/2}\log{n}))$.

In the clique region,
\[
p^{3/4}(\log{n})^{3/4} \, \le \delta_n\, \le\, \frac{1}{p^6 n^3}\, 
\]
the minimum in~\eqr{rdnis} is $\Theta(\delta_n^{2/3}p^2n^2\log{n})$.  As $\delta_n\ge p^{3/4}(\log{n})^{3/4}$ we have from~\eqr{binm0} that $\pr{e(G_p)>m_0}\, \le\, \exp(-\Omega(\delta_n^{2/3}p^2n^2\log{n}))$.  The lower bound comes from planting a clique on $2\delta_n^{1/3}pn$ vertices which creates at least $\delta_n p^3n^3$ extra triangles and has probability $\exp(-O(\delta_n^{2/3}p^2n^2\log{n}))$.  The upper bound is as above except that now $\CLIQUE(b,t)$ is the maximiser in Theorem~\ref{thm:mainm}, we obtain the bound $\exp(-\Omega(\delta_n^{2/3}p^2n^2\log{n}))$, as required.
\end{proof}

\section{Concluding Remarks}\label{sec:ConcRem}

One natural type of question to consider is whether it is possible to understand the tail probabilities discussed in this article even more precisely.  For example, can we find the asymptotic value of the log of deviation probabilities in general in the $G(n,m)$ and $G(n,p)$ models?

In both models, this has been achieved in certain regions.  For example, in the dense case, such results for both models are given for the normal regime in~\cite{GGS}.  These results also cover a part of the normal regime in slightly sparser random graphs.  And, as we discussed in the introduction, asymptotically tight bounds for the log probability associated with large deviations were also found for the triangle count in $G(n,p)$.  It would be of interest to know whether these results may be extended to cover the all localised regimes (star, hub and clique).

One may also ask about the typical structure of random graphs which exhibit certain deviations.  It seems likely that, with a little more work, weak structural versions of our results on triangles could be proved.  For example, given a pair $t,a$ which corresponds to the clique regime, one may ask whether $G_m$ conditioned on having $a$ more triangles than expected is very likely to contain a somewhat dense subgraph on $\Theta(a^{1/3})$ vertices.  A stronger structural result stating that in this case $G_m$ contains a clique (or at least a very dense subgraph) on $(1+o(1))a^{1/3}$ vertices would be very interesting.  See Theorem 1.8 of~\cite{HMS} for related results for large deviations in $G(n,p)$.

It is also natural to ask whether the results generalise to other graphs.  These results were obtained for relatively dense random graphs in~\cite{GGS}, but it would be interesting to know how far such results extend.

Let us also mention an elementary question, which seems to be open in general, and which is related to Lemma~\ref{lem:multi}.  We write $\mathcal{G}_{n,m}$ for the set of all graphs with $m$ edges on vertex set $[n]=\{1,\dots ,n\}$.  Given a set $S$ of vertices of a graph $G$, let us set $N(S):=\bigcup_{v\in S}N(v)$, the union of the neighbourhoods of the vertices $v\in S$.

\begin{question} Let $2\le k\le n\le m$.  What is the value of $\min_{G\in \mathcal{G}_{n,m}}\max_{S\subseteq [n]:|S|=k}|N(S)|$?
\end{question}

In particular, one may consider cases where $k$ is a constant, where $n$ is large and where $m\approx p\binom{n}{2}$, for some $p\in (0,1)$.  Two natural candidates for the extremal graph which ought to be considered are:
\begin{enumerate}
\item[(i)] the complete graph on $\approx p^{1/2}n$ vertices,
\item[(ii)] a random (or quasirandom) graph with density $p$.
\end{enumerate}

In case (i) it is trivial to optimise $|N(S)|$, as for any $S$ containing a vertex of the complete graph we have $|N(S)|\approx p^{1/2}n$.  While in case (ii) the size of $|N(S)|$ is roughly the same for all subsets $S$ with $|S|=k$, indeed $|N(S)|\approx (1-(1-p)^k)n$.  So we may ask in particular:

\begin{question} Let $k\ge 2$ and let $p\in (0,1)$.  In the limit as $n\to \infty$, is it true that 
\[
\min\left\{\max_{S\subseteq [n]:|S|=k}|N(S)|\, :\, G\in \mathcal{G}_{n,p\binom{n}{2}}\right\}\, =\, (1+o(1))\min\{p^{1/2}n, (1-(1-p)^k)n\}\, ?
\]
\end{question}

\bibliographystyle{plain}

\bibliography{thesparsebibliography}

%\bibliography{bibliofinal}

%In Section~\ref{sec:integral} we prove Proposition~\ref{prop:integral}.   In Section~\ref{sec:apdegs} we give the remaining proofs of Section~\ref{sec:degs}.  In Section~\ref{sec:LBvar} we prove Lemma~\ref{lem:LBvar}.

\appendix

\section{Proof of Proposition~\ref{prop:integral}}\label{Ap:A}

In this section, $\|x\|$ denotes the euclidean norm.

\begin{definition}
A function $F: \RR^d \rightarrow \RR$ is \emph{radial} if the value of $F(x)$ depends only on $\norm{x}$. If $F$ is radial, let $F_{rad}: \RR_0^{+} \rightarrow \RR$ be the function such that $F_{rad}(r) = F(x)$ for all $x$ with $\norm{x}=r$.  
\end{definition}

Given a subset $\LL \ssq \ZZ^d$ we say that $\LL$ has the \emph{non-zero} property if all coordinates $x_i$ of all $x\in \LL$ are non-zero.

%Furthermore, we say that $\LL$ is \emph{well-spaced} if $\norm{x-y}_{\inty}\ge 1$ for all $x,y\in \LL$.

%\begin{definition}
%Let $\LL \ssq \RR^d$. We say that $\LL$ has the non-zero property if $l(i) \neq 0$ for every $i \in [d]$ and every $l \in \LL$. Also, we say that $\LL$ has the 2-distance property if $\norm{l_1 - l_2}_{\infty} \geq 2$ whenever $l_1 \neq l_2$ and $l_i \in \LL$.
%\end{definition}

\begin{prop}
\label{prop:analy-tool}
Let $r \in \RR$ and let $F: \RR^d \rightarrow \RR^+$ be a continuous, integrable, radial function for which $F_{rad}$ is non-increasing.   Also, let $\LL$ be a subset of $\ZZ^d$ with the non-zero property. Then, 
\[
  \sum_{x \in \LL: \norm{x}\ge r} {F(x)} \leq \int_{A(r-d^{1/2})} F(u) du
\]
where $A(r) := \RR^d \setminus B(0,r)$.
\end{prop}

%Suppose there exists $\lambda \in (0,1)$ such that $r \geq \sqrt{d}/(1-\lambda)$.

\begin{proof}
We lose no generality in assuming $\norm{x}\ge r$ for all $x\in \LL$.  Now, for each $x\in \LL$, let $x_-$ be the point obtained by reducing the absolute value of each coordinate of $x$ by $1$ (i.e., reducing positive coordinates by $1$ and increasing negative coordinates by $1$).  Note that this definition depends on the non-zero property.  

Let $C_x$ be the open cube defined by $x$ and $x_-$, in the sense that $y\in C_x$ if each coordinate of $y$ lies between the corresponding coordinates in $x$ and $x_-$.  It is clear that distinct $x\in \LL$ have disjoint cubes $C_x$, and all cubes volume $1$.  And so, as $x$ is the point of largest norm in $C_x$ and $F_{rad}$ is non-increasing, we have that
\[ 
\sum_{x \in \LL} F(x)\, =\, \sum_{x \in \LL} \int F(u) \mathbbm{1}_{C_x}(u) du\, =\, \int  F(u) \mathbbm{1}_{\bigcup_{x\in \LL}C_x}(u) du\, .
\]
Since the diameter of a unit cube is $d^{1/2}$ and $\norm{x}\ge r$ for all $x\in \LL$, every point in the union has norm at least $r-d^{1/2}$.  The required result follows by monotonicity.
\end{proof}

What can we say if a family $\LL\ssq \ZZ^d$ does not have the non-zero property?  Well, in that case one can partition the family depending on the set $S\ssq [d]$ of non-zero coordinates and apply Proposition~\ref{prop:analy-tool}.  We obtain the following corollary.

\begin{cor}
\label{cor:analy-tool0}
Let $r \in \RR$ and let $F: \RR^d \rightarrow \RR^+$ be a continuous, integrable, radial function for which $F_{rad}$ is non-increasing.   Then 
\[
  \sum_{x \in \ZZ^d: \norm{x}\ge r} {F(x)} \, \leq\, 2^d\max_{S\ssq [d]} \int_{A_S(r-d^{1/2})} F_S(u) du
\]
where $A_S(r) := \RR^S \setminus B(0,r)$ and $F_S$ is the restriction of $F$ to $\RR^S$.
\end{cor}

We shall apply this result for the function $F(x)=\exp(-\beta \norm{x}^2)$, for $\beta\ge 1$.  Note that it is not difficult to bound $\int_{A(r)}\exp(-\beta \norm{u}^2) du$.  Indeed, we shall use that for $u\in \RR^d$ with $\norm{u}\ge r$ we have $\norm{u}^2\ge r^2/2\, +\, (u_1^2\, +\, \dots \, +\, u_d^2)/2$.  And so
\begin{align*}
\int_{A(r)}\exp(-\beta \norm{u}^2) \, &\le \,e^{-\beta r^2/2}\int_{A(r)}\exp(-\beta(u_1^2+\dots +u_d^2)/2)\\
&
\le \,e^{-\beta r^2/2}\int_{\RR^d}\exp(-(u_1^2+\dots +u_d^2)/2) du\\
& \le \,e^{-\beta r^2/2}\prod_{i=1}^{d}\int \exp(-u_i^2) du_i\\
& =\,  (2\pi)^{d/2} e^{-\beta r^2/2}\, .
\end{align*}
Note that in the last two lines we used Fubini and the well known identity $\int e^{-y^2/2} dy=\sqrt{2\pi}$.

It is clear that the same calculation in an $s$-dimensional subspace would give the upper bound $(2\pi)^{s/2}e^{-\beta r^2/2}$.  The required bound,
\[
  \sum_{x \in \ZZ^d: \norm{x}\ge r} \expb{-\beta \norm{x}^2} \, \leq\, (8\pi)^{d/2} e^{-\beta (r-d^{1/2})^2/2}
\]  
now follows from these estimates and Corollary~\ref{cor:analy-tool0}.  This completes the proof of Proposition~\ref{prop:integral}.

\section{The degree sequence -- proofs from Section~\ref{sec:degs}}\label{sec:apdegs}\label{Ap:B}

We shall prove Lemma~\ref{lem:largedegs} and Proposition~\ref{prop:sumsquare} stated in Section~\ref{sec:degs}.  We begin with Lemma~\ref{lem:largedegs}, which we restate here.

\largedegs*

\begin{proof} Fix $t\ge 2n^{-1} \log{n}$ and $b \geq 4tn$.  Note that $b\ge 8\log{n}$.
(i) Let $a=2b/\ell_b$.  There are $n$ vertices and for each vertex there is probability at most
\[
\binom{n}{a}t^a\, \le\, \exp\left(-a\log(a/etn)\right)\, \le\, \exp\left(\frac{-2b}{\ell_b}(\ell_b-\log{\ell_b})\right)\, \le\, \exp(-3b/2)
\]
that its degree is at least $a$. Since this is at most $\exp(-b)/2n$ for each vertex, item (i) fails with probability at most $\exp(-b)/2$. 

(ii) We shall prove, for each $j\ge 5$ that $\pr{|V_j|\ge b/tnj2^{j-6}}\, \le\, \exp(-2b)$.  The proof is then complete, by a union bound.  Fix $j\ge 4$.  By increasing $b$ if necessary (which corresponds to a slightly stronger bound) we may assume that $a:=b/tnj2^{j-6}$ is an integer.

If the event $|V_j|\ge a$ occurs then there exists a set $A$ of $a$ vertices whose degrees sum to at least $2^jatn$.  Let $G[A]$ denote the graph induced by a set $A$, and $G[A,B]$ the bipartite induced graph between sets $A$ and $B$.  It follows that either
\begin{enumerate}
\item[(i)] $e(G_m[A])\, \ge\, 2^{j-2}atn$, or
\item[(ii)] $e(G_m[A,V\setminus A])\, \ge\, 2^{j-1}atn$
\end{enumerate} 
In either case this random variable, $X$, has hypergeometric distribution and so we may use the bound~\eqr{h4} given in Lemma~\ref{lem:Chern}.  We note that in both cases $\mu\le atn$, and so setting $\nu=atn$ we have
\[
\pr{X\ge 2^{j-2}atn}\, \le\, \exp(-(j-2)2^{j-4}atn)\, \le\, \exp(-2b)\, ,
\]
as required.
\end{proof}

We now turn to Proposition~\ref{prop:sumsquare}, which we restate.  We recall the definitions
\[
\kappa(b,t)\, :=\, \begin{cases} tn^2\qquad & 1 \le b<t^{1/2}n\ell\\
 b^2/\ell^2 &t^{1/2}n\ell \le b<n\ell\\
bn/\ell &n\ell \le b\le tn^2\ell
\end{cases}
\]
and
\[
\kappa^{+}(b,t)\, :=\, \begin{cases} b^2/\ell_b^2 \qquad & 1\le b< n\ell\\
bn/\ell &n\ell \le b\le tn^2\ell
\end{cases}
\]

\sumsquare*

%There exists an absolute constant $C>0$ such that the following holds.  Suppose that $t\ge 2n^{-1} \log{n}$ and that $b\ge 4tn$.  Except with probability at most $\exp(-b)$ we have
%\[
%\sum_{u}D_{u}(G_i)^2\, \le \, C\kappa(b,t)
%\]
%and
%\[
%\sum_{u\, :\, |d_u(G_m)| \ge 32tn}D_{u}(G_i)^2\, \le \, C\kappa^{+}(b,t)
%\]
%for all steps $i\le m$.

\begin{proof} Fix $t\ge 2n^{-1} \log{n}$, and $i\le m$.  We prove that the bounds fail with probability at most $\exp(-2b)$.  The proposition then follows by a union bound.  We begin with the second statement, about the restricted sum $\sum_{u\, :\, d_u(G_m) \ge 32tn}d_{u}(G_m)^2$.  We shall show that the result follows from the bounds on $V_j$ stated in Lemma~\ref{lem:largedegs}.  Let us define $J=J(b)$ to be maximal such that $2^jtn \leq 2b/\ell_b$, if $b\le n\ell$, and $J=\log_{2}(1/t)$, if $b> n\ell$.  By Lemma~\ref{lem:largedegs}, except with probability at most $\exp(-2b)$, we have 
\begin{enumerate}
\item[(i)] $|V_j|\le b/tnj2^{j-7}$ for all $j\ge 5$, and 
\item[(ii)] $V_j=\emptyset$ for all $j> J$.
\end{enumerate}

For the range $32tn\le b\le n\ell$:  Assuming (i) and (ii), we have
\begin{align*}
    \sum_{u:d_u(G_m) \geq 32 tn} d_u^2(G_m) &\leq \sum_{j=5}^J 2^{2j+2}t^2 n^2 |V_j| \\ &\leq 2^9 \sum_{j=5}^J \dfrac{2^jbtn}{j} \\ &\leq \dfrac{2^{J+11}btn}{J} \\ &\leq\, \dfrac{2^{12}b^2}{\ell_b^2}\, .
\end{align*}
Where we used for the last inequality that $2^j \geq b/tn\ell_b$, and so, in particular,
\[
J\, \geq\, \log \left( \dfrac{b}{etn \ell_b} \right) \, =\, \ell_b - \log(\ell_b)\, \geq\, \dfrac{\ell_b}{2}.
\]
This completes the proof of the bound on $\sum_{u\, :\, d_u(G_m) \ge 32tn}d_{u}(G_m)^2$ in this range provided $C$ is chosen to be at least $2^{12}$.

For the range $b\ge n\ell$: In this case $J=\log_2(1/t)$, and so $\lfloor J\rfloor\ge \ell$ and $2^{\lfloor J\rfloor} \le t^{-1}$.  Assuming (i) and (ii), we may argue as above and obtain
\begin{align*}
\sum_{u:d_{u}(G_m)\ge 32tn}d_u(G_m)^2 \, &\le \, \sum_{j=5}^{\lfloor J\rfloor} 2^{2j+2}t^2 n^2 |V_j| \\ 
&\le\, \dfrac{2^{\lfloor J\rfloor +11}btn}{\lfloor J\rfloor}\\
&\le\,  \frac{2^{11} bn}{\ell}\, .
\end{align*}
This proves the required bound in this range provided $C\ge 2^{11}$.

We now prove the first bound on the unrestricted sum $\sum_{u}D_{u}(G_i)^2$.  Note that $\kappa(b,t)$ does not depend on $b$ in the first range, so the result in this range follows from the beginning of the next range (as increasing $b$ corresponds to a stronger result).  Also, we may continue to assume the bounds proved in the first part of the proof.  In particular, if $d_u(G_i) \geq 32 tn$ then $d_u(G_m) \geq 32 tn$ and we have $D_u(G_i)^2\le d_u(G_m)^2$.

%(these also apply for $D_u(G_i)$ as $D_u(G_i)\le d_u(G_i)\le d_u(G_m)$).  

In the range $t^{1/2}n\ell \le b<n\ell$ we have from the above bound that 
\[
\sum_{u:d_u(G_i) \geq 32 tn} D_u(G_i)^2\, \le\, \dfrac{2^{12}b^2}{\ell_b^2}
\]
and by Lemma~\ref{lem:sd2}, there is a constant $C_1$ such that except with probability $\exp(-2b)$ we have $\sum_{u:d_u(G_i) \leq 32tn} D_u(G_i)^2$ is at most $Ctn^2$ if $b\le n$, and at most $Cbtn$ if $n\le b\le n\ell$.  As $\max\{tn^2,btn,b^2/\ell^2\} =b^2/\ell^2=\kappa(b,t)$, we have 
\[
\sum_{u} D_u(G_i)^2\, \le\, (2^{12}+C_1)\kappa(b,t)\, ,
\]
which gives the result in this range.

In the range $n\ell\le b<tn^2\ell$ we similarly bound the two contributions by $2^{11} bn/\ell$ and $C_1btn$ respectively.  This is at most $C\kappa(b,t)$, provided we choose $C\ge 2^{11}+C_1$.
\end{proof}

%\section{Proof of Lemma~\ref{lem:multi}}
%
%We first restate the lemma.
%
%\multi*
%
%\begin{proof}Clearly the result is trivial if some vertex has degree at least $n/r$ so we may assume all degrees are less than $n/r$.
%
%Consider the digraph obtained from $G$ by replacing each edge by two oriented edges (one in each direction).  It clearly suffices to find $v_1,\dots, v_r$ such that the union of the out-neighbourhoods in $D$ of these vertices has cardinality at least $n/r$.  We note that all in and out degrees in $D$ are at most $n/r$ and that $e(D)=2e(G)\ge 2n^2/r^2$. 
%
%We may find $v_1,\dots ,v_r$ greedily.   Let $v_1$ be a vertex of maximum out-degree and let us write $d_1$ for this degree and $S_1$ for $N^{+}(v_1)$.  We now remove from the digraph all edges into $S_1$.  In the remaining digraph we find a vertex of maximum out-degree $v_2$ with out-degree $d_2$, set $S_2=S_1\cup N^{+}(v_2)$, and remove any other edges into $S_2$.  We continue.
%
%At step $i$, we may assume the current set $S_i$ has cardinality at most $n/r$ (else we are already done) and so the total number of removed edges so far is at most $n^2/r^2$.  It follows that at least $n^2/r^2$ edges remain and so $d_i\ge n/r^2$.
%
%It is clear this process terminates in at most $r$ steps, as $|S_i|=d_1+\dots +d_i\ge in/r^2$.
%\end{proof}

\section{Variance lower bound}\label{sec:LBvar}

We shall prove Lemma~\ref{lem:LBvar} which we restate for convenience.  We also recall the definition
\[
\X''_i\, :=\, 3\frac{(N-m)_{2}(m-i)}{(N-i)_3}\, X''_{\owedge}(G_i)\, +\, \frac{(N-m)_3}{(N-i)_3}\, X''_{\triangle}(G_i)\, ,
\]

\LBvar*

We remark that if $t$ is a constant, i.e., in the case of dense random graphs, then this follows easily from the variance bounds given in Section 7 of~\cite{GGS}.  Since the coefficient of $X''_{\owedge}(G_i)$ is $O(t)$ and the coefficient of $X''_{\triangle}(G_i)$ is $\Omega(1)$, it is clearly suffices to prove that, except with probability $\exp(-\Omega(n))$ we have
\eq{owlb}
\Ex{ X''_{\owedge}(G_i)^2\, \mid \, G_{i-1}}\, =\, O\left(\frac{n}{\ell^2}\right)
\eqe
and
\eq{trilb}
\Ex{ X''_{\triangle}(G_i)^2\, \mid \, G_{i-1}}\, =\, \Omega(t^2n)\, .
\eqe
Note that~\eqr{owlb} follows from Lemma~\ref{lem:varsch} (as the truncation can only reduce variance), so we need only prove that~\eqr{trilb} holds, except with probability $\exp(-\Omega(n))$.

The variance in the next value of $X_{\triangle}(G_i)$ comes from the different codegrees present in the graph (the variance would be zero if all codegrees $d_{uw}(G_{i-1})$ of non-edges $uw$ were equal).  Recall the notation $D_{uw}(G_m)$ for the codegree deviation of the pair $uw$ relative to its expected value in $G_m\sim G(n,m)$.  By~\eqr{Xtricodeg} we have
\eq{xtg}
X_{\triangle}(G_i)\,  =\, 6 D_{e_i}(G_{i-1})\, -\, \Ex{6 D_{e_i}(G_{i-1})\, \mid\, G_{i-1}}\, 
\eqe
where $e_i$ denotes the $i$th edge of the process.  In order to turn this into a comparison of codegrees we recall the elementary fact that for $X,Y$ iid random variables with mean $0$ we have $\Ex{(X-Y)^2}=2\Ex{X^2}$.  Similarly, if $X,Y$ are iid random variables with $\Ex{X|G_{i-1}}=0$, then $\Ex{(X-Y)^2|G_{i-1}}=2\Ex{X^2|G_{i-1}}$.  In the case of $X_{\triangle}(G_i)$ the randomness comes from the selection of the edge $e_i\in E(K_n)\setminus E(G_{i-1})$.
It follows that
\begin{align*}
\Ex{ X_{\triangle}(G_i)^2\, \mid \, G_{i-1}}\, & =\, \frac{1}{2(N-i+1)^2}\sum_{e',e^*\not\in E(G_{i-1})} \big(X_{\triangle}(G_{i-1}\cup \{e'\})-X_{\triangle}(G_{i-1}\cup\{e^*\})\big)^2\\
& =\, \frac{18}{(N-i+1)^2}\sum_{e',e^*\not\in E(G_{i-1})} \big(D_{e'}(G_{i-1})\, -\,  D_{e^{*}}(G_{i-1}\big)^2\\
& =\, \Omega(n^{-4})\sum_{e',e^*\not\in E(G_{i-1})} \big(D_{e'}(G_{i-1})\, -\,  D_{e^{*}}(G_{i-1}\big)^2\, .
\end{align*}
This does not yet consider the truncation, however, the truncation only affects deviations larger than $3t^2n$, and so, by the same argument,
\[
\Ex{ X''_{\triangle}(G_i)^2\, \mid \, G_{i-1}}\, = \, \Omega(n^{-4})\sum_{e',e^*\in E_0(G_{i-1})} \big(D_{e'}(G_{i-1})\, -\,  D_{e^{*}}(G_{i-1}\big)^2
\]
where $E_0(G)=\{uw\not\in E(G_{i-1}):D_{uw}(G_{i-1})\le 3t^2n\}$.

It is therefore sufficient to prove that, except with probability $\exp(-\Omega(n))$, there are at least $\Omega(n^2)$ pairs $uw\in E(K_n)\setminus E(G_{i-1})$ with 
\eq{small}
d_{uw}(G_{i-1})\,\le\, sd_{med}(G)\, -\, sn^{1/2}
\eqe
and $\Omega(n^2)$ pairs $uw\in E(K_n)\setminus E(G_{i-1})$ with 
\eq{large}
sd_{med}(G)\, +\, sn^{1/2}\, \le\, d_{uw}(G_{i-1})\,\le\, 1.1s^2n\, +\, 2sn^{1/2}\, 
\eqe
where $d_{med}(G)$ is the median degree.  Note this gives $\Omega(n^4)$ pairs of pairs $uw,u'w' \in E_0(G_{i-1})$ with $|D_{e'}(G_{i-1})\, -\,  D_{e^{*}}(G_{i-1})| \ge 2sn^{1/2}$, which is clearly sufficient (since $t/2 \le s \le t$).

We will now give a proof that~\eqr{large} holds for $\Omega(n^2)$ pairs $uw\in E(K_n)\setminus E(G_{i-1})$, except with probability $\exp(-\Omega(n))$. It will be clear that, up to minor variations, the same proof also gives~\eqr{small}.  

We begin with a lemma about the codegrees $d_{uw}(G)$ from the perspective of a single vertex $u$.  We work with $G\sim G(n,p)$ and the observe that it transfers to our setting.  For a vertex $u\in V(G)$ let $F_u$ be the event that
\begin{enumerate}
\item[(i)] $d(u)\in [0.9pn,1.1pn]$
\item[(ii)] $|\{w\in V(G)\setminus N(u)\, :\, d_{uw}(G)\in [pd(u)+pn^{1/2},pd(u)+2pn^{1/2}]\}|\, <\, n/100$
\end{enumerate}

\begin{lem} There exists a constant $c_1>0$ such that the following holds.  Let $n^{-1/2}/c_1 \le t\le 1/2$, let $p\in (t/2,t)$ and let $G\sim G(n,p)$.  Then, for each vertex $u\in V(G)$ we have
\[
\pr{F_u}\, \le\, \exp(-c_1 n)\, .
\]
\end{lem}

\begin{proof} If the event $F_u$ occurs then we must have that $d(u)\in [0.9pn,1.1pn]$.  And so if we consider conditioning on $N(u)=S$, for some set of vertices $S$, then we have
\begin{align*}
\pr{F_u}\, =\, &\sum_{0.9pn\le |S|\le 1.1pn}\pr{N(u)=S} \cdot\phantom{\bigg|} \\ &\pr{|\{w\in V(G)\setminus S\, :\, |N(w)\cap S|\in [p|S|+pn^{1/2},p|S|+2pn^{1/2}]\}|\, <\, n/100}\phantom{\bigg|}\, ,
\end{align*}
where we have already used that edges in $G(n,p)$ are independent, so that the edges between $w$ and $S$ are independent of the fact that $S$ is the neighbourhood of $u$.  Since the sum of the probabilities $\pr{N(u)=S}$ is at most $1$, and all $S$ of the same cardinality are equivalent, it suffices to show, for all integers $r$ in the range $0.9pn\le r\le 1.1pn$ that the event
\[
\Big|\{w\in V(G)\setminus \{1,\dots r\}: \big|N(w)\cap \{1,\dots ,r\}\big|\in [pr+pn^{1/2},pr+2pn^{1/2}]\}\Big|\, <\, n/100\\ 
\]
has probability at most $\exp(-\Omega(n))$.
For each fixed vertex $w\in V(G)\setminus \{1,\dots r\}$ the random variable $B_{w,r}:=|N(w)\cap \{1,\dots ,r\}|$ has distribution $Bin(r,p)$.  The central limit theorem states that $(B_{w,r}-rp)/(rp(1-p))^{1/2}$ converges to a standard normal as $r\to \infty$ (and so in particular as $n\to \infty$).  And, setting $\sigma=(rp(1-p))^{1/2}$ we see that $pn^{1/2}=\lambda \sigma$ where $\lambda=(pn/r(1-p))^{1/2}$ and one may check that $\lambda\in [0.6,1.1]$.  It follows that the probability that $B_{w,r}\in [pr+pn^{1/2},pr+2pn^{1/2}]$ converges to $\Phi(2\lambda)-\Phi(\lambda)$ for some $\lambda\in [0.6,1.1]$, and this value is always at least $1/10$.

As this event is independent as we vary over the values of $w\in  V(G)\setminus \{1,\dots r\}$ we see that $|\{w\in V(G)\setminus \{1,\dots r\}\, :\, |N(w)\cap \{1,\dots ,r\}|\in [pr+pn^{1/2},pr+2pn^{1/2}]\}|$ stochastically dominates a random variable with distribution $Bin(n-r,1/10)$, and so stochastically dominates $Y\sim Bin(n/3,1/10)$.  The required bound now follows, as the Chernoff's inequality gives us that $\pr{Y< n/100}\le\exp(-\Omega(n))$.
\end{proof}

The lemma will clearly be very useful in proving~\eqr{large} holds for $\Omega(n^2)$ pairs $uw\in E(K_n)\setminus E(G_{i-1})$.  We note immediately that up to a polinomial factor, which is easily absorbed in the $\exp(-\Omega(n))$ term, the same results apply for $G_{i-1}$, and we may also apply a union bound over vertices $u$.  So that, except with probability at most $\exp(-\Omega(n))$, we have 
\eq{somegoodw}
|\{w\in V(G_{i-1})\setminus N(u)\, :\, d_{uw}(G)\in [sd(u)+sn^{1/2},sd(u)+2sn^{1/2}]\}|\, <\, n/100
\eqe
for all vertices $u\in G_{i-1}$ with degree in the interval $[0.9sn,1.1sn]$.

By Proposition~\ref{prop:sumsquare} there is probability at most $\exp(-\Omega(s^2n^2\ell))\le \exp(-\Omega(n))$ that linearly many vertices of $G_i$ have degree outside the range $[0.9sn,1.1sn]$.  We may therefore assume that there are at least $n/10$ vertices $u$ with $0.9sn\le d_{med}(G)\le d_u(G_{i-1})\le 1.1sn$ and for such vertices there are at least $n/100$ vertices $w$, given by~\eqr{somegoodw}, for which 
\[
sd_{med}(G)\, +\, sn^{1/2}\, \le\, d_{uw}(G_{i-1})\,\le\, 1.1s^2n\, +\, 2sn^{1/2}\, .
\]
It follows immediately that~\eqr{large} holds for $\Omega(n^2)$ pairs $uw\in E(K_n)\setminus E(G_{i-1})$.

As remarked above, a very similar argument shows that~\eqr{small} holds for $\Omega(n^2)$ pairs $uw\in E(K_n)\setminus E(G_{i-1})$.  And this completes the proof of Lemma~\ref{lem:LBvar}.

\let\thefootnote\relax\footnotetext{This study was financed in part by the Coordenação de Aperfeiçoamento de Pessoal de Nível Superior – Brasil (CAPES) – Finance Code 001}

\end{document}